\documentclass{amsart}

\usepackage{amsfonts}
\usepackage{amsmath}
\usepackage{amssymb}
\usepackage{amsthm}

\usepackage{multicol}
\usepackage[margin=3cm]{geometry} 
\usepackage{mathabx} 

\usepackage{arydshln} 
\usepackage{makecell} 

\usepackage[pagebackref, pdftex]{hyperref} 
\usepackage{cite}

\renewcommand*{\backref}[1]{}
\renewcommand*{\backrefalt}[4]{
  \ifcase #1
  [No citations.]
  \or [#2]
  \else [#2]
  \fi }

\usepackage{graphicx} 
\usepackage{import}

\usepackage{tikz}
\usetikzlibrary{calc, arrows, decorations.markings, decorations.pathmorphing, positioning, decorations.pathreplacing}

\AtBeginDocument{%
   \def\MR#1{}
}

\newcommand{\To}{\longrightarrow}

\newcommand{\B}{\mathcal{B}}
\newcommand{\C}{\mathbb{C}}

\newcommand{\Disc}{\mathbb{D}}

\newcommand{\E}{\mathbb{E}}
\newcommand{\F}{\mathcal{F}}

\newcommand{\h}{\mathfrak{h}}
\newcommand{\HH}{\mathbb{H}}
\newcommand{\hyp}{\mathcal{H}}

\newcommand{\II}{\mathbb{I}}

\newcommand{\MF}{\mathcal{MF}}

\renewcommand{\P}{\mathcal{P}}
\newcommand{\pH}{\mathbf{\$ H}}

\newcommand{\R}{\mathbb{R}}

\renewcommand{\S}{\mathcal{S}}

\newcommand{\U}{\mathbb{U}}

\newcommand{\V}{\mathcal{V}}

\newcommand{\Z}{\mathbb{Z}}

\DeclareMathOperator{\Cl}{Cl}

\DeclareMathOperator{\End}{End}

\DeclareMathOperator{\Fr}{Fr}

\DeclareMathOperator{\Hor}{Hor}

\let\Im\relax
\DeclareMathOperator{\Im}{Im}
\let\Re\relax
\DeclareMathOperator{\Re}{Re}

\DeclareMathOperator{\Isom}{Isom}

\DeclareMathOperator{\Pin}{Pin}
\DeclareMathOperator{\pdet}{pdet}

\DeclareMathOperator{\Spin}{Spin}

\DeclareMathOperator{\Tr}{Tr}

\numberwithin{equation}{subsection}

\newtheorem{theorem}{Theorem}

\newtheorem{lem}[equation]{Lemma}

\newtheorem{prop}[equation]{Proposition}

\theoremstyle{definition}
\newtheorem{defn}[equation]{Definition}
\newtheorem{example}[equation]{Example}

\newcommand{\refsec}[1]{Section~\ref{Sec:#1}}
\newcommand{\refdef}[1]{Definition~\ref{Def:#1}}
\newcommand{\reffig}[1]{Figure~\ref{Fig:#1}}

\newcommand{\refeqn}[1]{\eqref{Eqn:#1}}
\newcommand{\reflem}[1]{Lemma~\ref{Lem:#1}}
\newcommand{\refprop}[1]{Proposition~\ref{Prop:#1}}
\newcommand{\refthm}[1]{Theorem~\ref{Thm:#1}}

\newcommand{\refeg}[1]{Example~\ref{Eg:#1}}

\begin{document}

\title{Quaternionic spinors and horospheres in 4-dimensional hyperbolic geometry} 

\author{Daniel V. Mathews}
\address{School of Mathematics, Monash University, 9 Rainforest Walk, Clayton VIC 3800, Australia; School of Physical and Mathematical Sciences, Nanyang Technological University, 21 Nanyang Link, Singapore 637371}
\email{dan.v.mathews@gmail.com}

\author{Varsha}
\address{Department of Mathematics, 
University College London, Gower Street, London WC1E  6BT}
\email{varsha.varsha.24@ucl.ac.uk}


\begin{abstract}
We give explicit bijective correspondences between three families of objects: certain pairs of quaternions, which we regard as spinors; certain flags in (1+4)-dimensional Minkowski space; and horospheres in 4-dimensional hyperbolic space decorated with certain pairs of spinorial directions. These correspondences generalise previous work of the first author, Penrose--Rindler, and Penner in lower dimensions, and use the description of 4-dimensional hyperbolic isometries via Clifford matrices studied by Ahlfors and others.

We show that lambda lengths generalise to 4 dimensions, where they take quaternionic values, and are given by a certain bilinear form on quaternionic spinors. They satisfy a non-commutative Ptolemy equation, arising from quasi-Pl\"{u}cker relations in the Gel'fand--Retakh theory of noncommutative determinants. We also study various structures of geometric and topological interest  that arise in the process.
\end{abstract}

\maketitle

\tableofcontents

\section{Introduction}
\label{Sec:introduction}

\subsection{Overview and main results}

In previous work \cite{Mathews_Spinors_horospheres}, the first author demonstrated a bijective correspondence between complex \emph{spinors} or \emph{spin vectors}, and \emph{spin-decorated horospheres} in 3-dimensional hyperbolic space. This extended the work of Penrose--Rindler \cite{Penrose_Rindler84}, which associated to such spin vectors certain \emph{flags} in Minkowski space $\R^{1,3}$. This previous work also gave several further connections between spinors and 3-dimensional hyperbolic geometry, and has been used to study knot complements \cite{HIMS_lambda_figure8} and generalise Descartes' classic circle theorem \cite{Mathews_Zymaris}.

In this paper, we generalise these results from the complex numbers $\C$ to the \emph{quaternions} $\HH$,
 from 3- to \emph{4-dimensional} hyperbolic space $\hyp^4$, and from $(1+3)$- to \emph{$(1+4)$-dimensional} Minkowski space $\R^{1,4}$. It is well known that the progression from 2 to 3 to 4 dimensions in hyperbolic geometry is closely related to the progression from $\R$ to $\C$ to $\HH$. In particular, the orientation preserving isometry groups of hyperbolic space of dimensions 2, 3, and 4 are respectively isomorphic to certain groups of M\"{o}bius transformations over $\R$, $\C$ and $\HH$ respectively, although the quaternionic case is more subtle than the others. The results of this paper show that this progression also extends to spinors, flags, and horospheres, with corresponding subtleties.

The first main result of this paper is a generalisation of the spinor-horosphere correspondence of \cite{Mathews_Spinors_horospheres}, also incorporating the corresponding notions of flags, which we call \emph{multiflags}.
\begin{theorem}[Spinor--multiflag--horosphere correspondence]
\label{Thm:main_thm_1} 
There is an explicit, smooth, bijective, $SL_2 \$$-equivariant correspondence between the following:
\begin{enumerate}
\item
quaternionic spinors; 
\item
spin multiflags; 
\item
spin-decorated horospheres in $\hyp^4$.
\end{enumerate}
\end{theorem}

Roughly, quaternionic spinors are pairs of quaternions $(\xi, \eta)$ satisfying certain conditions; multiflags consist of a pair of orthogonal 2-dimensional flags on a lightlike flagpole in $\R^{1,4}$; decorations on horospheres consist of orthogonal pairs of parallel direction fields; and $SL_2 \$ $ is a group of quaternionic matrices describing isometries of $\hyp^4$ analogously to $SL_2 \C$ for $\hyp^3$.
We properly explain these and all notions involved as we proceed.

In \cite{Mathews_Spinors_horospheres} the first author showed that the notion of (real) \emph{lambda length} introduced by Penner in the 2-dimensional context \cite{Penner87} as a distance between horospheres, extends to a \emph{complex} lambda length between spin-decorated horospheres in $\hyp^3$, and moreover is given by a natural antisymmetric bilinear form on spinors, taking a determinant. 

The second main result of this paper extends these results to quaternions and $\hyp^4$. We show there is a well-defined notion of \emph{quaternionic} lambda length between spin-decorated horospheres in $\hyp^4$, which uses the isomorphism of unit quaternions with $\Spin(3)$. This lambda length agrees a with a bilinear form on spinors, denoted $\{ \cdot, \cdot \}$, given by a certain \emph{pseudo-determinant} considered by Ahlfors in a series of papers \cite{Ahlfors_Clifford85, Ahlfors_Mobius85, Ahlfors_Mobius_86, Ahlfors_fixedpoints_85, Ahlfors_84}, which we denote $\pdet$, as follows.
\begin{theorem}[Lambda lengths are pseudo-determinants]
\label{Thm:main_thm_2}
Let $\kappa_1 = (\xi_1, \eta_1)$ and $\kappa_2 = (\xi_2, \eta_2)$ be two quaternionic spinors, corresponding to spin-decorated horospheres $(\h_1, W_1)$ and $(\h_2, W_2)$ in $\hyp^4$. Then the lambda length $\lambda_{12}$ from $\h_1$ to $\h_2$ is given by
\begin{equation}
\label{Eqn:lambda_pdet}
\lambda_{12} = \pdet 
\begin{pmatrix} \xi_1 & \xi_2 \\ \eta_1 & \eta_2 \end{pmatrix}
= \{ \kappa_1, \kappa_2 \} = \xi_1^* \eta_2 - \eta_1^* \xi_2
\end{equation}
\end{theorem}
Here $*$ denotes Clifford algebra reversion conjugation, which applies to quaternions as $(a+bi+cj+dk)^* = a+bi+cj-dk$. The lambda length is no longer antisymmetric but instead satisfies $\lambda_{12} = -\lambda_{21}^*$.

In \cite{Mathews_Spinors_horospheres}, the first author showed that given four complex spinors $\kappa_n$ corresponding to four spin-decorated horospheres $\h_n$, the six lambda lengths between them satisfy a \emph{Ptolemy equation}
\[
\lambda_{01} \lambda_{23} + \lambda_{03} \lambda_{12} = \lambda_{02} \lambda_{13},
\]
where $\lambda_{mn}$ is the lambda length from $\h_m$ to $\h_n$.
The third main result of this paper generalises this result to the non-commutative context of quaternions, and 4 dimensions.
\begin{theorem}[Non-commutative Ptolemy equation]
\label{Thm:main_thm_3}
Given four spin-decorated horospheres $(\h_n, W_n)$ in $\hyp^4$, $n=0,1,2,3$, let $\lambda_{mn}$ denote the lambda length from $(\h_m, W_m)$ to $(\h_n, W_n)$. Then
\begin{equation}
\label{Eqn:noncomm_Plucker}
\lambda_{02}^{-1} \lambda_{01} \lambda_{31}^{-1} \lambda_{32} + \lambda_{02}^{-1} \lambda_{03} \lambda_{13}^{-1} \lambda_{12} = 1.
\end{equation}
\end{theorem}
Similar non-commutative Ptolemy equations were discussed by Berenstein and Retakh \cite{Berenstein_Retakh_18, Retakh_OW_report_13} in the context of non-commutative surfaces and cluster algebras.
In the rest of this introduction we summarise the results of this paper, and give further details, background, and context.

\subsection{Quaternionic spinors and paravectors}
\label{Sec:intro_spinors_paravectors}

The spinors in \refthm{main_thm_1} are a quaternionic generalisation of the \emph{spin vectors} of Penrose and Rindler in \cite{Penrose_Rindler84}, which consisted of pairs of complex numbers. Here however, defining quaternionic spinors naively as pairs of quaternions cannot provide a bijection with spin-decorated horospheres: as we see below in \refsec{horospheres_decorations}, the space of spin-decorated horospheres is $7$-real-dimensional, so any bijection with quaternionic spinors must cut down the space of spinors from $\HH^2$ to a real-codimension-1 subset. 

The definition of quaternionic spinors can be motivated as follows. In the complex case, a spinor $(\xi, \eta) \in \C^2$ yields a point on the future light cone in $\R^{1,3}$ via the $2 \times 2$ matrix
\begin{equation}
\label{Eqn:spinor_matrix}
\begin{pmatrix} \xi \\ \eta \end{pmatrix}
\begin{pmatrix} \overline{\xi} & \overline{\eta} \end{pmatrix}
=
\begin{pmatrix} |\xi|^2 & \xi \overline{\eta} \\ \eta \overline{\xi} & |\eta|^2 \end{pmatrix}
\end{equation}
The diagonal elements of this matrix are non-negative, and the Pauli matrices can be used to identify a point on the future light cone in $(1+3)$-dimensional Minkowski space, by equating the above matrix with
\begin{equation}
\label{Eqn:R13_matrix}
\frac{1}{2} \begin{pmatrix} T+Z & X+iY \\ X-iY & T-Z \end{pmatrix},
\quad \text{corresponding to} \quad
(T,X,Y,Z) \in \R^{1,3}.
\end{equation}

If we now consider $\xi, \eta$ above as quaternions, then equation \refeqn{spinor_matrix} still makes sense, using the standard quaternion conjugation. 
The diagonal elements of \refeqn{spinor_matrix} are again non-negative reals, but the off-diagonal entries in general can be arbitrary quaternions. We can identify points in $(1+4)$-dimensional Minkowski space with matrices by a generalisation of \refeqn{R13_matrix}, namely
\begin{equation}
\label{Eqn:R14_matrix}
\frac{1}{2}
\begin{pmatrix}
T+Z & W+iX+jY \\
W-iX-jY & T-Z
\end{pmatrix}
\quad \text{corresponding to} \quad
(T,W,X,Y,Z) \in \R^{1,4}.
\end{equation}
Requiring quaternionic spinors $(\xi, \eta)$ to yield points in $\R^{1,4}$ then motivates the following definitions.
\begin{defn} \
\label{Def:quaternionic_spinor}
\begin{enumerate} 
\item
The real-linear subspace of $\HH$ spanned by $1$, $i$ and $j$ is denoted $\$\R^3$. Its elements are called \emph{paravectors}.
\item
A \emph{quaternionic spinor}, or just \emph{spinor}, is a pair of quaternions $\kappa = (\xi, \eta)$, not both zero, such that $\xi \overline{\eta} \in \$\R^3$. The set of quaternionic spinors is denoted $S\HH$.
\end{enumerate}
\end{defn}
Thus
\[
S\HH = \left\{ \left( \xi, \eta \right) \in \HH^2 \; \mid \; 
(\xi, \eta) \neq (0,0), \quad 
\xi \overline{\eta} \in \$\R^3 \right\}.
\]

The notation $\$\R^3$ and terminology follows Lounesto \cite[ch. 19]{Lounesto_Clifford_book_01} and comes from the context of Clifford algebras. Given a real vector space $V$ with a nondegenerate quadratic form $Q$, the Clifford algebra $\Cl(V,Q)$ contains $V$ as its subspace of degree-1 elements or \emph{vectors}, and $\R$ as its subspace of degree-0 elements of \emph{scalars}. \emph{Paravectors} are sums of elements of degree 0 or 1, so the space of paravectors is $\R \oplus V$. Paravectors arise naturally in the study of M\"{o}bius transformations and Clifford algebras, e.g. \cite{Ahlfors_84, Ahlfors_Clifford85, Ahlfors_Mobius85, Ahlfors_fixedpoints_85, Ahlfors_Mobius_86, Cao_Waterman_98, Lounesto_Latvamaa_80, Maass_49, Vahlen_1902, Waterman_93, Kellerhals01, Gongopadhyay_12, Ahlfors_Lounesto_89}. Often they are known as \emph{vectors}; we discuss terminology in \refsec{vectors_paravectors}.
When $V = \R^2$ and $Q$ is negative definite, we have $\Cl(V,Q) \cong \HH$, with the vectors being $\R i + \R j$, the scalars being $\R$, and the paravectors as in \refdef{quaternionic_spinor}. Paravectors are then precisely those quaternions $x$ such that $x^* = x$.

The defining condition $\xi \overline{\eta} \in \$\R^3$ of quaternionic spinors cuts out a 7-dimensional subset of $\HH^2$, as required for a bijection with spin-decorated horospheres. 
Indeed, $S\HH \cong S^3 \times S^3 \times \R$ (\reflem{topology_of_SH}). 
The condition $\xi \overline{\eta} \in \$\R^3$ is known (e.g. \cite{Ahlfors_Clifford85, Ahlfors_Mobius85, Ahlfors_Mobius_86, Ahlfors_fixedpoints_85, Ahlfors_84}) to have multiple equivalent reformulations: it is equivalent to $\xi^* \eta \in \$\R^3$ and, when $\eta \neq 0$, is equivalent to $\xi \eta^{-1} \in \$\R^3$. 

The matrices of \refeqn{R14_matrix} are precisely those $2 \times 2$ matrices $S$ with paravector entries and which satisfy $S = \bar{S}^T$, i.e. \emph{paravector Hermitian} matrices.
Just as Hermitian matrices with complex entries correspond to points of $\R^{1,3}$, paravector Hermitian matrices with quaternion entries correspond to points of $\R^{1,4}$.
For a paravector Hermitian matrix $S$ as in \refeqn{R14_matrix} corresponding to a point $p = (T,W,X,Y,Z) \in \R^{1,4}$, as in the complex case we have
\[
4 \det S = \langle p, p \rangle, \quad
\Tr S = T.
\]
Here $\langle \cdot, \cdot \rangle$ denotes the Minkowski inner product $dT^2 - dW^2 - dX^2 - dY^2 - dZ^2$, and $\det$ is the usual $2 \times 2$ determinant (which agrees with $\pdet$ for paravector Hermitian matrices). 

With \refdef{quaternionic_spinor} in hand, we can associate to a quaternionic spinor $\kappa = (\xi, \eta) \in S\HH$ a point $(T,W,X,Y,Z) \in \R^{1,4}$ by equating the matrices \refeqn{spinor_matrix} and \refeqn{R14_matrix}.
Since a matrix \refeqn{spinor_matrix} arising from a spinor has determinant $0$ and positive trace, all points in $\R^{1,4}$ arising from spinors $\kappa \in S\HH$ in fact lie on the  future light cone.

\subsection{Multiflags}

In \cite{Penrose_Rindler84}, Penrose and Rindler associated to a complex spinor $(\xi, \eta)$ certain \emph{flags}. Such a flag involves a point $p$ on the light cone and a 2-plane tangent to the light cone, containing the line $p \R$.
A family of complex spinors of the form $(\xi, \eta) e^{i\theta}$ all yield the same point $p$, but multiplication by $e^{i\theta}$ rotates the 2-plane by $2\theta$ about $p \R$. As shown in \cite{Mathews_Spinors_horospheres}, the flag can be given by the derivative of the map which sends $(\xi, \eta) \mapsto (T,X,Y,Z)$ in the direction of a related spinor $(\overline{\eta}, -\overline{\xi})i$.

In the quaternionic case, again we obtain a point $p$ on the future light cone, but \emph{two} such flags naturally arise, which we call the \emph{$i$-flag} and \emph{$j$-flag}. As the names suggest, these flags arise from the quaternions $i$ and $j$, and indeed they are given by the derivative of the map $(\xi, \eta) \mapsto (T,W,X,Y,Z)$ in certain directions $(\eta', -\xi')i$ and $(\eta', -\xi')j$.
(Here $'$ denotes another conjugation on $\HH$, $x' = \bar{x}^*$, which arises naturally from Clifford algebras. A discussion of the conjugations on Clifford algebras is given below in \refsec{Clifford_general}, and for $\HH$ specifically in \refsec{quaternion_involution}.) 
See \reffig{1}. This yields a \emph{multiflag}, formally defined in \refdef{multiflag}. 
When $(\xi, \eta)$ are both complex, the $i$-flag agrees with the flag of \cite{Penrose_Rindler84} and \cite{Mathews_Spinors_horospheres}.

Multiplying a spinor $(\xi, \eta)$ on the right by a unit quaternion $x$ yields another spinor. Generalising the complex case, the family of quaternionic spinors $(\xi, \eta)x$ all yield the same point $p$, and the multiflags rotate about $p \R$. This rotation is essentially given by the standard identification of unit quaternions with $\Spin(3)$, the double cover of the 3-dimensional rotation group.

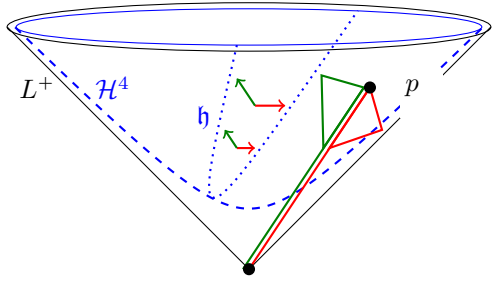
\begin{figure}[h]
\begin{center}
\begin{tikzpicture}[scale=0.8]
  \draw[black] (-4,4)--(0,0)--(4,4);
  \draw[blue, dashed, thick] plot[variable=\t,samples=1000,domain=-75.5:75.5] ({tan(\t)},{sec(\t)});
  \draw[black] (0,4) ellipse (4cm and 0.4cm);
  \draw[blue, dotted, thick] (-0.2,3.7) .. controls (-1,0.25) .. (1.8,4.27);
  \draw[blue] (0,4) ellipse (3.85cm and 0.3cm);
  \draw[red, thick] (0,0)--(2,3);
  \node[black] at (-3.5,3){$L^+$};
	\fill[white](2.5,2.5)--(3.2,2.5)--(3.2,3.3)--(2.5,3.3)--cycle;
  \node[black] at (2.7,3){$p$};
  \draw[red, thick] (2,3)--(2.2,2.3)--(1.33,2)--(2,3);
	\draw[green!50!black, thick] (-0.1,0)--(1.9,3);
  \draw[green!50!black, thick] (1.9,3)--(1.2,3.2)--(1.23,2)--(1.9,3);
	\fill[black] (0,0) circle (0.1cm);
  \fill[black] (2,3) circle (0.1cm); 

  \node[blue] at (-0.75,2.5){$\h$};
  \node[blue] at (-2.25,3){$\hyp^4$};
  \draw[red, ->, thick] (0.1,2.7)--(0.6,2.7);
	\draw[green!50!black, ->, thick] (0.1,2.7)--(-0.2,3.15);
  \draw[red, ->, thick] (-0.2,2)--(0.1,2);
	\draw[green!50!black, ->, thick] (-0.2,2)--(-0.4,2.3);
\end{tikzpicture}
\caption{Multiflag and decorated horosphere corresponding to a quaternionic spinor.}
\label{Fig:1}
\end{center}
\end{figure}

The map sending $\kappa = (\xi, \eta)$ to $p$, which we denote $\phi_1$ as in \cite{Mathews_Spinors_horospheres}, is closely related to the quaternionic Hopf fibration. Scaling $\kappa$ by a real factor also scales $p$ by a real factor and preserves the $i$- and $j$-flags. Scaling $\kappa$ to so that it lies on $S^7$, i.e. so that $|\xi|^2 + |\eta|^2 = 1$, this map is 
a restriction of the Hopf fibration. These spinors $S\HH \cap S^7$ are precisely the preimage of the equatorial $S^3 \subset S^4$ given by $\$\R^3 \cup \{\infty\} \subset \HH \cup \{\infty\}$ under the Hopf fibration. 
We discuss this in \refsec{spinors_to_light_cone}.

In defining multiflags, we find that the tangent space to $S\HH$ at a point $\kappa$ has an interesting structure: it is naturally parallelisable and decomposes into the direct sum of three orthogonal real subspaces (\reflem{TSH}). First, there is the 1-dimensional subspace in the radial direction. Second, there is the 3-dimensional subspace given by the fibres of the Hopf fibration: these are the directions in which $p$ remains constant but multiflags rotate. Third, there is a 3-dimensional subspace $(\eta', -\xi')\$\R^3$, a copy of the paravectors in the tangent space, which maps isomorphically and conformally onto the tangent space of the celestial sphere under the derivative of $\phi_1$. We prove this in \refprop{Derivs_props} and \refprop{paravectors_conformal}. 
 These latter two subspaces span a 6-dimensional subspace of the tangent space to the $S^7$ centred at the origin through $\kappa$. The mapping $(\xi, \eta) \mapsto (\eta', -\xi')$ is somewhat analogous the almost complex structure of a K\"{a}hler structure.

Multiflags in the above sense are equivalent to a choice of lightlike subspace $\ell = p \R$, and an orientation-preserving conformal linear identification $\psi$ of the Euclidean spacelike 3-dimensional space $\ell^\perp / \ell$ with $\$\R^3$. The $i$- and $j$-flags, which lie in $\ell^\perp$, project to oriented lines in $\ell^\perp/\ell$ which are identified with the directions of $i$ and $j$. The line $\ell$ can be regarded as an ideal point of $\hyp^4$, and the conformal isomorphism $\psi \colon \$\R^3 \To \ell^\perp/\ell$ as analogous to a decoration on a horosphere, so we call such pairs $(\ell, \psi)$ \emph{decorated ideal points}. We discuss decorated ideal points in \refsec{decorated_ideal_points}.

The space of multiflags $\MF$ is naturally diffeomorphic to $S^3 \times SO(3) \times \R$, with the $SO(3)$ factor corresponding to rotations of flags. Its double (universal) cover $\widetilde{\MF} \cong S^3 \times S^3 \times \R$ then consists of \emph{spin multiflags}, which must rotate through $4\pi$ to return to the same flag. The proof of \refthm{main_thm_1} constructs an explicit diffeomorphism $S\HH \cong \widetilde{\MF}$.

\subsection{Horospheres and decorations}
\label{Sec:horospheres_decorations}

As in \cite{Mathews_Spinors_horospheres, Penner87}, from a point $p$ on the future light cone in Minkowski space, one can associate a horosphere $\h$, by intersecting the hyperboloid model of hyperbolic space with the affine hyperplane given by the equation
\[
\langle p, x \rangle = 1.
\]
See \reffig{1}. This association between the future light cone and horospheres works in any dimension. Hence, just as real and complex spinors yield horospheres in $\hyp^2$ or $\hyp^3$, quaternionic spinors yield horospheres in $\hyp^4$. 

In the complex case, the flag associated to a spinor $\kappa$ describes a direction field on the corresponding horosphere $\h$ in $\hyp^3$, which is in fact parallel. This is possible as $\h$ is isometric to a Euclidean plane. In \cite{Mathews_Spinors_horospheres} we defined a \emph{decoration} on a horosphere in $\hyp^3$ to be such a parallel direction field. Multiplying $\kappa$ by $re^{i\theta}$ translates $\h$ by $2 \log r$ towards its centre and rotates its decoration by $\theta$.

In the quaternionic case, each of the two flags in a multiflag describes a direction field on the corresponding horosphere $\h$ in $\hyp^4$, which we call the \emph{$i$-direction field} and \emph{$j$-direction field}. Again these direction fields are parallel, which is possible since $\h$ is isometric to Euclidean 3-space; they are also perpendicular (just like $i$ and $j$ are in $\HH$). We show this in \refprop{line_fields_parallel}. 
We can thus define a decoration as follows.

\begin{defn}
\label{Def:decorated_horosphere}
A \emph{decoration} on a horosphere $\h$ in $\hyp^4$ is a pair of oriented parallel tangent direction fields $L_i$ and $L_j$ on $\h$, such that $L_i$ and $L_j$ are everywhere perpendicular.
\end{defn}
Thus, from a quaternionic spinor $\kappa \in S\HH$, we obtain a decoration on a horosphere $\h$.
Being parallel, once $L_i$ and $L_j$ are perpendicular at one point of $\h$, they are perpendicular at all points of $\h$. Multiplying $\kappa$ by $r > 0$ again translates $\h$ by $2 \log r$ towards its centre. Multiplying $\kappa$ by a unit quaternion rotates the direction fields of the decoration using the isomorphism of unit quaternions with $\Spin(3)$. 

In $\hyp^3$, a decoration on a horosphere $\h$, together with a choice of normal to the horosphere, determines a field of parallel oriented orthonormal frames along $\h$.  A \emph{spin decoration} is roughly a lift of such a frame field to the spin double cover of the frame bundle. This spin construction entails that a rotation of a frame about an angle of $2\pi$ does not result in the same frame, but a rotation of a frame by an angle of $4\pi$ does. Thus, whereas angles are usually measured modulo $2\pi$, after lifting to the spin double cover, angles are measured modulo $4\pi$.

In $\hyp^4$ a similar construction applies. Since a decoration on a horosphere $\h$ consists of two perpendicular direction fields, having a decoration, together with a choice of normal, again yields a parallel frame field along $\h$. 
A choice of decoration is equivalent to a choice of oriented orthonormal frame. The space of horospheres in $\hyp^4$ is diffeomorphic to $S^3 \times \R$, with a choice in $S^3 \cong \partial \hyp^4$ for the centre of the horosphere, and then a choice in $\R$ for its size. 
The space of decorated horospheres is diffeomorphic to $S^3 \times SO(3) \times \R$.

A \emph{spin decoration} in $\hyp^4$, in a similar way to $\hyp^3$, is essentially a lift of a frame field to the spin double cover of the frame bundle. We give the formal definition is \refdef{associated_spin_decorations}. Again, this results in a situation where rotation by $4\pi$ is required to return a frame to itself, so that angles are measured modulo $4\pi$. The space of spin-decorated horospheres is diffeomorphic to $S^3 \times S^3 \times \R$, as is $S\HH$. The proof of \refthm{main_thm_1} gives an explicit diffeomorphism between these spaces.

\subsection{Spin isometries of hyperbolic 4-space and M\"{o}bius transformations}
\refthm{main_thm_1} includes a statement about equivariance with respect to a group $SL_2\$$. This is a naturally arising group of matrices in the study of M\"{o}bius transformations and 4-dimensional hyperbolic isometries, appearing for example in \cite{Ahlfors_Clifford85, Ahlfors_Mobius85, Ahlfors_Mobius_86, Gongopadhyay_12, Ahlfors_fixedpoints_85, Ahlfors_84, Cao_Waterman_98, Waterman_93, Kellerhals01}. However it is more subtle than $SL_2$ over a commutative ring. 
We use the following definition. 
\begin{defn}
\label{Def:SL2H}
The group $SL_2 \$ $ consists of \emph{Clifford matrices}, which are matrices
\[
\begin{pmatrix} a & b \\ c & d \end{pmatrix}, \quad
a,b,c,d \in \HH,
\]
such that the following hold:
\begin{gather}
\label{Eqn:Vahlen_conditions_1}
ab^*, cd^*, c^* a, d^* b, ba^*, dc^*, a^* c, b^* d \in \$\R^3 \\
\label{Eqn:Vahlen_conditions_2}
ad^* - bc^* = da^* - cb^* = d^* a - b^* c = a^* d - c^* b = 1
\end{gather}
\end{defn}
Such matrices go by several other names in the literature;
we discuss our notation and terminology in \refsec{Mobius_hyperbolic}.

The conditions of \refeqn{Vahlen_conditions_1}--\refeqn{Vahlen_conditions_2} appear onerous but in fact are greatly redundant (and in fact such matrices satisfy many more similar conditions); we have simply included them for convenience and symmetry. 
For instance, $SL_2\$$ can be defined simply as the set of $2 \times 2$ quaternionic matrices with pseudo-determinant $1$ whose columns are spinors. In \refsec{Clifford_conditions} we discuss various equivalent formulations.

Clifford matrices have numerous interesting properties. We discuss some of them in \refsec{Clifford_properties}.
In particular, the action of $SL_2\$$ on $\HH^2$ by standard matrix-vector multiplication preserves the 7-real-dimensional subspace $S\HH$.
The action on $S\HH$ preserves two of the three summands of its tangent bundle discussed above, and its behaviour on the third summand is rather subtle but has interesting conformal properties, which are related to the equivariance in \refthm{main_thm_1}. We discuss this behaviour in \refsec{SL2_on_spinors} and \refsec{action_SL2_tangent_spinors}.

In the complex case, the group $SL_2 \C$ has quotient $PSL_2 \C$ by $\{\pm I\}$, which is is isomorphic to the group of M\"{o}bius transformations of $\C \cup \{\infty\}$:
\begin{equation}
\label{Eqn:complex_Mobius}
\pm \begin{pmatrix}
a & b \\ c & d
\end{pmatrix} 
\in PSL_2 \C
\quad
\text{corresponds to}
\quad
z \mapsto \frac{az+b}{cz+d}.
\end{equation}
Via these M\"{o}bius transformations, $PSL_2 \C$ acts on $\C \cup \{\infty\}$, which can be regarded as the boundary at infinity of the upper half space model of $\hyp^3$. Each M\"{o}bius transformation then extends to an orientation-preserving isometry of $\hyp^3$ and in fact there is an isomorphism $PSL_2 \C \cong \Isom^+ \hyp^3$ with the orientation-preserving isometry group of $\hyp^3$.

A similar story is known to arise with $SL_2\$$ and 4-dimensional hyperbolic geometry, and has been studied by Ahlfors, Cao, Gongopadhyay, Kellerhals, Waterman and others; further details and references are given in \refsec{Mobius_hyperbolic} and \refsec{paravector_Mobius}. We denote the quotient of $SL_2\$$ by the normal subgroup $\{ \pm I \}$ as $PSL_2\$$. Elements of $PSL_2\$$ can be regarded as M\"{o}bius transformations, although non-commutativity means that the fractional expression in \refeqn{complex_Mobius} no longer makes sense. Instead, we have 
\begin{equation}
\label{Eqn:Mobius_from_SL2H}
\pm \begin{pmatrix}
a & b \\ c & d
\end{pmatrix} 
\in PSL_2\$
\quad
\text{corresponding to}
\quad
v \mapsto (av+b)(cv+d)^{-1}.
\end{equation}
When $v$ is a paravector (i.e. $v$ lies in the 3-real-dimensional subspace $\$\R^3 \subset \HH$), $(av+b)(cv+d)^{-1}$ also lies in $\$\R^3$ (or is $\infty$). Thus, $PSL_2\$$ acts on $\$\R^3 \cup \{\infty\}$ by M\"{o}bius transformations with quaternionic coefficients. We can regard $\$\R^3 \cup \{\infty\}$ as the boundary at infinity of the upper half space model of $\hyp^4$. Each M\"{o}bius transformation of $\$\R^3 \cup \{\infty\}$ extends to an orientation-preserving isometry of $\hyp^4$ and there is an isomorphism $PSL_2\$ \cong \Isom^+ \hyp^4$.

The group $\Isom^+ \hyp^4$ is diffeomorphic to $\R^4 \times SO(4)$, and its spin double cover, or \emph{spin isometry} group $\Isom^S \hyp^4$, is diffeomorphic to $\R^4 \times \Spin(4)$. There is an isomorphism $SL_2\$ \cong \Isom^S \hyp^4$.

In \cite{Mathews_Spinors_horospheres}, we showed that the association of spin-decorated horospheres to complex spinors was $SL_2\C$-equivariant, where $SL_2 \C$ acts on complex spinors by matrix-vector multiplication, and on spin-decorated horospheres by spin isometries of $\hyp^3$. In a similar way here $SL_2\$$ acts on quaternionic spinors by matrix-vector multiplication, and on spin multiflags and spin-decorated horospheres by spin isometries of $\hyp^4$. \refthm{main_thm_1} asserts that these actions are equivariant.

\subsection{Explicit description in the upper half space model}

When the upper half space model of $\hyp^3$ is used, the spinor--horosphere correspondence has a particularly simple expression. In the 3-dimensional case, with the sphere at infinity $\partial \hyp^3 \cong S^2$ regarded as $\C \cup \{\infty\}$, horospheres appear either as Euclidean spheres tangent to $\C$, or horizontal planes (tangent to $\partial \hyp^3$ at $\infty$). In the former case, the highest point (``north pole") of a horosphere in the model has tangent plane parallel to $\C$, and a decoration can be described by a nonzero complex number, corresponding to the direction at this highest ``north pole" point. In the latter case, any point on the horosphere has tangent plane parallel to $\C$, and again the decoration can be described by a nonzero complex number. It is shown in \cite{Mathews_Spinors_horospheres} that the spinor $(\xi, \eta)$ corresponds to a horosphere centred at $\xi/\eta$. If $\eta \neq 0$ then $\xi/\eta \in \C$, and the horosphere has Euclidean diameter $|\eta|^{-2}$, with decoration described at the north pole by $i \eta^{-2}$. If $\eta = 0$, then $\xi/\eta = \infty$, and the horosphere lies at Euclidean height $|\xi|^2$, with decoration given by $i \xi^2$.

These descriptions generalise to the case of quaternions and the upper half space model $\U$ of $\hyp^4$, with the paravectors $\$\R^3$ taking the place of $\C$. Regarding $\partial \hyp^4$ as $\$\R^3 \cup \{\infty\}$, horospheres appear either as Euclidean 3-spheres tangent to $\$\R^3$, or horizontal planes (tangent to $\partial \hyp^4$ at $\infty$). In the former case, at the highest ``north pole" point of a horosphere, the tangent plane is parallel to $\$\R^3$, and in the latter case, any point on the horosphere has tangent plane parallel to $\$\R^3$. Thus parallel direction fields can be described by nonzero elements of $\$\R^3$.

\begin{theorem}
\label{Thm:main_thm_4}
Under the correspondence of \refthm{main_thm_1}, a spinor $(\xi, \eta)$ corresponds to a horosphere $\h$ centred at $\xi \eta^{-1}$ in $\U$. 
\begin{enumerate}
\item
If $\eta \neq 0$ then $\h$ appears in $\U$ as a Euclidean 3-sphere with Euclidean diameter $|\eta|^{-2}$, its $i$-direction field is north-pole described by $\eta^{-1*} i \eta^{-1}$, and its $j$-direction field is north-pole described by $\eta^{-1*} j \eta^{-1}$.
\item
If $\eta = 0$ then $\h$ appears in $\U$ as a Euclidean 3-plane at Euclidean height $|\xi|^2$, with $i$-direction field given by $\xi i \xi^*$ and its $j$-direction field given by $\xi j \xi^*$.
\end{enumerate}
\end{theorem}
This description is illustrated schematically in \reffig{2}.
For complex numbers, the $*$-conjugation is trivial, so when $(\xi, \eta)$ is a complex spinor, the $i$-line field points in the direction of the decorations of \cite{Mathews_Spinors_horospheres}; and since any complex number $z$ satisfies $zj = j\overline{z}$, the $j$-direction field simply points in the $j$ direction.
We may also observe that, just as for complex numbers, any nonzero quaternion $\eta$ satisfies $\eta^{-1} = \overline{\eta}/|\eta|^2$, so
\[
\xi \eta^{-1} = \frac{\xi \overline{\eta}}{|\eta|^2}.
\]
The condition $\xi \overline{\eta} \in \$\R^3$ in the definition of a quaternionic spinor thus ensures that $\xi \eta^{-1} \in \$\R^3 \cup \{\infty\}$, and hence that the description of the horosphere in \refthm{main_thm_4} makes sense.

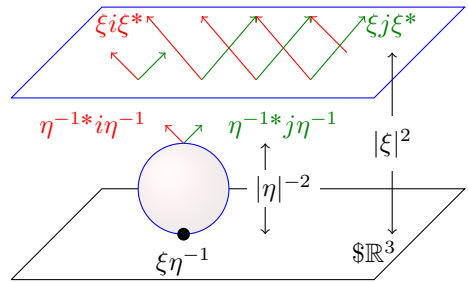
\begin{figure}[h]
\begin{center}
\begin{tikzpicture}[scale=1.2]
    \draw[black] (-2,-0.5)--(2,-0.5)--(3,0.5)--(-1,0.5)--(-2,-0.5);
    \fill[white] (-0.1,0.5) circle (0.5cm);
    \shade[ball color = red!40, opacity = 0.1] (-0.1,0.5) circle (0.5cm);
    \draw[blue] (-0.1,0.5) circle (0.5cm);
		\fill[black] (-0.1,0) circle (0.07cm);
		\draw[red, ->] (-0.1,1)--(-0.3,1.2);
		\draw[green!50!black, ->] (-0.1,1)--(0.1,1.2); 
		\node[red] at (-1.1,1.2) {$\eta^{-1*} i \eta^{-1}$};
		\node[green!50!black] at (1,1.2) {$\eta^{-1*} j \eta^{-1}$};
		\node[black] at (-0.1,-0.3) {$\xi \eta^{-1}$};
		\draw[<->] (0.8,0)--(0.8,1);
		\fill[white] (0.6,0.3)--(1.4,0.3)--(1.4,0.7)--(0.6,0.7)--cycle;
		\node[black] at (1,0.5) {$|\eta|^{-2}$};
    \draw[blue] (-2,1.5)--(2,1.5)--(3,2.5)--(-1,2.5)--(-2,1.5);
    \begin{scope}[xshift=0.5cm]
		\draw[red,->] (-1.1,1.7)--(-1.4,2);
    \draw[red,->] (-0.4,1.7)--(-1,2.4);
    \draw[red,->] (0.2,1.7)--(-0.4,2.4);
    \draw[red,->] (0.8,1.7)--(0.2,2.4);
    \draw[red,->] (1.2,2)--(0.8,2.4);
		\draw[green!50!black,->] (-1.1,1.7)--(-0.8,2);
    \draw[green!50!black,->] (-0.4,1.7)--(0.2,2.4);
    \draw[green!50!black,->] (0.2,1.7)--(0.8,2.4);
    \draw[green!50!black,->] (0.8,1.7)--(1.4,2.4);
		\node[red] at (-1.3,2.3) {$\xi i \xi^*$};
		\node[green!50!black] at (1.7,2.3) {$\xi j \xi^*$};
		\end{scope}
		\draw[<->] (2.2,0)--(2.2,2);
		\fill[white] (1.8,0.7)--(2.6,0.7)--(2.6,1.3)--(1.8,1.3)--cycle;
		\node[black] at (2.2,1) {$|\xi|^2$};
		\node[black] at (2,-0.2) {$\$\R^3$};
\end{tikzpicture}
\caption{Decorated horospheres as they appear in the upper half space model, with $i$-decoration fields shown in red and $j$-decoration fields in green. This picture is adapted from figure 2 of \cite{Mathews_Spinors_horospheres}.}
\label{Fig:2}
\end{center}
\end{figure}

Observe that all the directions given in \refthm{main_thm_4} are of the form $qvq^*$. This operation is well known to be an important one. Indeed, let us define the map $\sigma$ by
\begin{equation}
\label{Eqn:rho}
\sigma(q)(v) = qvq^*,
\end{equation}
where $q,v \in \HH$. A similar operation has been studied in, e.g., \cite{Ahlfors_Mobius85, Ahlfors_Clifford85, Ahlfors_Lounesto_89, Ahlfors_Mobius_86, Ahlfors_fixedpoints_85}. Any $\sigma(q)$ preserves $\$\R^3$, and when $|q|=1$ is a Euclidean rotation. In fact over the $S^3$ of unit quaternions, $\sigma$ provides the standard description of quaternions as rotations, $S^3 \cong \Spin(3)$.

In particular, all the claimed direction fields at north poles or on horizontal horospheres in \refthm{main_thm_4} lie in $\$\R^3$, so the statement makes sense.

In general, a nonzero quaternion $q$ can be written in polar form as $q = r e^{u\theta}$, where $r \geq 0$, $\theta \in \R$ is an angle, and $u$ is an imaginary unit quaternion. Then $\sigma(q)$, applied to $\$\R^3$, is a dilation by $r^2$, and rotation by $2 \theta$ about an axis determined by $u$, namely $-uk \in \$\R^3$. See \refprop{rho_rotation_dilation} for details. Extended to $\hyp^4$, $\sigma(q)$ is an isometry that translates by $2 \log r$ and rotates by $2\theta$.

\subsection{Quaternionic lambda lengths in hyperbolic 4-space}
\label{Sec:intro_lambda}

We give a rough idea of quaternionic lambda lengths in 4 dimensions; full details are presented in Sections \ref{Sec:orientations_frames_spin}--\ref{Sec:quaternionic_lambda}.

We first describe the 3-dimensional case. Any two horospheres $\h_1, \h_2$ in $\hyp^3$ with spin decorations $W_1, W_2$, have a complex translation distance $d = \rho + i \theta$ from $(\h_1, W_1)$ to $(\h_2, W_2)$ as follows. Let $\rho$ be the oriented distance from $\h_1$ to $\h_2$ along their common perpendicular geodesic $\gamma$. Translating a spin frame of $W_1$ along $\gamma$, by distance $\rho$, takes a frame from a point of $\h_1$ to a point of $\h_2$. Then, rotation by some angle $\theta$ aligns the frame of $W_1$ with that of $W_2$. If $W_1, W_2$ are just decorations (rather than spin-decorations), then $\theta$ is well defined modulo $2\pi$; if $W_1, W_2$ are spin decorations, then $\theta$ is well defined modulo $4\pi$.  The lambda length $\lambda_{12}$ from $\h_1$ to $\h_2$ is defined to be
\[
\lambda_{12} = e^{d/2} = \exp \left( \frac{\rho + i \theta}{2} \right).
\]
When dealing with non-spin decorations, $\theta/2$ is only well defined modulo $\pi$, so $\lambda$ is ambiguous up to sign. With spin decorations however $\lambda$ is well defined.

\begin{figure}[h]
\def\svgwidth{0.38\columnwidth}
\begin{center}
\begingroup%
  \makeatletter%
  \providecommand\color[2][]{%
    \errmessage{(Inkscape) Color is used for the text in Inkscape, but the package 'color.sty' is not loaded}%
    \renewcommand\color[2][]{}%
  }%
  \providecommand\transparent[1]{%
    \errmessage{(Inkscape) Transparency is used (non-zero) for the text in Inkscape, but the package 'transparent.sty' is not loaded}%
    \renewcommand\transparent[1]{}%
  }%
  \providecommand\rotatebox[2]{#2}%
  \newcommand*\fsize{\dimexpr\f@size pt\relax}%
  \newcommand*\lineheight[1]{\fontsize{\fsize}{#1\fsize}\selectfont}%
  \ifx\svgwidth\undefined%
    \setlength{\unitlength}{247.70801514bp}%
    \ifx\svgscale\undefined%
      \relax%
    \else%
      \setlength{\unitlength}{\unitlength * \real{\svgscale}}%
    \fi%
  \else%
    \setlength{\unitlength}{\svgwidth}%
  \fi%
  \global\let\svgwidth\undefined%
  \global\let\svgscale\undefined%
  \makeatother%
  \begin{picture}(1,0.64212724)%
    \lineheight{1}%
    \setlength\tabcolsep{0pt}%
    \put(0,0){\includegraphics[width=\unitlength,page=1]{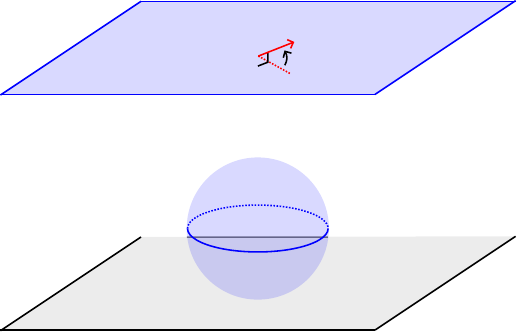}}%
    \put(0.77703916,0.58512719){\color[rgb]{0,0,0}\makebox(0,0)[lt]{\lineheight{1.25}\smash{\begin{tabular}[t]{l}$H_1$\end{tabular}}}}%
    \put(0,0){\includegraphics[width=\unitlength,page=2]{complex_lambda_lengths.pdf}}%
    \put(0.45375586,0.10699279){\color[rgb]{0,0,0}\makebox(0,0)[lt]{\lineheight{1.25}\smash{\begin{tabular}[t]{l}$H_2$\end{tabular}}}}%
    \put(0,0){\includegraphics[width=\unitlength,page=3]{complex_lambda_lengths.pdf}}%
    \put(0.39200427,0.4027793){\color[rgb]{0,0,0}\makebox(0,0)[lt]{\lineheight{1.25}\smash{\begin{tabular}[t]{l}$\rho$\end{tabular}}}}%
    \put(0.57579002,0.51343701){\color[rgb]{0,0,0}\makebox(0,0)[lt]{\lineheight{1.25}\smash{\begin{tabular}[t]{l}$\theta$\end{tabular}}}}%
    \put(0,0){\includegraphics[width=\unitlength,page=4]{complex_lambda_lengths.pdf}}%
  \end{picture}%
\endgroup%
 
\caption{Complex distance between horospheres, copied from \cite{Mathews_Spinors_horospheres}.}
\label{Fig:3}
\end{center}
\end{figure}

To define the sense of the rotation, one must proceed a little more carefully. A spin decoration is defined to include spin frames of \emph{either} orientation on a horosphere, \emph{inwards} and \emph{outwards}. In order for orientations to match, we translate the \emph{inwards} spin frame of $W_1$ along $\gamma$, and then rotate it to align with the \emph{outwards} spin frame of $W_2$.

To describe rotations of frames on horospheres in $\hyp^4$, we use the operation $\sigma$, and identify the tangent spaces of horospheres with $\$\R^3$. Different identifications lead to different numbers in $\$\R^3$ describing the axis of the rotation. However, the $i$- and $j$-direction fields of a decoration on a horosphere $\h$ provide a canonical way to identify the directions of $i$ and $j$ in $T\h$. The remaining vector in an orthonormal frame, of desired orientation, can be identified with $1$. With the three basis vectors of an orthonormal frame thus identified with $i,j$ and $1$, the decoration provides a \emph{paravector  identification}, i.e. an identification of each tangent space of $\h$ with $\$\R^3$. A normal direction can then be given by $k$. We describe these matters in detail in \refsec{orientations_frames_spin}.

Given two spin-decorated horospheres $(\h_1, W_1)$ and $(\h_2, W_2)$, we then have notions of quaternionic translation distance and lambda length as follows. As in the 3-dimensional case, we have a signed distance $\rho$ from $\h_1$ to $\h_2$ along the oriented geodesic $\gamma$. Using the paravector identification coming from a spin-decoration (as it turns out, it does not matter which decoration we use, $W_1$ or $W_2$; see \reflem{rotation_coordinates_invariant}), translating the inwards spin frame of $W_1$ by $\rho$ along $\gamma$ and then rotating by angle $\theta$ (well defined modulo $4\pi$) about an axis of rotation given by a unit vector $v \in \$\R^3$ (using the paravector identification) aligns it with the outward spin frame of $W_2$. The quaternionic translation distance from $(\h_1, W_1)$ to $(\h_2, W_2)$ is then $d = \rho + \theta vk $; note that as $v \in \$\R^3$, $vk$ is a unit imaginary quaternion. The lambda length $\lambda_{12}$ is then given by
\begin{equation}
\label{Eqn:lambda_length}
\lambda_{12} = e^{d/2} = \exp \left( \frac{\rho + \theta vk}{2} \right).
\end{equation}
Full details are given in \refsec{spin_decorations_multiflags} and \refsec{quaternionic_lambda}.

When the spinors are complex, the corresponding decorated horospheres have $j$-direction fields given by the $j$-direction in the upper half space model, so the rotation of frames from $W_1$ to $W_2$ lie entirely in the real-2-dimensional tangent plane to a horosphere identified with $\C$. Thus the axis of rotation is in the $\pm j$ direction, so we have $v = \pm j$ and $vk = \pm i$. Then $e^{\theta vk/2} = e^{\pm \theta i}$ and the lambda length reduces to the definition of \cite{Mathews_Spinors_horospheres}.

Although it can be deduced directly from \refthm{main_thm_2} that $\lambda_{12} = - \lambda_{21}^*$, this can also be proved directly from the geometric definition of lambda length; we prove this in \refprop{lambda_length_antisymmetric}.

\subsection{Quasideterminants, Pl\"{u}cker and Ptolemy equations}

The theory of determinants of matrices over noncommutative rings like $\HH$ is quite different from the commutative case. The pseudo-determinant of \refeqn{lambda_pdet} is just one among several notions of determinant. A general theory of \emph{quasideterminants} over noncommutative rings was developed by Gel'fand and Retakh in several papers \cite{Gelfand_Retakh_91, Gelfand_Retakh_92, Gelfand_Retakh_97}; see also \cite{Retakh_Wilson_lectures, GGRW_05}. We apply this theory to matrices of quaternions. For general accounts of determinants of quaternionic matrices, though not including the Gel'fand--Retakh theory, see \cite{Aslaksen_96, Zhang_97}.

In \cite{Mathews_Spinors_horospheres}, with complex lambda lengths given by $2 \times 2$ complex determinants, the Pl\"{u}cker relation between determinants of $2 \times 2$ submatrices of a $2 \times 4$ matrix yields a Ptolemy equation on lambda lengths. In other words, the Ptolemy equation is a Pl\"{u}cker equation.

In the noncommutative case, Gel'fand and Retakh show that there are analogues of Pl\"{u}cker relations. In particular, they define \emph{left quasi-Pl\"{u}cker coordinates} $p_{lm}^I (A)$ of a matrix $A$ over a noncommutative ring, and prove they satisfy certain relations. We 
show in \reflem{quasi-Plucker_lambda} that when $A$ is a quaternionic matrix with spinor columns $\kappa_l$ we obtain 
\begin{equation}
\label{Eqn:quasi-Plucker_lambda}
p_{lm}^n (A) = \lambda_{nl}^{-1} \lambda_{nm}.
\end{equation}
After proving \refthm{main_thm_2} and \refeqn{quasi-Plucker_lambda}, the Gel'fand--Retakh Pl\"{u}cker relation yields \refthm{main_thm_3} immediately. We show this in \refsec{Plucker}. 

Gel'fand--Retakh also prove a certain \emph{skew-symmetry} property of quasi-Pl\"{u}cker coordinates, namely certain triples of them multiply to $-1$. This yields the following result on lambda lengths, proved in  \refsec{Plucker}.
\begin{prop}
\label{Prop:triangle_holonomy}
Given any 3 spin-decorated horospheres $(\h_n, W_n)$ in $\hyp^4$, $n=1,2,3$, let $\lambda_{mn}$ denote the lambda length from $(\h_m, W_m)$ to $(\h_n, W_n)$. Then
\[
\lambda_{12} \lambda_{32}^{-1} \lambda_{31} \in \$\R^3 \cup \{\infty\}.
\]
\end{prop}

Berenstein--Retakh in \cite{Berenstein_Retakh_18, Retakh_OW_report_13} discussed related quantities as \emph{noncommutative angles}, satisfying \emph{triangle relations} which are similar to the above.
A possible interpretation of \refthm{main_thm_3} is that it provides information about  ``lambda length holonomy" around paths between horospheres. We can regard $\lambda_{mn}$ as describing the holonomy involved in translating and rotating a spin frame along the oriented geodesic $\h_m \rightarrow \h_n$, and $\lambda_{mn}^{-1}$ referring to the holonomy backwards from $\h_n$ to $\h_m$ along the oriented geodesic $\h_m \to \h_n$. Because of the alternating inwards and outwards frames, the holonomy moves alternately forwards and backwards along edges, in a manner reminiscent of the oscillating curves of the first author and Purcell \cite{Mathews_Purcell_Ptolemy_hyperbolic}.
After rewriting \refeqn{noncomm_Plucker} as
\[
\lambda_{02} = \lambda_{01} \lambda_{31}^{-1} \lambda_{32} 
+ \lambda_{03} \lambda_{13}^{-1} \lambda_{12},
\]
\refthm{main_thm_3} can be interpreted as describing the holonomy from $\h_0 \to \h_2$ in terms of the holonomy detouring to $\h_1$ and $\h_3$ along the way; and \refprop{triangle_holonomy} describes the holonomy around a triangle. 
\[
\begin{tikzpicture}[scale=0.9]
	\fill[black] (0,0) circle (0.1cm);
	\node[black] at (-0.5,0) {$\h_0$};
	\fill[black] (1,1) circle (0cm);
	\fill[black] (1,-1) circle (0cm);
	\fill[black] (2,0) circle (0.1cm);
	\node[black] at (2.6,0) {$\h_2$};
	
	\draw[black, -latex] (0,0) -- (1,0);
	\draw[black] (1,0)--(2,0);
	\node[black] at (1,0.5) {$\lambda_{02}$};
	
	\draw[ultra thick, green!50!black, -latex, rounded corners] (0.2,-0.2) -- (1.8,-0.2);
	\node[black] at (3,0) {$=$};
\end{tikzpicture}
\begin{tikzpicture}[scale=0.9]
	\fill[black] (0,0) circle (0.1cm);
	\node[black] at (-0.5,0) {$\h_0$};
	\fill[black] (1,1) circle (0.1cm);
	\node[black] at (1.5,1) {$\h_3$};
	\fill[black] (1,-1) circle (0.1cm);
	\node[black] at (0.5,-1) {$\h_1$};
	\fill[black] (2,0) circle (0.1cm);
	\node[black] at (2.6,0) {$\h_2$};
	
	\draw[black, -latex] (0,0) -- (0.5,0.5);
	\draw[black] (0,0)--(1,1);
	\node[black] at (0,0.7) {$\lambda_{03}$};

	\draw[black, -latex] (1,-1)--(1,0);
	\draw[black] (1,0)--(1,1);
	\node[black] at (1.4,0.2) {$\lambda_{13}$};
	
	\draw[black, -latex] (1,-1) -- (1.5,-0.5);
	\draw[black] (1.5,-0.5)--(2,0);
	\node[black] at (1.8,-0.7) {$\lambda_{12}$};
	
	\draw[ultra thick, green!50!black, -latex, rounded corners] (0.3,0.1) -- (0.9,0.7) -- (0.9,-0.7) -- (1.1,-0.7) -- (1.7,-0.1);
	\node[black] at (3,0) {$+$};
\end{tikzpicture}
\begin{tikzpicture}[scale=0.9]
	\fill[black] (0,0) circle (0.1cm);
	\node[black] at (-0.5,0) {$\h_0$};
	\fill[black] (1,1) circle (0.1cm);
	\node[black] at (0.5,1) {$\h_3$};
	\fill[black] (1,-1) circle (0.1cm);
	\node[black] at (1.5,-1) {$\h_1$};
	\fill[black] (2,0) circle (0.1cm);
	\node[black] at (2.6,0) {$\h_2$};
	
	\draw[black, -latex] (1,1) -- (1.5,0.5);
	\draw[black] (1.5,0.5)--(2,0);
	\node[black] at (1.8,0.7) {$\lambda_{32}$};

	\draw[black, -latex] (1,1)--(1,0);
	\draw[black] (1,0)--(1,-1);
	\node[black] at (1.4,-0.2) {$\lambda_{31}$};
	
	\draw[black, -latex] (0,0) -- (0.5,-0.5);
	\draw[black] (0.5,-0.5)--(1,-1);
	\node[black] at (0,-0.7) {$\lambda_{01}$};
	
	\draw[ultra thick, green!50!black, -latex, rounded corners] (0.3,-0.1) -- (0.9,-0.7) -- (0.9,0.7) -- (1.1,0.7) -- (1.7,0.1);
\end{tikzpicture},
\quad 
\begin{tikzpicture}[scale=0.9]
	\node[black] at (-2,0) {and};
	\fill[black] (0,0) circle (0.1cm);
	\node[black] at (-0.5,0) {$\h_1$};
	\fill[black] (1,1) circle (0.1cm);
	\node[black] at (1.5,1) {$\h_2$};
	\fill[black] (1,-1) circle (0.1cm);
	\node[black] at (1.5,-1) {$\h_3$};

	\draw[black, -latex] (0,0) -- (0.5,0.5);
	\draw[black] (0,0)--(1,1);
	\node[black] at (0,0.7) {$\lambda_{12}$};

	\draw[black, -latex] (1,-1)--(1,0);
	\draw[black] (1,0)--(1,1);
	\node[black] at (1.6,0) {$\lambda_{32}$};
	
	\draw[black, -latex] (1,-1) -- (0.5,-0.5);
	\draw[black] (0.5,-0.5)--(0,0);
	\node[black] at (0,-0.7) {$\lambda_{31}$};
	
	\draw[ultra thick, green!50!black, -latex, rounded corners] (0.3,0.1) -- (0.9,0.7) -- (0.9,-0.7) -- (0.3,-0.1);
			\node[black] at (2.5 ,0) {$\in \$\R^3$};

\end{tikzpicture}
.
\]

\subsection{Structure and approach of this paper}

This paper, roughly speaking, uses three quite distinct bodies of previous work: work of Penner, Penrose, Rindler, and the first author on complex spinors, low-dimensional hyperbolic geometry, and lambda lengths; work of Ahlfors, Cao, Kellerhals, Lounesto, Maass, and Vahlen on M\"{o}bius transformations, Clifford matrices, and 4-dimensional hyperbolic isometries; and work of Berenstein, Gel'fand and Retakh on noncommutative determinants. We do not imagine that readers familiar with all these topics are particularly common. Therefore, we have tried to make this paper as self-contained as possible. 
We give background and references on all these subjects as we proceed.

We also need to develop some necessary background for our main theorems, particularly regarding quaternion paravectors, Clifford matrices and 4-dimensional hyperbolic isometries. For instance, we need an explicit description of paravector rotations on quaternions for our notion of quaternionic distance between horospheres; and we need some facts about parabolic translation matrices not to our knowledge in the literature. Therefore, we proceed by establishing the background for our theorems step by step.

The basic structure of the proof of the main \refthm{main_thm_1} is the same as for the corresponding theorem in \cite{Mathews_Spinors_horospheres}: we construct the maps of the following commutative diagram.
\begin{equation}
\label{Eqn:main_thm_diagram}
\begin{array}{ccccc}
			& 											& \widetilde{MF} & \stackrel{\widetilde{\Phi_2}}{\To} & \Hor^S \\
			& \stackrel{\widetilde{\Phi_1}}{\nearrow} & \downarrow && \downarrow \\
S\HH & \stackrel{\Phi_1}{\To} & \MF \cong \S^{+D} & \stackrel{\Phi_2}{\To} & \Hor^D \\
& \stackrel{\phi_1}{\searrow} & \downarrow & & \downarrow \\
& & L^+ & \stackrel{\phi_2}{\To} & \Hor
\end{array}
\end{equation}
Here $S\HH$ is the space of spinors; $\widetilde{\MF} \To \MF$ is the double cover from spin multiflags to multiflags; $\Hor^S \To \Hor^D$ is the double cover from spin-decorated horospheres to decorated horospheres; $\Hor$ is the space of horospheres; $L^+$ is the future light cone; and $\S^{+D}$ is the space of decorated ideal points. The remaining downward maps are forgetful, and the maps $\widetilde{\Phi_1}, \widetilde{\Phi_2}$ are the correspondences of \refthm{main_thm_1}.

In \refsec{Clifford_geometry}, we consider Clifford algebras in general, for motivation and context, before specialising to the quaternions. We begin with background on Clifford algebras in general (\refsec{Clifford_general}), including M\"{o}bius transformations, Lipschitz groups, and hyperbolic isometries in general dimension (\refsec{Mobius_hyperbolic}). In order to use quaternions for 4-dimensional hyperbolic geometry, we use the approach to M\"{o}bius transformations from paravectors, which we introduce in \refsec{vectors_paravectors}, and discuss how M\"{o}bius transformations, Lipschitz groups, and hyperbolic isometries arise in this approach (\refsec{paravector_Mobius}). We briefly consider complex numbers as a simple special case (\refsec{complex}), then proceed to consider quaternions. We discuss their basic properties (\refsec{quaternion_involution}), their paravectors and Lipschitz group (\refsec{quaternion_para}). The standard relationship of quaternions to 3-dimensional geometry appears a little differently when we take the paravector approach, which we recount, including dot and cross product (\refsec{dot_cross_paravector}), rotations (\refsec{para_rotation}), and the operation $\sigma$ of \refeqn{rho} (\refsec{actions_on_paravectors}). We suspect none of this was unknown to Ahlfors or Lounesto, though we could not find some of the explicit descriptions in the literature.

In \refsec{spinors} we turn to pairs of quaternions, and spinors. Much of this is, in some sense, new, but many of the technical properties of the spinors involved were already known from the study of Clifford matrices. We first introduce a standard inner product on $\HH^2$ (\refsec{inner_product_norm_H2}), then discuss various equivalent ways of defining spinors (\refsec{spinor_conditions}). We introduce the bracket $\{ \cdot, \cdot \}$ and some of its properties  in \refsec{bracket}, which was previously studied to some extent as the quasideterminant. We introduce certain ``complementary" spinors, which arise naturally in the tangent space to spinors, in \refsec{complementary}; these are in a certain sense analogous to a K\"{a}hler structure. We then consider the space of spinors globally (\refsec{space_of_spinors}), including its tangent space and its subspace of paravectors (\refsec{paravectors_in_TSH}). We then turn to linear maps on spinors, which are given by the Clifford matrices, a well-studied topic. We discuss various known equivalent formulations of Clifford matrices in \refsec{Clifford_conditions}, and some of their properties in \refsec{Clifford_properties}. We develop required notions of parabolic Clifford matrices in \refsec{parabolic_clifford}. We then consider the action of Clifford matrices on spinors in \refsec{SL2_on_spinors}, and its derivative in \refsec{action_SL2_tangent_spinors}.

We can then turn in \refsec{spinor_to_flag} to the construction of the map $\Phi_1$ of the first correspondence of \refthm{main_thm_1}, from spinors to multiflags. We introduce paravector Hermitian matrices and their relationship to Minkowski space in \refsec{Hermitian_Minkowski}, and discuss orientations on the various spaces involved in \refsec{orientations}. We construct the map $\phi_1$ from spinors to the light cone in \refsec{spinors_to_light_cone}, and then consider its derivative in \refsec{deriv_phi1}, as required to obtain flags. We consider the conformality of this derivative in \refsec{conformal_paravector}, and define flags (\refsec{flags}) and multiflags, so that we can construct $\Phi_1$ (\refsec{multiflags}). We use conformality to define and discuss decorated ideal points (\refsec{decorated_ideal_points}), and finally discuss the equivariance of the map in \refsec{SL2_on_paravectors_etc}.

In \refsec{Minkowski_horospheres} we turn to hyperbolic geometry, constructing the map $\Phi_2$ of the second correspondence of \refthm{main_thm_1}, from multiflags to decorated horospheres.
We establish background on 4-dimensional hyperbolic geometry in \refsec{hyp_geom_general} and then study horospheres in \refsec{horospheres_geometry}, including their isometries and relationship to parabolic matrices, and the map $\phi_2$ from the light cone to horospheres. Then in \refsec{multiflags_to_horospheres} we can define decorated horospheres, and the map $\Phi_2$. In \refsec{H4_models} we then proceed from the hyperboloid to the upper half space model, establishing the explicit description of \refthm{main_thm_4}.

In \refsec{spin_decorations} we introduce spin geometry, constructing the lifts $\widetilde{\Phi_1}$ and $\widetilde{\Phi_2}$ and proving the correspondences of \refthm{main_thm_1}. In \refsec{orientations_frames_spin} we discuss the frame fields, orientations and paravector identifications associated to a decorated horosphere. In \refsec{spin_decorations_multiflags} we can define spin decorations on horospheres, the maps $\widetilde{\Phi_1}$ and $\widetilde{\Phi_2}$, and prove \refthm{main_thm_1}. In \refsec{quaternionic_lambda} we introduce quaternionic distances and define lambda lengths. In \refsec{antisym} we prove the antisymmetry result $\lambda_{12} = -\lambda_{21}^*$ geometrically. In \refsec{lambda_lengths} we show that the pseudo-determinant gives lambda length, proving \refthm{main_thm_2}.

Finally, in \refsec{Ptolemy} we prove the Ptolemy equation of \refthm{main_thm_3}. In \refsec{quasidet_Plucker} we discuss Gel'fand--Retakh quasideterminants and quasi-Pl\"{u}cker coordinates, and in \refsec{Plucker} use their Pl\"{u}cker relation to prove \refthm{main_thm_3}.

\subsection{Remarks on notation and terminology}

In this paper, several standard conventions for notation clash, and we mitigate these issues as best we can. We discuss these issues as they arise, but the following comments apply throughout.

The letter H is standard to describe many of the objects discussed. We use $\h$ for horospheres, $\HH$ for Hamilton's quaternions, $\mathbf{H}$ for Hermitian matrices, and $\hyp$ for hyperbolic space.

We use $\hyp^4$ for 4-dimensional hyperbolic space in general and, when our considerations are model-dependent, for the hyperboloid model. We denote the conformal disc model by $\Disc$ and the upper half space model by $\U$.

The letters $i,j,k$ denote quaternions, but are also used commonly for indices. We thus use $l,m,n, \ldots$ for indices instead.

\subsection{Acknowledgments}

The first author is supported by Australian Research Council grant DP210103136.

\section{Geometry of Clifford algebras, quaternions, and paravectors}
\label{Sec:Clifford_geometry}

In this section we introduce the geometric aspects of quaternions as we need them. Much of this is well known and we simply recall results; however formulations involving paravectors may be less well known.

\subsection{Clifford algebras in general}
\label{Sec:Clifford_general}

Following the usual definition as in e.g. Lounesto \cite[ch. 14]{Lounesto_Clifford_book_01}, given a real vector space $V$ endowed with a nondegenerate quadratic form $Q$, we define the \emph{Clifford algebra} $\Cl(V,Q)$ to be the associative algebra generated by distinct subspaces of scalars $\R$ and vectors $V$ such that for all $v \in V$ we have $v^2 = Q(v)$. For present purposes, we always have $V = \R^n$. When $Q$ has signature $p,q$ with $p+q=n$, we write $\R^{p,q}$ instead of $(V,Q)$, and may take a basis $e_1, \ldots, e_n$ such that $e_m^2 = 1$ for $1 \leq m \leq p$, $e_m^2 = -1$ for $p+1 \leq m \leq n$, and $e_l e_m = - e_m e_l$ for $1 \leq l, m \leq n$. When $q=0$ we simply write $\R^n$ rather than $\R^{n,0}$.

As a real vector space, $\Cl(\R^{p,q})$ has dimension $2^n = 2^{p+q}$. A basis is given by products of distinct $e_i$ (including the empty product), i.e. $e_{m_1} \ldots e_{m_l}$ where $q \leq m_1 < \cdots < m_l \leq n$. Such a product $e_{m_1} \cdots e_{m_l}$ has \emph{degree} $l$, and the linear combinations of such products are the elements of $\Cl(\R^{p,q})$ of \emph{degree} $l$. The real vector subspace of $\Cl(\R^{p,q})$ of elements of degree $l$ has dimension $\binom{n}{l}$.

It follows from $v^2 = Q(v)$ that for any $v,w \in \R^{p,q}$ we have $vw+wv = 2 \langle v,w \rangle$, where $\langle v,w \rangle = \frac{1}{2} \left( Q(v+w) - Q(v) - Q(w) \right)$ is the symmetric nondegenerate bilinear form induced by the quadratic form $Q$. We then have $\langle v,v \rangle = Q(v)$.

A Clifford algebra $\Cl(\R^{p,q})$ naturally has three conjugation involutions. 
Several conventions exist for the notation but we follow numerous authors in the following \cite{Ahlfors_Mobius85, Ahlfors_Clifford85, Ahlfors_Lounesto_89, Ahlfors_Mobius_86, Ahlfors_84, Ahlfors_fixedpoints_85, Cao_Waterman_98, Gongopadhyay_12, Kellerhals01, Wada_90, Waterman_93}. The homomorphism $q \mapsto q'$ is induced by $v \mapsto -v$ for $v \in V$; the anti-homomorphism $q\mapsto q^*$ is induced by reversing the words $e_{m_1} \cdots e_{m_l}$ forming a basis of $\Cl(\R^{p,q})$ as a vector space; and the anti-homomorphism $q \mapsto \bar{q}$ is given by the composition of the previous two conjugations, $\bar{q} = q'^*$. 

An invertible vector $v \in \R^{p,q}$ acts on $\Cl(\R^{p,q})$ by conjugation, $x \mapsto vxv^{-1}$. If $x$ is a vector, i.e. $x \in \R^{p,q}$, then $vxv^{-1} \in \R^{p,q}$ also, and in fact $-vxv^{-1}$ is the reflection of $x$ in the plane orthogonal to $v$ in $\R^{p,q}$ with respect to the inner product $\langle \cdot, \cdot \rangle$. Thus, the action of $v$ is by negative reflections. 
The actions of $v$ and any nonzero scalar multiple of $v$ on $V$ are identical, so to obtain all reflections it suffices to consider $v$ such that $Q(v) = \pm 1$, i.e. $v^2 = \pm 1$.

Following Lounesto \cite{Lounesto_Clifford_book_01}, we define the \emph{Lipschitz group} $\Gamma_{p,q}$ of $\Cl(\R^{p,q})$ to be the multiplicative subgroup of $\Cl(\R^{p,q})$ generated by invertible vectors. 
In the literature this group is often called the \emph{Clifford group} \cite{Ahlfors_Clifford85, Ahlfors_Lounesto_89, Ahlfors_Mobius85, Ahlfors_Mobius_86, Ahlfors_84, Ahlfors_fixedpoints_85, Kellerhals01, Cao_Waterman_98}.
According to Ahlfors and Lounesto \cite{Ahlfors_Lounesto_89}, this terminology goes back to Chevalley \cite{Chevalley_54}, but the notion goes back to Lipschitz \cite{Lipschitz_1886}, and we prefer Lounesto's terminology, to avoid potential confusion with Clifford algebras and the group of Clifford matrices.

Again when $q=0$ we simply write $\Gamma_n$ rather than $\Gamma_{n,0}$. Although an element of $\Gamma_{p,q}$ may not have a well-defined integer degree, it is is either a linear combination of even-degree elements of odd-degree elements and thus has a well defined parity. Each element of $\Gamma_{p,q}$ acts on $\R^{p,q}$ by conjugation, where it acts as a composition of negative reflections.
The subgroup of $\Gamma_{p,q}$ of even-degree elements is denoted $\Gamma^+_{p,q}$ and acts on $\R^{p,q}$ by rotations.
The subgroup of $\Gamma_{p,q}$ generated by invertible vectors $v$ such that $v^2 = \pm 1$ forms the group $\Pin(p,q)$, and the subgroup of $\Gamma_{p,q}$ generated by invertible vectors $v$ such that $v^2 = 1$ forms the group $\Pin^+ (p,q)$. The subgroup of $\Pin(p,q)$ (resp. $\Pin_+ (p,q)$) of even degree elements (i.e. $\Pin(p,q) \cap \Gamma^+_{p,q}$, resp. $\Pin^+ (p,q) \cap \Gamma^+_{p,q}$) forms the group $\Spin(p,q)$ (resp. $\Spin^+ (p,q)$). The groups $\Pin(p,q)$, $\Spin(p,q)$ and $\Spin^+ (p,q)$ are the spin double covers of $O(p,q)$, $SO(p,q)$ and $SO^+ (p,q)$ respectively. See e.g. \cite[ch. 17.2]{Lounesto_Clifford_book_01} for further details.

\subsection{M\"{o}bius transformations and hyperbolic isometries}
\label{Sec:Mobius_hyperbolic}

M\"{o}bius transformations can be obtained from Clifford algebras in two distinct ways; see e.g. \cite{Ahlfors_Lounesto_89, Ahlfors_84} or \cite[ch. 19]{Lounesto_Clifford_book_01}. We first recount the more straightforward method, before discussing the second method in \refsec{paravector_Mobius}, which we rely on in the sequel. We follow the general approach of Ahlfors in several papers \cite{Ahlfors_Clifford85, Ahlfors_Mobius85, Ahlfors_Mobius_86, Ahlfors_84, Ahlfors_fixedpoints_85}, which in turn goes back to work of Maass from 1949 \cite{Maass_49}, Fueter from 1926\cite{Fueter_1926}, and Vahlen from 1902 \cite{Vahlen_1902}.

As discussed by Lounesto in \cite[sec. 19.2]{Lounesto_Clifford_book_01}, we can regard the upper half space model $\U^n$ of $\hyp^n$ as the set of $x_1 e_1 + \cdots + x_{n} e_{n}$ in $\R^n$, where $e_1, \ldots, e_n$ are standard basis vectors, all $x_1, \ldots, x_n \in \R$ and $x_n > 0$. Forming the Clifford algebra $\Cl(\R^{n})$ of $\R^{n}$ with positive definite quadratic form, and Lipschitz group $\Gamma_n$, we may consider certain matrices
\[
\begin{pmatrix} a & b \\ c & d \end{pmatrix}
\quad \text{such that} \quad
a,b,c,d \in \Gamma_n \cup \{0\}, \quad
a b^*, b^* d, d c^*, c^* a \in \R^n, \quad
ad^* - bc^* = \pm 1. 
\]
These matrices form a group; they are designed to act on $\Cl(\R^n) \cup \{\infty\}$ by M\"{o}bius transformations, preserving the subspace $\R^n \cup \{\infty\}$, by 
\[
\begin{pmatrix} a & b \\ c & d \end{pmatrix} \cdot x
= (ax+b)(cx+d)^{-1}.
\]
In fact, this group is a 4-fold cover of the group of M\"{o}bius transformations of $\R^n \cup \{\infty\}$.

Following Ahlfors and others, we call these matrices \emph{Clifford matrices} \cite{Ahlfors_Clifford85, Ahlfors_Mobius85, Ahlfors_Mobius_86, Ahlfors_fixedpoints_85, Kellerhals01, Waterman_93, Cao_Waterman_98, Wada_90} and, following \cite{Ahlfors_Clifford85, Ahlfors_Mobius85, Ahlfors_Mobius_86}, we denote the group by $SL_2 \Gamma_n$. 
These matrices can be traced back to the work of Maass \cite{Maass_49} and Vahlen \cite{Vahlen_1902}, and 
are also variously known as $SL(2,C_n)$ \cite{Kellerhals01, Waterman_93, Cao_Waterman_98}, \emph{Vahlen matrices} \cite{Lounesto_Clifford_book_01}, $SL(2,\HH)$ \cite{Gongopadhyay_12}, and $\mathcal{C}^n$ \cite{Wada_90}. Ahlfors in \cite[sec. 3.3]{Ahlfors_Mobius_86} suggests the notation $SL_2 (\Gamma_n \cup \{0\})$ would be better but is needlessly cumbersome.

By the same argument down a dimension, $SL_2 \Gamma_{n-1}$ acts on $\Cl(\R^{n-1})$ preserving $\R^{n-1} \cup \{\infty\}$. Regarding $\R^{n-1} \subset \R^{n}$ as the subspace spanned by $e_1, \ldots, e_{n-1}$, and $\Gamma_{n-1} \subset \Gamma_{n}$ then $SL_2 \Gamma_{n-1}$ acts on $\Cl(\R^{n})$ preserving $\R^{n-1} \cup \{\infty\}$ and $\R^n \cup \{\infty\}$. Hence, being orientation preserving, the action of $SL_2 \Gamma_{n-1}$ on $\R^n \cup \{\infty\}$ preserves the upper half space $\hyp^n$ consisting of $x_1 e_1 + \cdots + x_{n-1} e_{n-1} + x_{n} e_{n}$ where $x_1, \ldots, x_{n} \in \R$ and $x_{n} > 0$.

In this way, $SL_2 \Gamma_{n-1}$ yields a 4-fold cover of the group of orientation-preserving M\"{o}bius transformations of $\R^{n-1} \cup \{\infty\}$, which is also the group $\Isom^+ \hyp^n$ of orientation-preserving isometries of $\hyp^n$. In particular, $\hyp^4$ has orientation-preserving isometry group 4-fold covered by $SL_2 \Gamma_3$, where $\Gamma_3$ is the Lipschitz group of $\Cl(\R^3)$.

\subsection{Vectors and paravectors in Clifford algebras}
\label{Sec:vectors_paravectors}

Following Lounesto \cite[sec. 19.3]{Lounesto_Clifford_book_01}, we define a \emph{paravector} in a Clifford algebra $\Cl(V,Q)$ to be a real linear combination of a scalar and a vector, i.e. an element of $\R \oplus V \subset \Cl(V,Q)$. In the literature these elements are often called ``vectors", with the elements of $V$ called ``pure vectors" (e.g. \cite{Cao_Waterman_98, Ahlfors_Clifford85, Ahlfors_Mobius85, Ahlfors_Mobius_86, Ahlfors_84, Ahlfors_fixedpoints_85, Kellerhals01}), but we prefer to use the word ``vector" in its natural meaning, as the elements of the vector space $V$. Following Lounesto's general approach, we notate objects using paravectors with a $ \$ $ sign.

In $\Cl(\R^{p,q})$ then a paravector is an element of the form $x + y$ where $x \in \R$ and $y \in \R^{p,q}$. We temporarily denote by $\$V$ the vector subspace of paravectors in $\Cl(\R^{p,q})$, so $\$ V = \R \oplus \R^{p,q}$. 

The approach of the present paper, and indeed much of the work on Clifford algebras and M\"{o}bius transformations, relies on the fact that, roughly speaking, \emph{paravectors in $\R^{p,q}$ behave like vectors in $\R^{q+1,p}$}. We now briefly explain why, and refer to Lounesto \cite[ch. 19]{Lounesto_Clifford_book_01} for further details.

Given a paravector $v \in \$ V$, we can write $v = a+b+c$ where $a \in \R$ is a scalar, $b \in \R^{p,0}$ is a positive definite vector, and $c \in \R^{0,q}$ is a negative definite vector, so that $a^2, b^2 \geq 0$, $c^2 \leq 0$, and $bc+cb = 0$. Then defining $\overline{v} = a-(b+c)$, we have 
\[
v \overline{v} = (a+b+c)(a-b-c) = a^2 - b^2 - c^2 \in \R.
\]
Since $a^2 \geq 0$, $-b^2 \leq 0$, and $-c^2 \geq 0$, the $a,b,c$ components of $v$ have respectively become positive, negative and positive definite. Thus $\$ V$ naturally has a quadratic form $\$ Q(v) \colon \$ V \To \R$, given by $\$ Q(v) = v \overline{v}$, of signature $(q+1,p)$, and we have an isomorphism of vector spaces with quadratic forms $\$ V \cong \R^{q+1,p}$. We regard this $\$ V$ as the ``paravector version of $\R^{q+1,p}$, and thus write $\$ V$ as $\$ \R^{q+1,p}$. When $p=0$ we simply write $\$\R^{q+1}$ rather than $\$\R^{q+1,0}$. This means that the paravector version $\$\R^{p,q}$ of $\R^{p,q}$  is the set of paravectors in $\Cl(\R^{q,p-1})$.

A paravector $v$ is invertible if and only if $\$Q (v) = v \overline{v} \neq 0$, in which case $v^{-1} 
= \overline{v} / (v \overline{v})$. 

The quadratic form $\$ Q$ on $\$\R^{q+1,p}$ induces a symmetric nondegenerate bilinear form, also of signature $(q+1,p)$, given by
\[
\langle \cdot, \cdot \rangle_{\$} \colon \$ \R^{q+1,p} \times \$ \R^{q+1,p} \To \R, \quad
\langle v, w \rangle_\$ 
= \frac{\$ Q(v+w) - \$ Q (v) - \$ Q (w)}{2}
= \frac{v \overline{w} + w \overline{v}}{2}.
\]

\subsection{M\"{o}bius transformations from paravectors}
\label{Sec:paravector_Mobius}

The approach to M\"{o}bius transformations based on paravectors essentially replaces the vector space $\R^{p,q}$ with the paravector space $\$\R^{p,q}$. This is the approach used, for instance, 
by Ahlfors \cite{Ahlfors_Clifford85, Ahlfors_Mobius85, Ahlfors_Mobius_86, Ahlfors_84, Ahlfors_fixedpoints_85},
Cao--Waterman \cite{Cao_Waterman_98},
Gongopadhyay \cite{Gongopadhyay_12},
Kellerhals \cite{Kellerhals01} and
Waterman \cite{Waterman_93}. We refer to those papers, or Lounesto's book \cite[ch. 19]{Lounesto_Clifford_book_01} for further details.

Just as vectors in $\R^{p,q}$ have a geometric action on $\Cl(\R^{p,q})$ preserving $\R^{p,q}$, 
paravectors $v \in \$ \R^{q+1,p} \subset \Cl(\R^{p,q})$ also have nice geometric actions on $\Cl(\R^{p,q})$, which preserve $\$ \R^{q+1,p}$. 

One such action is $x \mapsto vx\overline{v}^{\, -1} = vxv/(v\overline{v})$, which preserves $\$\R^{q+1,p}$. 
If $v \in \R$ then this action is trivial, and otherwise one can show that this action is a rotation in the 2-plane spanned by $1$ and $v$ \cite{Ahlfors_Mobius85}; we show this explicitly for the quaternions below.

Another such action is $x \mapsto vxv$, which is the same action, multiplied by $v\overline{v} = \$ Q(v)$. Thus it is another action of $\$\R^{q+1,p}$ on $\Cl(\R^{p,q})$ preserving $\$\R^{q+1,p}$, on which it acts as a rotation in the 2-plane spanned by $1$ and $v$, composed with a dilation of $\$\R^{q+1,p}$ by $\$Q(v)$.

We can then define the \emph{paravector Lipschitz group} $\$\Gamma_{q+1,p}$ to be the multiplicative group generated by invertible paravectors in $\$\R^{q+1,p} \subset \Cl(\R^{p,q})$. 
The notation is also due to Lounesto \cite[sec. 19.3]{Lounesto_Clifford_book_01} (though he does not give the group a name); he attributes the idea to Porteous' 1969 work \cite{Porteous_69_81}.
Again this group also often goes by the name of \emph{Clifford group}, and the same reasoning applies for our terminology.

The conjugation $v \mapsto \overline{v}$ on paravectors $\$\R^{q+1,p}$ agrees with the conjugation $x \mapsto \overline{x}$ from $\Cl(\R^{p,q})$, so we obtain an anti-homomorphic involution $v \mapsto \overline{v}$ on $\$\Gamma_{q+1,p}$ such that $0 \neq v \overline{v} \in \R$ and $v^{-1} = \overline{v}/(v\overline{v})$ for all $v \in \$\Gamma_{q+1,p}$. 
The action of $\$\R^{q+1,p}$ on $\$\R^{q+1,p}$ by $x \mapsto vx\overline{v}^{-1}$ extends to an action of $\$\Gamma_{q+1,p}$ on $\$\R^{q+1,p}$ by orientation-preserving isometries: for $v \in \$\Gamma_{q+1,p}$, given as $v = v_1 \cdots v_m$ where each $v_k \in \$\R^{q+1,p}$ is invertible, and $x \in \$\R^{q+1,p}$ we have
\begin{align*}
x &\mapsto 
v_1 \cdots v_m \; x \; \overline{v_m}^{\,-1} \, \cdots \, \overline{v_1}^{\, -1} 
= v x (v')^{-1}
= \frac{ vxv^* }{|v|^2}
\end{align*}
since $v^* = v_m \cdots v_1$ so that $\overline{v^*} = \overline{(v_m \cdots v_1)} = \overline{v_1} \; \overline{v_2} \; \cdots \; \overline{v_m}$ and $(\overline{v^*})^{-1} = (v')^{-1} = \overline{v_m}^{\,-1}\cdots\overline{v_1}^{\,-1}$.

Similarly, the action $x \mapsto vxv$ of $\$\R^{q+1,p}$ on $\$\R^{q+1,p}$ extends to an action of $v \in \$\Gamma_{q+1,p}$ on $x \in \$\R^{q+1,p}$ given by $x \mapsto vxv^*$, which is an orientation-preserving isometry of $\$\R^{q+1,p}$ composed with a dilation by $v\overline{v}$.

When $p=0$, i.e. $q=n$, and we begin from a negative definite quadratic form and the Clifford algebra $\Cl(\R^{0,n})$, then the paravectors have signature $(n+1,0)$ and so are written as $\$\R^{n+1}$. The paravector Lipschitz group of $\Cl(\R^{0,n})$ is $\$\Gamma_{n+1,0}$, which we simply write as $\$\Gamma_{n+1}$. Thus $\$\Gamma_n$ is the paravector Lipschitz group of $\Cl(\R^{0,n-1})$. 
In fact $\$\Gamma_{n} \cong \Gamma_{n}^+$ \cite[sec. 19.3]{Lounesto_Clifford_book_01}. 

There is a corresponding set of Clifford matrices, given (e.g. \cite{Lounesto_Clifford_book_01, Vahlen_1902}) by
\begin{equation}
\label{Eqn:Vahlen_general}
\begin{pmatrix} a & b \\ c & d \end{pmatrix}, \quad
a,b,c,d \in \$\Gamma_n \cup \{0\}, \quad
\overline{a} \, b, \; b \, \overline{d}, \; \overline{d} \, c, \; c \,  \overline{a} \in \$\R^n, \quad
a d^* - b c^* = 1.
\end{equation}
These matrices form a group, which we denote $SL_2 \$\Gamma_n$. Again numerous other terminologies exist; our choice of terminology and notation is as in \refsec{Mobius_hyperbolic}, with the addition of a $\$$ for paravectors.

Numerous other equivalent definitions of the group are possible and exist in the literature; we discuss these in \refsec{Clifford_conditions}. But we observe that the defining condition of spinors $\xi \overline{\eta} \in \$\R^3$ arises naturally.

The Clifford matrices $SL_2 \$\Gamma_n$ act on $\Cl(\R^{0,n-1}) \cup \{\infty\}$ preserving the paravectors $\$\R^n \cup \{\infty\}$, by
\[
\begin{pmatrix} a & b \\ c & d \end{pmatrix} \cdot v
= (av+b)(cv+d)^{-1}
\quad \text{for } v \in \$\R^n \cup \{\infty\}.
\]
The group $SL_2 \$ \Gamma_n$ forms a 2-fold cover of the group of M\"{o}bius transformations of $\$\R^n \cup \{\infty\}$. Similarly, $SL_2 \$ \Gamma_{n-1}$ acts on $\$\R^{n-1} \cup \{\infty\}$. Via the natural inclusions $\Cl(\R^{0,n-2}) \subset \Cl(\R^{0,n-1})$ and $\$\R^{n-1} \subset \$\R^n$ and $\$\Gamma_{n-1} \subset \$\Gamma_n$ then $SL_2\$\Gamma_{n-1}$ acts on $\Cl(\R^{0,n-1})$ preserving $\$\R^{n-1} \cup \{\infty\}$ and $\$\R^n \cup \{\infty\}$. Being orientation-preserving, the action of $SL_2 \$\Gamma_{n-1}$ preserves the upper half space $\hyp^n$, where it acts via orientation-preserving isometries, yielding a 2-fold cover of the orientation-preserving isometry group of $\hyp^{n}$ in the upper half space model. In other words, $SL_2\$\Gamma_{n-1} \cong \Isom^S \hyp^n$ and $PSL_2\$\Gamma_{n-1} \cong \Isom^+ \hyp^n$

\subsection{Complex numbers}
\label{Sec:complex}

As a preliminary example, consider $(p,q) = (0,1)$. We obtain $\Cl(\R^{0,1}) \cong \C$, with $e_1 = i$. Under this isomorphism, the real numbers in $\C$ are scalars in $\Cl(\R^{0,1})$, and pure imaginary numbers are vectors. The conjugation $x \mapsto x^*$ is trivial, and the other conjugations are the standard complex number conjugation, $\overline{x} = x'$. The paravectors naturally have signature $(2,0)$ and form $\$\R^2$, corresponding to the entire space $\C$. The Lipschitz group $\Gamma_{0,1}$ is generated by nonzero imaginary complex numbers. The odd elements of $\Gamma_{0,1}$ are nonzero imaginary elements, and the even elements, i.e. the elements of $\Gamma_{0,1}^+$, are nonzero real numbers. The action of every element of $\Gamma_{0,1}$ on $\R^{0,1}$ is by negation (reflection of $\R^{0,1}$), as the algebra is commutative. The set of invertible paravectors is $\C \setminus \{0\}$, so the paravector Lipschitz group of $\Cl(\R^{0,1})$, which is $\$\Gamma_{2}$, is isomorphic to $\C^\times$. The action of $v = re^{\theta i} \in \$\Gamma_2$ (where $r>0$ and $\theta \in \R$ as in the usual polar form of a nonzero complex number) on the space of paravectors $\C$ given by $x \mapsto vx(v')^{-1}$ sends $x$ to $x e^{2 \theta i}$. Thus $v$ acts by rotation by $2\theta$. Similarly the action $x \mapsto vxv^*$ acts as multiplication by $v^2 = r^2 e^{2\theta i}$, i.e. rotation by $2\theta$ and dilation by $r^2$ from the origin. The group $SL_2 \$\Gamma_2$ is then just the usual $SL_2 \C$, which acts on $\$\R^2 \cup \{\infty\} = \C \cup \{\infty\}$ by M\"{o}bius transformations in the usual way, forming a 2-fold cover of the group of M\"{o}bius transformations $PSL_2 \C \cong \Isom^+ \hyp^3$.

\subsection{Quaternions and their basic properties}
\label{Sec:quaternion_involution}
\label{Sec:quaternions_inner_product}

For the rest of this paper, we specialise to the case $(p,q) = (0,2)$, and $\Cl(\R^{0,2}) \cong \HH$.

We begin by recalling some basic facts about quaternions. We obtain , with $e_1 = i$, $e_2 = j$, and $e_1 e_2 = k$, and for the remainder of this paper we identify these two spaces. Under this isomorphism, the real numbers in $\HH$ are scalars, the real linear combinations of $i$ and $j$ are vectors, the paravectors are as in \refdef{quaternionic_spinor}, and real multiples of $k$ are bivectors.

For the remainder of this subsection we write a general quaternion $q \in \HH$ as $q = a+bi+cj+dk$ where $a,b,c,d \in \R$.
The three Clifford algebra conjugations apply as follows.
\begin{defn}
\label{Def:involutions}
For a quaternion $q = a+bi+cj+dk$, 
\begin{enumerate}
\item $q \mapsto q'$ is the homomorphic involution given by $q' = a-bi-cj+dk$
\item $q \mapsto \overline{q}$ is the anti-homomorphic involution given by $\overline{q} = a-bi-cj-dk$
\item $q \mapsto q^*$ is the anti-homomorphic involution given by $q^* = a+bi+cj-dk$.
\end{enumerate}
\end{defn}
The following fact is easily verified.
\begin{lem}
\label{Lem:conjugation_combination}
For any $q \in \HH$, 
\[
q  + \overline{q} - q^* - q' = 0.
\]
\qed
\end{lem}

A quaternion $q$ has \emph{norm} $|q| \geq 0$ given by $|q|^2 = q \overline{q} = a^2 + b^2 + c^2 + d^2$, which extends to all of $\HH$ the quadratic form $\$ Q$ which exists on paravectors in general.
The related bilinear form $\langle \cdot, \cdot \rangle_\$ $, which we henceforth simply denote $\langle \cdot, \cdot \rangle$, also extends to a map $\HH \times \HH \to \R$, given by $\langle v, w \rangle =  \frac{1}{2}(v \overline{w} + w \overline{v}) = \Re(v \overline{w}) = \Re (w \overline{v})$.
If $q \neq 0$ then $q^{-1} = \overline{q}/|q|^2$. A \emph{unit quaternion} is a $q$ such that $|q| = 1$. 

An \emph{imaginary} quaternion is a real linear combination of $i,j$ and $k$. We denote the imaginary quaternions by $\II$, so as real vector spaces $\HH = \R \oplus \II$ and $\II = \R i \oplus \R j \oplus \R k$.
By the \emph{real part} and \emph{imaginary part} of $q$ we mean $a$ and $bi+cj+dk$ respectively. 

For any non-real quaternion $u$, the real span of $1$ and $u$ is algebraically isomorphic to $\C$. Indeed, if $u$ is unit imaginary then $u^2 = -u\overline{u} = -|u|^2 = -1$ so this isomorphism can be chosen to take $u$ to $i$. We then have $e^{u \theta} = \cos \theta + u \sin \theta$ for real $\theta$.

A quaternion $q$ can be written in \emph{polar} form as $q = r e^{\theta u}$ where $r, \theta \in \R$, $r = |q| \geq 0$, and $u$ is unit imaginary. Indeed, $re^{u \theta} = r\cos \theta + r u \sin \theta$ so $u$ is a real multiple of $\Im q$.
In polar form the conjugations are given by $(re^{u \theta})' = re^{u' \theta}$, $\overline{re^{u \theta}} = r e^{\overline{u}\, \theta} = r e^{-u\theta}$, and $(re^{u \theta})^* = re^{u^* \theta}$.

\subsection{Quaternion paravectors and Lipschitz group}
\label{Sec:quaternion_para}
We collect some straightforward facts about quaternion paravectors.

The paravectors in $\HH$ have signature $(q+1,p) = (3,0)$ so are written as $\$\R^3$ and agree with \refdef{quaternionic_spinor}(i).

The following facts are used without further comment throughout this paper.
\begin{lem}
Let $q \in \HH$.
\begin{enumerate}
\item
$q \in \$\R^3$ iff $q = q^*$.
\item
$q \in \$\R^3$ iff $qk \in \II$, and $q \in \II$ iff $qk \in \$\R^3$.
\item
$qq^* \in \$\R^3$. 
\end{enumerate}
\end{lem}
\begin{proof}
Straightforward checks.
\end{proof}
The converse of (iii) is true: any paravector can be expressed as $qq^*$. In fact, the following stronger statement is true. 
\begin{lem}
\label{Lem:paravector_square_root}
For any paravector $v$, there exists a paravector $w$ such that $v = ww^* = w^2$.
\end{lem}

\begin{proof}
Let $v = x+yu$ where $x,y \in \R$, $u \in \II \cap \$\R^3$ and $|u| = 1$. Let $a + b i \in \C$ be a complex square root of $x+yi$, where $a,b \in \R$. Then $w = a+b u$ has the desired properties.
\end{proof} 

Since all nonzero elements of $\HH$ are invertible, the invertible paravectors are $\$\R^3 \setminus \{0\}$. The paravector Lipschitz group $\$\Gamma_3$ is the multiplicative subgroup of $\HH$ generated by $\$\R^3 \setminus \{0\}$.
\begin{lem}
$\$\Gamma_{3} = \HH^\times$.
\end{lem}
Here and throughout, $\HH^\times$ denotes the multiplicative subgroup of $\HH$, i.e. $\HH \setminus \{0\}$.

\begin{proof}
By multiplying paravectors by $i,j,ij = k \in \$\Gamma_3$ we find that $\$\Gamma_3$ contains all nonzero $a+bi+cj+dk$ such that $a,b,c,d \in \R$ and at least one of $a,b,c,d$ is zero. Now any $a+bi+cj+dk$ with $a,b,c,d$ all nonzero can be expressed as $(1 + \frac{ab-cd}{a^2+d^2} i + \frac{ac+bd}{a^2+d^2} j)(a+dk)$.
\end{proof}

\subsection{Dot and cross products with paravectors}
\label{Sec:dot_cross_paravector}

It is common to identify $\R^3$ with the imaginary quaternions, $(x,y,z) \leftrightarrow xi+yj+zk$, and scalars with the real quaternions, under which we have the standard equations for the standard dot and cross product of $v,w \in \R^3$:
\[
vw = v \cdot w + v \times w, \quad
v \cdot w = \frac{vw + wv}{2} = \Re (vw) = \Re (wv), \quad
v \times w = \frac{vw - wv}{2} = \Im(vw) = \Im (-wv).
\]
The observation $\Re(vw) = \Re(vw)$ is true for any $v,w \in \HH$, and implies $vw - wv$ is imaginary, so we note the following fact which will be useful in the sequel:
\begin{equation}
\label{Eqn:commutator_identity}
ab - ba + \overline{(ab - ba)} 
= ab - ba + \overline{b} \, \overline{a} - \overline{a} \overline{b}
= 0, \quad
\text{for any } a,b \in \HH.
\end{equation}

Although identifying $(x,y,z)$ with $xi+yj+zk$ is common, we will identify $\R^3$ with the paravectors in a standard way: via
\begin{equation}
\label{Eqn:R3_R3}
\R^3 \cong \$\R^3, 
\quad (x,y,z) = x+yi+zj.
\end{equation}
The space $\$\R^3$ then inherits an orientation from the standard orientation on $\R^3$, with $(1,i,j)$ forming an oriented basis.
We obtain the following equations for dot and cross products, for $v,w \in \$\R^3$:
\begin{gather}
\label{Eqn:quaternion_geometry_1}
v \overline{w} = v \cdot w - (v \times w) k, \quad
v \cdot w = \frac{v \overline{w} + w \overline{v}}{2} = \Re ( v \overline{w} ) = \Re (w \overline{v} ), \\
\label{Eqn:quaternion_geometry_2}
v \times w= \frac{(v \overline{w} - w \overline{v}) k}{2}, \quad
(v \times w) k = \Im( - v \overline{w} ) = \Im ( w \overline{v} ).
\end{gather}
In particular, the dot product on $\R^3 \cong \$\R^3$ is the restriction of the inner product $\langle \cdot, \cdot \rangle $ to $\$\R^3$, and the Euclidean norm on $\R^3 \cong \$\R^3$ is the restriction of the norm on $\HH$.

\subsection{Rotations with paravectors}
\label{Sec:para_rotation}

The rotations produced by paravectors, discussed in \refsec{paravector_Mobius} above, are standard; for example Ahlfors in \cite{Ahlfors_Mobius85} describes them explicitly. However we need the explicit formulation for quaternions of \reflem{paravector_rotation} and \refprop{rho_rotation_dilation}, and we give a self-contained proof.

\begin{lem}
\label{Lem:quaternion_geometry_facts}
Let $u \in \II$, so that $-uk \in \$\R^3$.
Let $x \in \$\R^3$.
\begin{enumerate}
\item
\label{Lem:paravector_commutation}
$uk = -ku^*$.
\item
$xk = k\overline{x}$.
\item
$(-uk) \cdot x = 0$ if and only if $ux=xu^*$. In this case, $(-uk) \times x = ux = xu^*$.
\end{enumerate}
\end{lem}

\begin{proof}
The first two statements are straightforward checks.
For (iii), using \refeqn{quaternion_geometry_1}, (i), (ii), and $\overline{u^*} = -u^*$ (since $u^*$ is imaginary) we have 
\[
2 (-uk) \cdot x 
= -uk \overline{x} + x \overline{(-uk)} 
= -uxk + x \overline{ku^*} 
= \left( -ux + x u^* \right) k.
\]
We can then calculate the cross product using \refeqn{quaternion_geometry_1}, (i), (ii), and $\overline{u^*} = -u^*$ as
\[
2 (-uk) \times x 
= \left( (-uk) \overline{x} - x \overline{(-uk)} \right) k
= \left( -uxk - x \overline{(ku^*)} \right) k
= \left( -ux - x u^* \right) k^2
= ux+xu^*
\]
and the statement follows from noting $ux=xu^*$.
\end{proof}

\begin{lem}
\label{Lem:paravector_rotation}
Let $u \in \II$, so that $-uk \in \$\R^3$, and $x \in \$\R^3$. Suppose $(-uk) \cdot x = 0$ and further $|u| = 1$. Then rotation of $\$\R^3$ by angle $\theta$ about $-uk$ sends $x$ to
\[
e^{u \theta} x = x e^{u^* \theta} = e^{u \theta/2} x e^{u^* \theta/2}.
\]
\end{lem}

\begin{proof}
Using the cross product and \reflem{quaternion_geometry_facts}(iii), the stated rotation sends $x \in \$\R^3$ to
\[
x \cos \theta + (-uk) \times x \sin \theta
= x \cos \theta + ux \sin \theta
= \left( \cos \theta + u \sin \theta \right) x = e^{u \theta} x,
\]
the last equality holding since $u$ is imaginary and $|u| = 1$. 
Since $ux=xu^*$, the above expression equals
\[
x \cos \theta + x u^* \sin \theta = x \left( \cos \theta + u^* \sin \theta \right) = x e^{u^* \theta},
\]
since $u*$ is also unit imaginary. 
Since $e^{u\theta} x = x e^{u^* \theta}$ for all real $\theta$, we also have
\[
e^{u \theta} x = e^{u \theta/2} e^{u \theta/2} x
= e^{u \theta/2} x e^{u^* \theta/2}.
\]
\end{proof}

\subsection{Actions of quaternions and paravectors}
\label{Sec:actions_on_paravectors}

The action of $v \in \$\Gamma_3$ on paravectors $\$\R^3$ by $x \mapsto v x (v')^{-1}$, often denoted $\rho$ in the literature \cite{Ahlfors_Lounesto_89, Ahlfors_Mobius85, Ahlfors_fixedpoints_85, Waterman_93}, yields rotations.
We use $\rho$ for other purposes (e.g. \refsec{intro_lambda}), and are interested in the slightly different action $x \mapsto vxv^*$, so we notate it by $\sigma$; in fact it extends to an action of $\HH$ on $\HH$ as follows. 
\begin{defn}
\label{Def:rho}
We define an action of $\HH$ into linear endomorphisms on $\HH$ as follows.
\[
\sigma \colon \HH \To \End(\HH), \quad
\sigma(v) (x) = vxv^*.
\]
\end{defn}

Throughout this paper we will be considering rotations in 2-planes $\Pi$ in various ambient dimensions. We can describe the sense of such rotations by giving two vectors $v,w$ spanning the 2-plane $\Pi$ of the rotation, and saying that we rotate from $v$ \emph{towards} $w$ by some angle. This means that, for small angles, $v$ rotates into the half-2-plane $\Pi \setminus \R v$ containing $w$.

Clearly, when $v \in \R$, $\sigma(v)$ is multiplication by $v^2$, so it suffices to consider non-real $\sigma(v)$. 
\begin{prop}
\label{Prop:rho_rotation_dilation}
\label{Prop:rho_rotation}
Let $v \in \HH \setminus \R$ be given in polar form as $v = r e^{\theta u}$, so $\sigma(v)$ is a linear endomorphism of $\HH$. Then $\sigma(v)$ respects the splitting $\HH = \$\R^3 \oplus \R k$, and acts on each summand as follows.
\begin{enumerate}
\item
The restriction of $\sigma(v)$ to $\$\R^3$, identified with $\R^3$ by \refeqn{R3_R3}, is a rotation of angle $2\theta$ about the axis $-uk$, composed with multiplication by $|v|^2$.  
\item
If $v \in \$\R^3$ then the rotation described above is also a rotation of angle $2\theta$ in the 2-plane $\Pi$ spanned by $1$ and $v$, from $1$ towards $v$.
\item
The restriction of $\sigma(v)$ to $\R k$ is multiplication by $|v|^2$.
\end{enumerate}
\end{prop}

Since the action $x \mapsto vx(v')^{-1}$ is obtained from $\sigma$ by post-multiplication by $|v|^{-2}$, under this action the splitting $\$\R^3 \oplus \R k$ is again preserved, $v$ acts to preserve $\R k$ pointwise, and $v$ acts on $\$\R^3$ by the rotation described above. In particular, the action is by orientation-preserving isometries.

\begin{proof}
Since $\sigma(v)$ is linear, it suffices to prove (i) for $x \in \$\R^3$ such that $x = -uk$, and for $x \in \$\R^3$ which are perpendicular to $-uk$.

If $x = -uk$ then we have $uk = -ku^*$ (\reflem{quaternion_geometry_facts}(i)), so 
\begin{align*}
vxv^* 
&= \left( r e^{u \theta} \right) (-uk) \left( r e^{u^* \theta} \right)
 = - r^2 \left( \cos \theta + u \sin \theta \right) u k \left( \cos \theta + u^* \sin \theta \right) \\
&= - r^2 u \left( \cos \theta + u \sin \theta \right) \left( \cos \theta - u \sin \theta \right) k 
= - r^2 uk = r^2 x = |v|^2 x.
\end{align*}

If $x \in \$\R^3$ is perpendicular to $-uk$ then 
$vxv^* 
= r^2 e^{u \theta} x e^{u^* \theta}$
by \reflem{quaternion_geometry_facts}(iv) is $|v|^2$ times the rotation of $x$ by $2\theta$ about $-uk$.

If $v \in \$\R^3$, then $v = re^{\theta u}$ where $u$ is a unit imaginary paravector, and 
we may take $0 < \theta < \pi$, so that $\sin \theta > 0$ and $\Im(v)$ is a positive multiple of $u$. Then a rotation in $\Pi$ from $1$ towards $v$ is a rotation about the axis $1 \times v$, which by \refeqn{quaternion_geometry_2} is $\frac{1}{2}(\overline{v}-v)k = - \Im(v)k$, which is a positive multiple of $-uk$.

To see (iii), note that if $v \in \$\R^3$ then we have $v^* = v$ and by \reflem{quaternion_geometry_facts}(i) $vk=k\overline{v}$ so $\sigma(v)(k) = vkv = k\overline{v} v = k |v|^2$. Successively multiplying by other paravectors we obtain $\sigma(v)(k) = k |v|^2$ for all $v \in \$\Gamma_3 = \HH \setminus \{0\}$.
\end{proof}

The unit quaternions form a subspace of $\HH^3$ diffeomorphic to $S^3$, which is a multiplicative subgroup; we denote this group simply by $S^3$. Restricting $\sigma$ to the action of $S^3$ on $\$\R^3$, the action is by rotations, and we obtain a homomorphism $S^3 \To SO(3)$. It is clear from the description above that this homomorphism is surjective. 
Two elements $e^{u_1 \theta_1}, e^{u_2 \theta_2}$ describe the same rotation iff $- u_1 k = - u_2 k$ and $2\theta_1 \equiv 2\theta_2$ mod $2\pi$, or $-u_1 k = u_2 k$ and $2\theta_1 = - 2\theta_2$ mod $2\pi$. This occurs iff  $e^{u_1 \theta_1} = \pm e^{u_2 \theta_2}$. We obtain the 
standard 2--1 covering map $S^3 \To SO(3)$ and we regard $S^3$ as $\Spin(3)$.

\section{Geometry of spinors}
\label{Sec:spinors}

\subsection{Inner product and norm on $\HH^2$}
\label{Sec:inner_product_norm_H2}

We now consider pairs of quaternions $\kappa = (\xi, \eta)$, i.e. elements of $\HH^2$. We begin by introducing a standard real-valued inner product and norm.

If $\kappa_m = (\xi_m, \eta_m)$, $m=1,2$, then we define
\[
\langle \cdot, \cdot \rangle \colon \HH^2 \times \HH^2 \To \R, \quad
\langle \kappa_1, \kappa_2 \rangle = \frac{\xi_1 \overline{\xi_2} + \xi_2 \overline{\xi_1} + \eta_1 \overline{\eta_2} + \eta_2 \overline{\eta_1}}{2} = \Re \left( \xi_1 \overline{\xi_2} + \eta_1 \overline{\eta_2} \right),
\]
which is real-bilinear, nondegenerate, symmetric, and positive definite, with the following properties.
\begin{lem}
Let $\kappa, \kappa_1, \kappa_2 \in \HH^2$ and $x,y \in \HH$. Then we have the following.
\label{Lem:quaternion_inner_product_mult}
\label{Lem:imaginary_multiplication_orthogonal}
\begin{enumerate}
\item
$\langle \kappa_1 x, \kappa_2 x \rangle = \langle x \kappa_1, x \kappa_2 \rangle = |x|^2 \langle \kappa_1, \kappa_2 \rangle$.
\item
$\langle \kappa x, \kappa y \rangle = |\kappa|^2 \langle x,y \rangle$.
\item
If $x \in \HH$ is imaginary, then $\langle \kappa, \kappa x \rangle = \langle \kappa x, \kappa \rangle = 0$.
\end{enumerate}
\end{lem}
Note in (ii) that the first inner product is for $\HH^2$, the second for $\HH$, as in \refsec{quaternions_inner_product}.

\begin{proof}
For (i), we observe $\langle x \kappa_1, x \kappa_2 \rangle = x \langle \kappa_1, \kappa_2 \rangle \overline{x} = |x|^2 \langle \kappa_1, \kappa_2 \rangle$. A direct calculation also shows $\langle \kappa_1 x, \kappa_2 x \rangle = |x|^2 \langle \kappa_1, \kappa_2 \rangle$. For (ii), let $\kappa = (\xi, \eta)$. Noting that $\langle x, y \rangle = \frac{1}{2}( x \overline{y} + y \overline{x})$ is real, we obtain
\begin{align*}
\langle \kappa x, \kappa y \rangle
&= \frac{\xi x \, \overline{y} \, \overline{\xi} + \xi y \, \overline{x} \, \overline{\xi} + \eta x \, \overline{y} \, \overline{\eta} + \eta y \, \overline{x} \, \overline{\eta}}{2} 
= \frac{ \xi \left( x \overline{y} + y \overline{x} \right) \overline{\xi} + \eta \left( x \overline{y} + y \overline{x} \right) \overline{\eta}}{2} \\
&= \langle x, y \rangle \left( \xi \overline{\xi} + \eta \overline{\eta} \right) 
= \langle x, y \rangle \; |\kappa|^2.
\end{align*}

The final statement then follows from (ii) upon noting that $\langle 1, x \rangle = x + \overline{x} = 0$ when $x$ is imaginary.
\end{proof}

The associated norm is given, for $\kappa = (\xi, \eta) \in \HH^2$, by
\begin{equation}
\label{Eqn:H2_norm}
| \cdot | \colon \HH^2 \To \R_{\geq 0}, \quad
|\kappa|^2 = \langle \kappa, \kappa \rangle = |\xi|^2 + |\eta|^2.
\end{equation}
We will sometimes identify $\HH^2$ with $\R^8$ in a standard way,
\begin{equation}
\label{Eqn:H2_R8}
\HH^2 \cong \R^8, \quad
(\xi,\eta) = (x_0 + x_1 i + x_2 j + x_3 k, y_0 + y_1 i + y_2 j + y_3 k) \leftrightarrow (x_0, x_1, x_2, x_3, y_0, y_1, y_2, y_3).
\end{equation}
Then $\langle \cdot, \cdot \rangle$ is the standard dot product, and the norm is the standard Euclidean norm. In particular, the set of unit $\kappa \in \HH^2$, i.e. $\kappa = (\xi, \eta)$ with $|\kappa|^2 = |\xi|^2 + |\eta|^2 = 1$, is the standard $S^7 \subset \R^8$.

\subsection{Conditions for spinors}
\label{Sec:spinor_conditions}

We have defined the space $S\HH$ of quaternionic spinors in \refdef{quaternionic_spinor}. The following lemma allows equivalent characterisations, and we use it repeatedly throughout; it includes lemma 1.4 of \cite{Ahlfors_Mobius85}.

\begin{lem}
\label{Lem:spinor_condition}
Let $\xi = x_0 + x_1 i + x_2 j + x_3 k$, $\eta = y_0 + y_1 i + y_2 j + y_3 k$ where all $x_m, y_m \in \R$. The following are equivalent.
\begin{enumerate}
\item $\xi \overline{\eta} \in \$\R^3$
\item $\xi^* \eta \in \$\R^3$
\item $x_0 y_3 + x_1 y_2 - x_2 y_1 - x_3 y_0 = 0$.
\end{enumerate}
\end{lem}

\begin{proof}
The rotation $\sigma(\eta^*/|\eta|)$ takes $\xi \overline{\eta}$ to $\eta^* \xi \overline{\eta} \eta / |\eta|^2 = \xi^* \eta$. Since $\sigma$ preserves $\$\R^3$, each of $\xi \overline{\eta}, \xi^* \eta$ lies in $\$\R^3$ simultaneously with the other. Alternately, and for the equivalence of (iii), we compute that the coefficient of $k$ in both $-\xi \overline{\eta}$ and $\xi^* \eta$ is $x_0 y_3 + x_1 y_2 - x_2 y_1 - x_3 y_0$.
\end{proof}

Since $\$\R^3$ is mapped to itself bijectively under the operations $v \mapsto \overline{v}$ and $ \mapsto v^*$, as well as their composition $v \mapsto v'$, the first two conditions above are also equivalent to any one of
\begin{equation}
\label{Eqn:spinor_conditions}
\overline{(\xi \overline{\eta})} = \eta \overline{\xi}, \quad
\left( \xi \overline{\eta} \right)^* = \eta' \xi^*, \quad
\left( \xi \overline{\eta} \right)' = \xi' \eta^*, \quad
\overline{\xi^* \eta} = \overline{\eta} \xi', \quad
\left( \xi^* \eta \right)^* = \eta^* \xi, \quad
\left( \xi^* \eta \right)' = \overline{\xi} \eta'
\end{equation}
lying in $\$\R^3$. Since $\eta^{-1} = \overline{\eta}/|\eta|^2$ is a real multiple of $\overline{\eta}$, when $\eta \neq 0$, we can replace $\overline{\eta}$ with $\eta^{-1}$ in any of the above. So $(\xi, \eta) \in S\HH$ if and only if any one of the above conditions (hence all of them)  are satisfied.

Also, spinors are preserved under right multiplication and swapping coordinates.
\begin{lem}
\label{Lem:spinor_right_multiplication}
Let $\kappa \in S\HH$. Then we have the following.
\begin{enumerate}
\item
If $x \in \HH^\times$ then $\kappa x \in S\HH$.
\item
$(\eta, \xi) \in S\HH$.
\end{enumerate}
\end{lem}

\begin{proof}
We have $\kappa \neq 0$ and $\xi \overline{\eta}, \xi^* \eta \in \$\R^3$. Then $(\xi x) \overline{(\eta x)} = \xi x \, \overline{x} \overline{\eta} = \xi \overline{\eta} |x|^2 \in \$\R^3$, and $\kappa x $ is nonzero since $x \neq 0$, so $\kappa x \in S\HH$. Similarly, $(\eta, \xi)$ is nonzero, and $\eta^* \xi = (\xi^* \eta)^* \in \$\R^3$, so $(\eta, \xi) \in S\HH$.
\end{proof}

\subsection{Bracket on $\HH^2$}
\label{Sec:bracket}

We now define another bilinear form on $\HH^2$. Unlike $\langle \cdot, \cdot \rangle$, this bilinear form $\{ \cdot, \cdot \}$ is quaternion-valued, and is not symmetric. We call it the \emph{bracket}, as seen in \refeqn{lambda_pdet}: for $\kappa_1, \kappa_2 \in \HH^2$, $\kappa_m = (\xi_m, \eta_m)$, denoting by $(\kappa_1, \kappa_2)$ the $2 \times 2$ matrix with $\kappa_1, \kappa_2$ as columns, we have
\[
\{ \cdot, \cdot \} \colon \HH^2 \times \HH^2 \To \HH, \quad
\{\kappa_1, \kappa_2\} = \pdet(\kappa_1, \kappa_2) 
= \xi_1^* \eta_2 - \eta_1^* \xi_2.
\]
This bilinear form is real-bilinear, nondegenerate, and satisfies the antisymmetry condition 
\[
\{\kappa_2, \kappa_1\} = - \{\kappa_1, \kappa_2\}^*.
\]

We now demonstrate several properties of the bracket.
\begin{lem}
\label{Lem:bracket_properties}
Let $\kappa \in \HH^2$.
\begin{enumerate}
\item 
$\{\kappa, \kappa \} \in \R k$.
\item
$\{ \kappa, \kappa \} = 0$ if and only if $\kappa \in S\HH \cup \{0\}$.
\end{enumerate}
\end{lem}

\begin{proof}
Letting $\kappa = (\xi, \eta)$ we have $\{\kappa, \kappa\} 
= \xi^* \eta - \eta^* \xi$,
which is sent to its negative under the $*$-involution, hence lies in $\R k$. It is zero precisely when $\xi^* \eta \in \$ \R^3$, i.e. $\kappa \in S\HH$ or $\kappa = 0$.
\end{proof}

The bracket also has the multiplicativity property
\begin{equation}
\label{Eqn:bracket_multiplication}
\{\kappa_1 x_1, \kappa_2 x_2 \}
= x_1^* \; \{ \kappa_1, \kappa_2 \} \; x_2
\quad \text{for any } \kappa_1, \kappa_2 \in \HH^2 \text{ and } x_1, x_2 \in \HH.
\end{equation}

\begin{lem}
\label{Lem:nondegeneracy_of_spinor_form}
Let $\kappa = (\xi, \eta) \in S\HH$ and $\nu = (\alpha, \beta) \in \HH^2$. Then $\{\kappa, \nu \} = 0$ if and only if $\nu = \kappa x$ for some $x \in \HH$ (and hence $\nu = 0$ or, by \reflem{spinor_right_multiplication}, $\nu \in S\HH$).
\end{lem}

\begin{proof}
If $\nu = \kappa x$ then either $x = 0$ so that $\nu = 0$ and $\{\kappa, \nu\} = 0$ trivially, or $x \neq 0$ so $\nu \in S\HH$ by \reflem{spinor_right_multiplication}. Then $\{\kappa, \nu \} = \{\kappa, \kappa x \} 
= \{ \kappa, \kappa \} x
= 0$, using \refeqn{bracket_multiplication} and antisymmetry of the bracket.

For the converse, suppose 
\[
\{\kappa, \nu \} 
= \xi^* \beta - \eta^* \alpha
= 0.
\]
If $\nu = 0$ then we may take $x=0$, so we may assume $\nu \neq 0$. If $\beta = 0$ then we have 
$\eta^* \alpha = 0$,
but $\alpha \neq 0$ (since $\nu \neq 0$), so $\eta = 0$; then as $\kappa = (\xi, 0) \in S\HH$ we have $\xi \neq 0$; we then have $\nu = (\alpha, 0) = (\xi, 0) x = \kappa x$ where $x = \xi^{-1} \alpha$. Thus we may assume $\beta \neq 0$. Similarly, if $\eta = 0$ then we have 
$\xi^* \beta = 0$,
but $\beta \neq 0$ (from above) and $\xi \neq 0$ (as $\kappa \in S\HH$), a contradiction. Thus we may assume $\eta \neq 0$.

The equation 
$\xi^* \beta  = \eta^* \alpha$
is then equivalent to 
$(\xi \eta^{-1})^* = \alpha \beta^{-1}$.
As $\kappa \in S\HH$ then $\xi \eta^{-1} \in \$\R^3$, hence is invariant under $*$, so also $\alpha \beta^{-1} \in \$\R^3$ and we then have $\xi \eta^{-1} = \alpha \beta^{-1}$. Letting $x = \xi^{-1} \alpha$ we then have
$\alpha = \xi x$ and $\beta = \eta \xi^{-1} \alpha = \eta x$ so $\nu = (\alpha, \beta) = (\xi, \eta) x = \kappa x$.
\end{proof}

\subsection{Complementary elements}
\label{Sec:complementary}

Given $\kappa = (\xi, \eta) \in \HH^2$, it will be useful to introduce the following ``complementary" element of $\HH^2$. It is perhaps analogous to an almost complex structure in a K\"{a}hler structure, ``$\check{\kappa}$ is like $J\kappa$". Indeed the map $J$ sending $\kappa \mapsto \check{\kappa}$ satisfies $J^2 = -1$. Similarly, ``$\langle \cdot, \cdot \rangle$ is a metric" and ``$\{ \cdot, \cdot \}$ is like a symplectic form". See \refsec{paravectors_in_TSH} for further discussion and \refsec{deriv_phi1} for observations along similar lines.

\begin{defn}
For $\kappa = (\xi, \eta) \in \HH^2$, let $\check{\kappa} = (\eta', -\xi')$.
\end{defn}

Clearly $|\kappa| = |\check{\kappa}|$.

\begin{lem}
\label{Lem:complementary_spinor_facts}
Let $\kappa = (\xi, \eta) \in \HH^2$.
\begin{enumerate}
\item
If $\kappa \in S\HH$ then $\check{\kappa} \in S\HH$.
\item
$\langle \kappa, \check{\kappa} \rangle = 0$. In fact, for any $x,y \in \HH$ such that $x \overline{y} \in \$\R^3$, we have $\langle \kappa x, \check{\kappa} y \rangle = 0$.
\item
If $\kappa \in S\HH$ then $\langle \kappa x, \check{\kappa} y \rangle = 0$ for any $x,y \in \HH$.
\item
$\{ \kappa, \check{\kappa} \} = - |\kappa|^2$. More generally, for any $x,y \in \HH$ we have $\{ \kappa x, \check{\kappa} y \} 
= - x^* y |\kappa|^2$.
\end{enumerate}
\end{lem}

\begin{proof}
For (i), if $(\xi, \eta) \in S\HH$ then $\xi \overline{\eta} \in \$\R^3$. Then $(\eta')\overline{(-\xi')} = - \eta' \xi^* = (\xi \overline{\eta})^* \in \$\R^3$, so $\check{\kappa} \in S\HH$.

For (ii), noting that $x \overline{y} \in \$\R^3$ implies $(x \overline{y})^* = x\overline{y}$, i.e. $y' x^* = x \overline{y}$, and hence $\overline{x\overline{y}} = \overline{y' x^*}$, so 
$y \overline{x} = \overline{(y' x^*)} = x' y^*$, we compute 
\begin{align*}
\langle \kappa x, \check{\kappa} y \rangle
&= \langle (\xi, \eta)x , (\eta', -\xi') y \rangle 
= \frac{ \xi x \, \overline{y} \, \eta^* + \eta' y \, \overline{x} \, \overline{\xi} - \eta x \, \overline{y} \xi^* - \xi' y \, \overline{x} \, \overline{\eta} }{2} \\
&= \frac{ \xi x \, \overline{y} \, \eta^* + \eta' y \, \overline{x} \,  \overline{\xi} - \eta y' x^* \xi^* - \xi' x' y^* \, \overline{\eta}}{2} 
\end{align*}
which is zero, since equal to $(a + \overline{a} - a^* + a')/2$ where $a = \xi x \, \overline{y} \, \eta^*$, by \reflem{conjugation_combination}.

If $\kappa \in S\HH$ then we have $\xi \eta^{-1} = v \in \$\R^3$, so $\xi = v \eta$ where $v = v^*$ and $\overline{v} = v'$, and following the above calculation we obtain 
\begin{align*}
\langle \kappa x, \check{\kappa} y \rangle
&= \frac{v \eta x \, \overline{y} \, \eta^* + \eta' y \, \overline{x} \, \overline{\eta} \, \overline{v} - \eta x \, \overline{y} \, \eta^* v^* - v' \eta' y \, \overline{x} \, \overline{\eta}}{2} \\
&= \frac{v \eta x \, \overline{y} \, \eta^* + \eta' y \, \overline{x} \, \overline{\eta} \, \overline{v} - \eta x \, \overline{y} \, \eta^* v - \overline{v} \, \eta' y \, \overline{x} \, \overline{\eta}}{2} \\
&= \frac{ab + \overline{b} \, \overline{a} - b a - \overline{a} \, \overline{b}}{2} 
\quad \text{where $a = v$ and $b = \eta x \, \overline{y} \eta^*$},
\end{align*}
which is zero by \refeqn{commutator_identity}, proving (iii). Finally, for $\kappa \in \HH^2$ we compute
\begin{align*}
\{ \kappa, \check{\kappa} \}
= \pdet \begin{pmatrix} \xi & \eta' \\ \eta & -\xi' \end{pmatrix}
= - \xi^* \xi' - \eta^* \eta'
= -|\xi|^2 - |\eta|^2  
= - |\kappa|^2,
\end{align*}
from which statement (iv) follows by \refeqn{bracket_multiplication}.
\end{proof}

We will need the following technical fact about $\kappa$ and $\check{\kappa}$ in the sequel.
\begin{lem}
\label{Lem:factorisation_fact}
Let $\kappa = (\xi, \eta) \in S\HH$, considered as a $2 \times 1$ column vector. Then there exists a $1 \times 2$ row vector $\tau = (\alpha, \beta) \in \HH^2$ such that precisely one of $\tau \kappa, \tau \check{\kappa}$ is zero.
\end{lem}

\begin{proof}
Suppose to the contrary that 
$\{ (\alpha, \beta) \in \HH^2 \mid \alpha \xi + \beta \eta = 0 \} = \{ (\alpha, \beta) \in \HH^2 \mid \alpha \eta' - \beta \xi' = 0 \}$. If $\xi = 0$ or $\eta = 0$ (we cannot have $\xi = \eta = 0$ since $\kappa \in S\HH$) then the two sets are clearly distinct, so we may assume $\xi, \eta$ are both nonzero. In this case, any $(\alpha, \beta) \neq (0,0)$ in either set has $\alpha, \beta$ both nonzero.
Other than $(0,0)$, the first set contains all $(\alpha, \beta)$ such that 
$\beta^{-1} \alpha = -\eta \xi^{-1}$, and the second set contains all $(\alpha, \beta)$ 
such that $\beta^{-1} \alpha = \xi' \eta'^{-1}$. 
So $- \eta \xi^{-1} = \xi' \eta'^{-1}$. 

Now let $v = \xi \eta^{-1}$. Then $v \neq 0$, and as $\kappa \in S\HH$ we have $v \in \$\R^3$. Then $\eta \xi^{-1} = v^{-1}$ and $\xi' \eta'^{-1} = v'$, so the equation $-\eta \xi^{-1} = \xi' \eta'^{-1}$ implies $-v^{-1}= v'$. Comparing norms yields $|v| = 1$, so $v^{-1} = \overline{v}$. Thus we obtain $-\overline{v} = v'$, and hence $v = -v^*$. But as $v \in \$\R^3$ then $v = v^*$, so $v=0$, a contradiction.
\end{proof}

\subsection{The space of spinors}
\label{Sec:space_of_spinors}

We now consider the space $S\HH$. Firstly we consider its topology. Note that as a smooth real manifold, $\HH^2 \cong \R^8$ and $S\HH$ is a subset of the non-compact manifold $\HH^2 \setminus \{0\} \cong S^7 \times \R$.
\begin{lem}
\label{Lem:topology_of_SH}
$S\HH$ is diffeomorphic to $S^3 \times S^3 \times \R$. In particular, $S\HH$ is orientable.
\end{lem}

\begin{proof}
The equation in \reflem{spinor_condition}(iii) can be written as $(x_0 + y_3)^2 - (x_0 - y_3)^2 + (x_1 + y_2)^2 - (x_1 - y_2)^2 - (x_2 + y_1)^2 + (x_2 - y_1)^2 - (x_3 + y_0)^2 + (x_3 - y_0)^2 = 0$, so changing variables to the expressions in brackets, the equation is the zero set of a standard quadratic Morse function $\R^8 \To \R$ of index 4. Hence its zero set (minus the origin), which is $S\HH$, is diffeomorphic to $S^3 \times S^3 \times \R$.
\end{proof}
Thus, $S\HH$ is a 7-(real-)dimensional subset of the 8-dimensional space $\HH^2 \setminus \{0\} \cong S^7 \times \R$, cut out by the real quadratic equation of \reflem{spinor_condition}(iii). 

A spinor $\kappa  = (\xi, \eta) \in S\HH$ by definition has $\xi \eta^{-1} \in \$\R^3 \cup \{\infty\}$. Thus there is a map $S\HH \To \$\R^3 \cup \{\infty\}$ given by $(\xi, \eta) \mapsto \xi \eta^{-1}$. We now describe the fibres of this map; they are diffeomorphic to $S^3 \times \R \cong \HH^\times$.

\begin{lem}
\label{Lem:division_fibre}
Let $z \in \$\R^3 \cup \{\infty\}$ and let $S\HH_z = S\HH \cap \{ (\xi, \eta) \mid \xi \eta^{-1} = z\}$. Then $S\HH_z$ is diffeomorphic to $S^3 \times \R$, and if $(\xi_0, \eta_0) \in S\HH_z$ then $S\HH_z = (\xi_0, \eta_0)\HH^\times$.
\end{lem}

\begin{proof}
Take $z \in \$\R^3 \cup \{\infty\}$. Note $S\HH_z$ is nonempty; for instance it contains $(z,1)$ when $z \in \$\R^3$, and $(1,0)$ when $z = \infty$. Let $\kappa_0 = (\xi_0, \eta_0)$ be an arbitrary element of $S\HH_z$. We will show $S\HH_z = \kappa_0 \HH^\times$.

For any $x \in \HH^\times$ we have $(\xi_0 x)(\eta_0 x)^{-1} = \xi_0 x x^{-1} \eta_0^{-1} = \xi_0 \eta_0^{-1} = z$. Thus $\kappa_0 \HH^\times \subseteq S\HH_z$. Conversely, suppose $(\xi, \eta) \in S\HH_z$. Then $\xi \eta^{-1} = \xi_0 \eta_0^{-1} = z$. If $z = \infty$ then we have $\eta = \eta_0 = 0$ and $\xi, \xi_0 \neq 0$, so letting $x = \xi_0^{-1} \xi \in \HH^\times$ we have $(\xi, \eta) = (\xi_0 x, 0) = (\xi_0, \eta_0)x \in \kappa_0 \HH^\times$. If $z \in \$\R^3$ then $\eta, \eta_0$ are nonzero; letting $x = \eta_0^{-1} \eta \in \HH^\times$, from $\xi \eta^{-1} = \xi_0 \eta_0^{-1} = z$ we have $\xi_0^{-1} \xi = \eta_0^{-1} \eta = x$, so $(\xi, \eta) = (\xi_0, \eta_0)x \in \kappa_0 \HH^\times$. Thus $S\HH_z \subseteq \kappa_0 \HH^\times$ and hence $S\HH_z = \kappa_0 \HH^\times$.
\end{proof}

\begin{lem}
\label{Lem:TSH}
Let $\kappa = (\xi, \eta) \in S\HH$. Then the tangent space to $S\HH$ is given by
\[
T_\kappa S\HH = \left\{ (\alpha, \beta) \in \HH^2 \mid \alpha \overline{\eta} + \xi \overline{\beta} \in \$\R^3 \right\}
\]
and decomposes into orthogonal real subspaces as
\[
T_\kappa S\HH = \kappa \, \HH \; \oplus \; \check{\kappa} \, \$ \R^3
= \kappa \R \oplus \kappa \II \oplus \check{\kappa} \$\R^3
\]
\end{lem}

\begin{proof}
The first statement is obtained by differentiating the defining equation $(\xi + t \alpha) \overline{(\eta + t \beta)} \in \$\R^3$ with respect to $t \in \R$, as $\$\R^3$ is a real vector subspace. 

Taking $(\alpha, \beta) = (\xi, \eta) x$ with $x \in \HH$, we have 
\[
\alpha \, \overline{\eta} + \xi \, \overline{\beta} = 
\xi x \, \overline{\eta} + \xi \, \overline{x} \, \overline{\eta}
= \xi (x + \overline{x}) \, \overline{\eta}
= 2 \Re(x) \xi \overline{\eta} \in \$\R^3
\]
Thus $\kappa \HH \subset T_\kappa S\HH$. Taking $(\alpha, \beta) = \check{\kappa} v = (\eta', -\xi')v$ where $v \in \$\R^3$, 
\[
\alpha \overline{\eta} + \xi \overline{\beta} 
= \eta' v \overline{\eta} - \xi \overline{v} \xi^*
= \sigma(\eta')(v) - \sigma(\xi)(\overline{v}) \in \$\R^3,
\]
since $\sigma$ preserves $\$\R^3$.

Thus $\kappa \HH$ and $\check{\kappa} \$\R^3$ are subspaces of $T_\kappa S\HH$, of real dimension 4 and 3 respectively. By \reflem{complementary_spinor_facts}(iii) they are orthogonal. Thus we have the first orthogonal direct sum. Since $\HH = \R \oplus \II$ we have the second direct sum decomposition, and by \reflem{imaginary_multiplication_orthogonal} it is also orthogonal.
\end{proof}

We note that it is also possible to show that $T_\kappa S\HH$ contains $(-\xi \overline{x}, \eta x)$ for all $x \in \HH$, $(v \eta, \overline{v} \xi)$ for all $v \in \$\R^3$, and $(k \xi, k \eta)$.

The summand $\check{\kappa} \$\R^3$ of $T_\kappa S\HH$ provides us with a copy of the paravectors $\$\R^3$ at each point of $S\HH$, and it is naturally oriented via the orientation on $\$\R^3$ discussed in \refsec{dot_cross_paravector}. Thus $(\check{\kappa}, \check{\kappa}i,\check{\kappa}j)$ forms an oriented basis of $\check{\kappa}\$\R^3$.
The other summand $\kappa \HH$ is the tangent space to the fibre $\kappa \HH^\times$ of \reflem{division_fibre}.

It follows from this lemma, combined with \reflem{complementary_spinor_facts}(iii) and \reflem{imaginary_multiplication_orthogonal} that $T_\kappa S\HH$ has an orthogonal basis given by
\[
\kappa, \; \kappa i, \; \kappa j, \; \kappa k, \;
\check{\kappa}, \; \check{\kappa} i, \; \check{\kappa} j
\]
and an orthonormal basis is given by dividing each of these 7 vectors by $|\kappa| = |\check{\kappa}|$. Indeed, in this way we obtain 7 orthonormal sections hence a trivialisation of the tangent bundle of $S\HH$.

If we add $\check{\kappa} k$ to these vectors, we obtain an orthogonal basis of $\HH^2$ (this follows from the same lemmas), which can be made orthonormal in the same way, so $\check{\kappa} k$ is normal to $T_\kappa S\HH$, and we have an orthogonal decomposition $\HH^2 = \kappa \, \HH \oplus \check{\kappa} \HH$. Indeed, given an element of $\HH^2$, we can explicitly express it as a quaternion combination of $\kappa$ and $\check{\kappa}$ as follows.

\begin{lem}
\label{Lem:explicit_decomposition_in_TSH}
Let $\kappa = (\xi, \eta) \in S\HH$ and $\nu = (\alpha, \beta) \in \HH^2$. Then there exist unique $x,y \in \HH$ such that
\[
\nu = \kappa x + \check{\kappa} y,
\quad \text{namely} \quad
x = \frac{ \overline{\xi} \, \alpha + \overline{\eta} \, \beta}{|\kappa|^2}, \quad
y = \frac{ \eta^* \alpha - \xi^* \beta }{|\kappa|^2}.
\]
\end{lem}

\begin{proof}
The orthogonality argument above shows that there exist unique $x$ and $y$; to see that they in fact are as given, we can check explicitly that $\alpha = \xi x + \eta' y$ and $\beta = \eta x - \xi' y$. For the first equation we have
\[
\xi x + \eta' y = \frac{ \xi \left( \overline{\xi} \, \alpha + \overline{\eta} \, \beta \right) + \eta' \left( \eta^* \alpha - \xi^* \beta \right)}{|\kappa|^2}
= \frac{|\xi|^2 \alpha + \xi \overline{\eta} \beta + |\eta|^2 \alpha - \eta' \xi^* \beta}{|\kappa|^2}
\]
As $\kappa \in S\HH$ we have $\xi \overline{\eta} \in \$\R^3$, hence $\xi \overline{\eta} = \left( \xi \overline{\eta} \right)^* = \eta' \xi^*$. So the above expression simplifies to $\alpha$ as desired. A similar calculation establishes the second equality.
\end{proof}

\subsection{Paravectors in the tangent space of spinors}
\label{Sec:paravectors_in_TSH}

The summand $\check{\kappa} \$\R^3$ of $T_\kappa S\HH$ in \reflem{TSH} provides us with a copy of the paravectors $\$\R^3$ naturally at hand in the tangent space to $S\HH$ at any point. We therefore define the following family of sections of $TS\HH$, parametrised by paravectors.

\begin{defn}
\label{Def:Z}
For any $v \in \$\R^3$, the section $s_v \colon S\HH \To TS\HH$ of $TS\HH$ is
\[
s_v (\kappa) = \check{\kappa} v.
\]
More generally, the family of such sections forms a map
\[
s \colon \$\R^3 \times S\HH \To TS\HH, \quad
s(v, \kappa) = s_v (\kappa) = \check{\kappa} v.
\]
\end{defn}

In particular, we have sections $s_1, s_i, s_j$ of $TS\HH$ are given by 
\[
s_1 (\kappa) = \check{\kappa}, \quad
s_i (\kappa) = \check{\kappa} i, \quad
s_j (\kappa) = \check{\kappa} j.
\]
These sections also behave nicely under the bracket: \reflem{complementary_spinor_facts}(iv) tells us that
\begin{equation}
\label{Eqn:inner_product_with_sv}
\left\{ \kappa, s_v \kappa \right\} 
= \left\{ \kappa, \check{\kappa} v \right\}
= - v \left( |\xi|^2 + |\eta|^2 \right)
= - v |\kappa|^2.
\end{equation}
Moreover, $s_v$ and $s$ are real-linear in $v$: for $v,w \in \$\R^3$ and $x \in \R$ we have
\[
s_{v+w} = s_v + s_w
\quad \text{and} \quad
s_{xv} = x s_v, \quad \text{i.e.} \quad
s(v+w, \cdot) = s(v, \cdot) + s(w, \cdot)
\quad \text{and} \quad
s(xv, \cdot) = x s(v, \cdot).
\]

We can write
\begin{equation}
\label{Eqn:J_eqn}
\check{\kappa} = J \kappa'
\quad \text{where} \quad
J = \begin{pmatrix} 0 & 1 \\ -1 & 0 \end{pmatrix},
\end{equation}
so that $s_v (\kappa) = J \kappa' v$.
For complex spinors, $\kappa' = \overline{\kappa}$ and $s_i$ reduces to the map $Z$ of \cite{Mathews_Spinors_horospheres}. (The $J$ there is $Ji$ here.) 

The copy of the paravectors inside the tangent space $T_\kappa S\HH$ at $\kappa$ is given by the map
\begin{equation}
\label{Eqn:conformal_paravectors_in_TSH}
s( \cdot, \kappa) \colon \$\R^3 \To \check{\kappa} \$\R^3 \subset T_\kappa S\HH,
\end{equation}
which is an $\R$-linear isomorphism. The paravectors $\$\R^3$, identified with $\R^3$ as in \refeqn{R3_R3}, have the standard dot product and Euclidean norm, which agrees with the standard inner product and norm on $\HH$, as discussed in \refsec{dot_cross_paravector}. As $\check{\kappa} \$\R^3$ is a real-linear subspace of $\HH^2$, it has an inner product and norm as in \refsec{inner_product_norm_H2}. 

In fact, the map $s(\cdot, \kappa)$ of \refeqn{conformal_paravectors_in_TSH} is conformal, and we can compute the scaling factor explicitly. By scaling factor, we mean the following.
\begin{defn}
\label{Def:scaling_factor}
Let $V,W$ be real vector spaces equipped with nondegenerate bilinear symmetric real-valued forms $\langle \cdot, \cdot \rangle_V, \langle \cdot, \cdot \rangle_W$.
Let $f \colon V \to W$ be a conformal linear map. The real constant $K$ such that for all $v_1, v_2 \in V$,
\[
\langle f(v_1), f(v_2) \rangle_W = K \langle v_1, v_2 \rangle_V,
\]
is the \emph{scaling factor} of $f$.
\end{defn}
Note this definition allows for bilinear forms of arbitrary signature; $f$ above scales lengths in the usual sense by $\sqrt{|K|}$. The bilinear forms above induce real-valued norms $Q_V, Q_W$ defined by $Q_V (v) = \langle v,v \rangle$ and $Q_W (w) = \langle w,w \rangle$, which may take negative values. By the polarisation identity, $f$ has scaling factor $K$ iff for all $v \in V$, $Q_W (f(v)) = K Q_V (v)$.

\begin{lem}
\label{Lem:paravector_first_conformal}
For any $\kappa \in \HH$, the map $s(\cdot, \kappa)$ of \refeqn{conformal_paravectors_in_TSH} is conformal with scaling factor $|\kappa|^2$.
\end{lem}
Since $s(\cdot, \kappa)$ sends $v \in \$\R^3$ to $s(v, \kappa) = s_v (\kappa) = \check{\kappa} v$, this means that for all $v, w \in \$\R^3$,
\[
\langle \check{\kappa} v, \check{\kappa} w \rangle = |\kappa|^2 \; v \cdot w.
\]

\begin{proof}
From \reflem{quaternion_inner_product_mult}(ii) we have
$\langle \check{\kappa} v, \check{\kappa} w \rangle
= \left| \check{\kappa} \right|^2 \; \langle v, w \rangle$. 
The result then follows from $|\check{\kappa}| = |\kappa|$ and $\langle v, w \rangle = v \cdot w$.
\end{proof}

\subsection{Conditions for Clifford matrices}
\label{Sec:Clifford_conditions}

We defined a group of Clifford matrices $SL_2\$$ in \refdef{SL2H}, with numerous conditions. This is a shortened notation for the group $SL_2\$\Gamma_3$, one of the family $SL_2\$\Gamma_n$ in \refeqn{Vahlen_general}. We now discuss these groups in detail, following arguments of \cite[sec. 2]{Ahlfors_Mobius85} .

Note that any matrix satisfying \refdef{SL2H} has spinor columns: for the first column, $a^* c \in \$\R^3$, and moreover $(a,c) \neq (0,0)$, since $a=c=0$ would imply 
$a^* d - c^* b = 0$;
a similar argument applies to the second column.

Therefore, consider quaternionic matrices $A$ satisfying the following minimal requirements: $A$ has spinor columns, and pseudo-determinant $1$, i.e.
\begin{equation}
\label{Eqn:minimal_matrix}
A = \begin{pmatrix} a & b \\ c & d \end{pmatrix}, \quad
\text{such that} \quad
(a,c), (b,d) \in S\HH
\quad \text{and} \quad
\pdet A = a^* d - c^* b = 1.
\end{equation}
We will show that such a matrix in fact satisfies the maximal requirements of \refdef{SL2H}.

\begin{lem}
\label{Lem:Clifford_matrix_1}
For $A$ as in \refeqn{minimal_matrix}, all of $a b^*, cd^*, c^* a, d^* b, ba^*, dc^*, a^* c, b^* d $ lie in $\$ \R^3$.
\end{lem}
In other words, matrices satisfying \refeqn{minimal_matrix} satisfy \refeqn{Vahlen_conditions_1}. This is equivalent to saying that $(a^*,b^*)$ and $(c^*,d^*)$ lie in $S\HH$.

\begin{proof}
As the columns of $A$ lie in $S\HH$, we have $a^* c, b^* d \in \$\R^3$ and, by \reflem{spinor_condition}, $a \overline{c}, b \overline{d} \in \$\R^3$ also. We then have $(b \overline{d})^* = d' b^* \in \$\R^3$. 

Since $\pdet A = 1$ we have 
$(a^* d - c^* b)' = 1$,
i.e. $\overline{a} d' - \overline{c} b' = 1$. Thus 
\[
a b^* 
= a (\overline{a} d' - \overline{c} b' ) b^*
= |a|^2 d' b^* - |b|^2 a \overline{c} \in \$\R^3.
\]
Similarly we have $\overline{(a \overline{c})} = c \overline{a}$ and $(b \overline{d})' = b' d^*$ in $\$\R^3$, so
\[
cd^*
= c (\overline{a} d' - \overline{c} b') d^*
= |d|^2 c \overline{a} - |c|^2 b' d^* \in \$\R^3.
\]
Now as $a^* c$, $b^* d$, $ab^*$ and $cd^*$ are all in $\$\R^3$, so too are $c^* a = (a^* c)^*$, $d^* b = (b^* d)^*$, $ba^* = (ab^*)^*$, and $dc^* = (cd^*)^*$.
\end{proof}

\begin{lem}
\label{Lem:Clifford_matrix_2}
For $A$ as in \refeqn{minimal_matrix}, 
then $ad^* - bc^* = da^* - cb^* 
= d^* a - b^* c = 1$.
\end{lem}
In other words, matrices satisfying \refeqn{minimal_matrix} satisfy \refeqn{Vahlen_conditions_2}. 

\begin{proof}
We first show $ad^* - bc^* = 1$.
As $(b,d) \in S\HH$ we have $b \overline{d} \in \$\R^3$, so $b \overline{d} = (b \overline{d})^* = d' b^*$.
And by \reflem{Clifford_matrix_1} we have $cd^* \in \$\R^3$ so $cd^* = (cd^*)^* = dc^*$.
Thus
\begin{equation}
\label{Eqn:Ahlfors_adapted_1}
|d|^2 ( a^* d - c^* b)
=  d (a^* d - c^* b) \overline{d}
= |d|^2 d a^* - c d^* d' b^*
= |d|^2 (da^* - c b^*).
\end{equation}
Similarly, using $cd^* = (cd^*)^* = dc^*$ and $a \overline{c} = (a \overline{c})^* = c' a^* \in \$\R^3$, we obtain 
\begin{equation}
\label{Eqn:Ahlfors_adapted_2}
|c|^2 ( a^* d - c^* b )
= c' ( a^* d - c^* b ) c^*
= a \overline{c} c d^* - |c|^2 b c^*
= |c|^2 ( ad^* - bc^* ).
\end{equation}
Now, since $a^* d - c^* b = 1$,
$c$ and $d$ cannot both be zero. If $d \neq 0$ then \refeqn{Ahlfors_adapted_1} shows $da^* - cb^* = 1$, and taking $*$, we obtain $ad^* - bc^* = 1$;
if $c \neq 0$ then \refeqn{Ahlfors_adapted_2} shows 
$ad^* - bc^* = 1$.

Taking $*$ of $a^* d - c^* b = ad^* - bc^* = 1$
then yields 
$d^* a - b^* c = da^* - cb^* = 1$.
\end{proof}

We also show that all prospective definitions agree.
\begin{lem}
Let $A$ be a quaternionic $2 \times 2$ matrix. The following are equivalent.
\begin{enumerate}
\item
$A$ satisfies \refdef{SL2H}.
\item
$A$ satisfies \refeqn{Vahlen_general} with $n=3$.
\item
$A$ satisfies \refeqn{minimal_matrix}.
\end{enumerate}
\end{lem}

\begin{proof}
Clearly (i) implies (iii), and \reflem{Clifford_matrix_1} and \reflem{Clifford_matrix_2}  show (iii) implies (i).

If $A$ satisfies \refeqn{Vahlen_general} with $n=3$ then the conditions $\overline{a} b, b \overline{d}, \overline{d} c, c \overline{a} \in \$\R^3$ imply that the columns $(a,c), (b,d) \in \HH$. We also have $a^* b' = (\overline{a} b)' \in \$\R^3$, so $(a^*, b^*) \in S\HH$. Similarly $c^* d' = (\overline{d} c)^* \in \$\R^3$, so $(c^*, d^*) \in S\HH$. Thus all the conditions of \refeqn{Vahlen_conditions_1} are satisfied. From \refeqn{Vahlen_general} we also have $ad^* - bc^* = 1$ and, taking $*$, $da^* - cb^* = 1$. Again we cannot have both $c$ and $d$ zero, so applying \refeqn{Ahlfors_adapted_1} if $d \neq 0$, and \refeqn{Ahlfors_adapted_2} if $c \neq 0$ as above, we conclude 
$a^* d - c^* b = 1$
and, taking $*$, 
$d^* a - b^* c = 1$.
Thus all the conditions of \refeqn{Vahlen_conditions_2}, hence also \refdef{SL2H}, are satisfied.

Conversely, if $A$ satisfies \refdef{SL2H} then the columns are spinors, as are $(a^*, b^*), (c^*, d^*)$, and $ad^* - bc^* = 1$, so \refeqn{Vahlen_general} is satisfied.
\end{proof}

Thus $SL_2\$ = SL_2\$\Gamma_3$. Preferring the less cumbersome notation, we henceforth simply write $SL_2\$$.

\subsection{Properties of Clifford matrices}
\label{Sec:Clifford_properties}

Applying the general theory discussed in \refsec{paravector_Mobius} to the case $(p,q) = (0,2)$, and the paravectors $\$\R^3$ in $\Cl(\R^{0,2}) \cong \HH$, we have that $SL_2\$$ is a group, and that elements of $SL_2\$$ yield well defined M\"{o}bius transformations of $\$\R^3 \cup \{\infty\}$, given for $v \in \$\R^3$ by 
\begin{equation}
\label{Eqn:Mobius_from_Clifford_for_quaternions}
\begin{pmatrix} a & b \\ c & d \end{pmatrix}
\quad \rightsquigarrow \quad
v \mapsto (av+b)(cv+d)^{-1}.
\end{equation}
The only matrices in $SL_2\$$ yielding the identity M\"{o}bius transformation are $\pm 1$; we write $PSL_2\$ = SL_2\$/\{\pm 1\}$. Then $SL_2\$$ is the double cover of the group of M\"{o}bius transformations of $\$\R^3 \cup \{\infty\}$, which is also $\Isom^+ \hyp^4$. We have $SL_2\$ \cong \Isom^S \hyp^4$ and $PSL_2\$ \cong \Isom^+ \hyp^4$. As $\Isom^+ \hyp^4$ is diffeomorphic to $\R^4 \times SO(4)$, and $SL_2\$$, as its spin/universal double cover, we have diffeomorphisms $SL_2\$\Gamma \cong \R^4 \times \Spin(4) \cong \R^4 \times S^3 \times S^3$ and $PSL_2\Gamma \cong \R^4 \times SO(4)$.

We now collect some elementary facts about Clifford matrices. First, we have some immediate observations, which also appear in the literature, e.g. \cite{Ahlfors_Mobius85, Ahlfors_Clifford85, Ahlfors_Mobius_86, Ahlfors_84, Ahlfors_fixedpoints_85, Kellerhals01, Waterman_93}. 
\begin{lem} \
\label{Lem:elementary_vahlen_properties}
\begin{enumerate}
\item
$SL_2 \C \subset SL_2\$$.
\item
The diagonal matrices in $SL_2\$$ are precisely those of the form
\[
\begin{pmatrix} a & 0 \\ 0 & a^{-1*} \end{pmatrix}, \quad
a \in \HH^\times,
\]
which correspond to the M\"{o}bius transformations $v\mapsto ava^* = \sigma(a)(v)$ over $a \in \HH^\times$.
\item
The inverse of a Clifford matrix is given by
\[
\begin{pmatrix} a & b \\ c & d \end{pmatrix}^{-1}
=
\begin{pmatrix} d^* & -b^* \\ -c^* & a^* \end{pmatrix}.
\]
\end{enumerate}
\qed
\end{lem}

Although the pseudo-determinant is not in general multiplicative, it is preserved by $SL_2\$$.
\begin{lem}
\label{Lem:pdet_preserved}
Let $M$ be a $2 \times 2$ quaternion matrix, and 
\[
A = \begin{pmatrix} a & b \\ c & d \end{pmatrix} \in SL_2\$.
\quad \text{Then }
\pdet(AM) = \pdet (M).
\]
\end{lem}

\begin{proof}
Writing $M$ with columns $(\xi_1, \eta_1), (\xi_2, \eta_2) \in \HH^2$, we have
\[
AM 
= 
\begin{pmatrix} a & b \\ c & d \end{pmatrix}
\begin{pmatrix} \xi_1 & \xi_2 \\ \eta_1 & \eta_2 \end{pmatrix}
= 
\begin{pmatrix}
a \xi_1 + b \eta_1 & a \xi_2 + b \eta_2 \\
c \xi_1 + d \eta_1 & c \xi_2 + d \eta_2
\end{pmatrix}
\]
which has pseudo-determinant
\begin{align*}
\left( a \xi_1 + b \eta_1 \right)^* \left( c \xi_2 + d \eta_2 \right)
&- \left( c \xi_1 + d \eta_1 \right)^* \left( a \xi_2 + b \eta_2 \right)
=
\left( \xi_1^* a^* + \eta_1^* b^* \right) \left( c \xi_2 + d \eta_2 \right) - \left( \xi_1^* c^* + \eta_1^* d^* \right) \left( a \xi_2 + b \eta_2 \right) \\
&= 
\xi_1^* \left( a^* c - c^* a \right) \xi_2
+ \xi_1^* \left( a^* d - c^* b \right) \eta_2
+ \eta_1^* \left( b^* c - d^* a \right) \xi_2
+ \eta_1^* \left( b^* c - d^* b \right) \eta_2.
\end{align*}
Now as $A \in SL_2\$$ we have $a^* c, b^* d \in \$\R^3$  so $a^* c = c^* a$ and $b^* d = d^* b$; we also have 
$a^* d - c^* b = d^* a - b^* c = 1$.
Hence the pseudo-determinant is 
$\xi_1^* \eta_2 - \eta_1^* \xi_2 = \pdet M$.
\end{proof}

Any spinor $\kappa$ can be a column of a Clifford matrix; the complementary spinor $\check{\kappa}$ helps us find one.
\begin{lem}
\label{Lem:Vahlen_with_arbitrary_column}
For any $\kappa \in S\HH$, $(\kappa, -\frac{\check{\kappa}}{|\kappa|^2}) \in SL_2\$$.
\end{lem}
Here $(\kappa_1, \kappa_2)$ denotes the matrix with columns $\kappa_1, \kappa_2$, as in \refsec{bracket}.
\begin{proof}
It is sufficient to show the columns are spinors and $\pdet = 1$. These follow from \reflem{complementary_spinor_facts}(i) and (iv). 
\end{proof}

\subsection{Parabolic translation matrices}
\label{Sec:parabolic_clifford}

We will need to consider those Clifford matrices which, when regarded as M\"{o}bius transformations and hyperbolic isometries in the upper half space model, translate along horospheres. M\"{o}bius transformations of various types, including in 4 and higher dimensions and involving quaternions, have been studied in numerous papers of
Ahlfors, Cao, Foreman, Gongopadhyay, Kellerhals, Parker, Short, Waterman, and others, e.g. \cite{Ahlfors_Clifford85, Gongopadhyay_12, Gongopadhyay_Kulkarni_09, Ahlfors_fixedpoints_85, Kellerhals01, Cao_07, CPW_04, Cao_Waterman_98, Foreman_04, Parker_Short_09, Waterman_93}.

When $n=2$ or $3$, the nontrivial isometries in $\Isom^+ \hyp^n$ translating along horospheres are precisely those with a single fixed point at infinity (and no fixed points in $\hyp^n$), also known as \emph{parabolic} isometries. However in $\hyp^4$, there are isometries which have a single fixed point at infinity, and no fixed points in $\hyp^4$, but which do not translate along horospheres. For example, the isometry of $\hyp^4$ given in the upper half space model by the M\"{o}bius transformation $z \mapsto e^{i \theta} z e^{i \theta} +j$ ($0< \theta < \pi/2$ say) has the unique fixed point $\infty$, and preserves all the horospheres about $\infty$, on which it acts as screw motions.

It is standard in the literature that a M\"{o}bius transformation is \emph{parabolic} if it has a single fixed point in $\overline{\hyp^n} = \hyp^n \cup \partial \hyp^n$, which is necessarily in $\partial \hyp^n$. 
Following the terminology of 	\cite{Gongopadhyay_12, Kellerhals01, Waterman_93}, we consider a smaller set of \emph{translations}, which we define below via several equivalent characterisations. Waterman \cite{Waterman_93} calls them \emph{strictly parabolic} and Cao--Parker--Wang \cite{CPW_04} call them \emph{simple parabolics}. Gongopadhyay \cite{Gongopadhyay_12} defines a notion of \emph{$k$-rotary parabolic} isometries, as those with $k$ rotation angles (see also \cite{Gongopadhyay_Kulkarni_09}); the $0$-rotary parabolics are translations.

\begin{lem}
\label{Lem:parabolic_conditions}
For a matrix $A \in SL_2\$$, the following are equivalent.
\begin{enumerate}
\item
$A \neq 1$ and $(A-1)^2 = 0$.
\item
$A$ is conjugate to $P_0 = \begin{pmatrix} 1 & 1 \\ 0 & 1 \end{pmatrix}$. 
\item
$A = \begin{pmatrix} 1-ac^* & aa^* \\ -cc^* & 1+ca^* \end{pmatrix}$ for some $(a,c) \in S\HH$.
\end{enumerate}
\end{lem}

\begin{proof}
We show equivalence of (ii) and (iii) by a calculation. Considering a matrix
\[
B = \begin{pmatrix} a & b \\ c & d \end{pmatrix} \in SL_2\$,
\]
we compute
\begin{align}
\label{Eqn:parabolic_calculation}
B \begin{pmatrix} 1 & 1 \\ 0 & 1 \end{pmatrix} B^{-1}
&= \begin{pmatrix} ad^* - bc^* - ac^* & -ab^* + ba^* + aa^* \\
cd^* - dc^* - cc^* & -cb^* + da^* + ca^* \end{pmatrix} 
= \begin{pmatrix} 1 - ac^* & aa^* \\ -cc^* & 1 + ca^* \end{pmatrix}
\end{align}
since $ad^* - bc^* = da^* - cb^* = 1$, and $ab^*, cd^* \in \$\R^3$ so $ab^* =ba^*$ and $cd^* = dc^*$. This shows (ii) implies (iii). If $A$ satisfies (iii), then by \reflem{Vahlen_with_arbitrary_column} there exists $B \in SL_2\$$ with $(a,c)$ as its first column, and then \refeqn{parabolic_calculation} shows $A$ satisfies (ii).

Now a matrix $A \in SL_2\$$ satisfies (i) iff any conjugate of $A$ satisfies (i). Clearly the matrix $P_0$ satisfies  (i), so (ii) implies (i). It is now sufficient to show (i) implies (ii).

So, suppose $A = \begin{pmatrix} \alpha & \beta \\ \gamma & \delta \end{pmatrix}$  satisfies (i). We claim that $A$ is conjugate to an upper or lower triangular matrix with both diagonal entries $1$.

If $\gamma = 0$, then $A \in SL_2\$$ tells us $\delta = \alpha^{*-1}$. The top left entry of $(A-1)^2 = 0$ is $(\alpha - 1)^2 = 0$, so $\alpha = \delta = 1$, satisfying the claim. If $\gamma \neq 0$, then following \cite[sec. 7]{Ahlfors_Clifford85} or \cite[sec. 3]{Ahlfors_fixedpoints_85}, $A$ has a conjugate in the normal form 
\[
N = \begin{pmatrix} s \gamma & s \gamma s - \gamma^{*-1} \\ \gamma & \gamma s \end{pmatrix} \quad \text{where } s \in \$\R^3.
\]
Then we can calculate that the lower left entry of 
\begin{align*}
0 = (N-1)^2 &= 
\begin{pmatrix} s \gamma - 1 & s \gamma s - \gamma^{*-1} \\ \gamma & \gamma s - 1 \end{pmatrix}^2
\end{align*} 
is $\gamma(s \gamma - 1) + (\gamma s - 1) \gamma = 2(\gamma s \gamma - \gamma)$, so $\gamma s \gamma = \gamma$, hence as $\gamma \neq 0$ we have $s = \gamma^{-1}$. Then $\gamma$ is also a paravector. So $N$ has both diagonal entries $1$ and upper right entry $0$. This proves the claim.

Thus $A$ is conjugate to a matrix $N$ of the form
\[
N = \begin{pmatrix} 1 & \beta \\ 0 & 1 \end{pmatrix}
\quad \text{or} \quad
\begin{pmatrix} 1 & 0 \\ \gamma & 1 \end{pmatrix}.
\]
Now by \reflem{paravector_square_root}, there is a quaternion (in fact, a paravector) $v$ such that $vv^* = \beta$ or $-\gamma$ respectively, so $N$ takes the form (iii) with $(a,c) = (v,0)$ or $(0,v)$, which is a spinor. By equivalence of (ii) and (iii), $N$ is conjugate to $P_0$. Hence $A$ is conjugate to $P_0$, satisfying (ii).
\end{proof}

\begin{defn}
\label{Def:parabolic}
We call a matrix $A \in SL_2\$$ that satisfies any, hence all, of the conditions in \reflem{parabolic_conditions} a \emph{parabolic translation}. The set of parabolic translation matrices is denoted $\P$.
\end{defn}
Gongopadhyay and Waterman \cite{Gongopadhyay_12, Waterman_93} use the conjugacy condition (ii) of \reflem{parabolic_conditions} for their definitions. Polynomials satisfied by M\"{o}bius transformations, as in (i), were studied in \cite{Gongopadhyay_12, Foreman_04}. 

Note it follows from the end of the proof of \reflem{parabolic_conditions} that any matrix of the form
\begin{equation}
\label{Eqn:upper_triangular_unipotent_parabolic}
\begin{pmatrix} 1 & v \\ 0 & 1 \end{pmatrix}
\quad \text{or} \quad
\begin{pmatrix} 1 & 0 \\ v & 1 \end{pmatrix}
\quad \text{with $v \in \$\R^3$ is parabolic.}
\end{equation}

The description given in \reflem{parabolic_conditions} demonstrates several properties of parabolic translation matrices, which we state here. Unlike $SL_2\R$ and $SL_2\C$, they do not necessarily have trace $2$, but have several similar-looking properties. See the work of Cao and Waterman \cite{Cao_Waterman_98, Waterman_93} for further results. In particular, Waterman \cite{Waterman_93} shows that the real part of $a+d^*$ is conjugation invariant, and gives a geometric interpretation.

\begin{lem}
\label{Lem:parabolic_facts}
Let $A = \begin{pmatrix} a & b \\ c & d \end{pmatrix} \in SL_2\$$ be a parabolic translation.
\begin{enumerate}
\item
$a + d^* = 2$.
\item
The off-diagonal entries $b,c$ are paravectors.
\item
$ab+bd = 2b$ and $dc+ca = 2c$.
\item
For any $B \in SL_2\$$, $ABA^{-1}$ is also a parabolic translation.
\end{enumerate}
\end{lem}

\begin{proof}
From the form in \reflem{parabolic_conditions}(iii), we immediately deduce (i) and (ii). The off-diagonal entries of the equation $(A-1)^2 =0$ of \reflem{parabolic_conditions}(i) are precisely the equations of (iii). 
By \reflem{parabolic_conditions}(ii), $\P$ is the conjugacy class of $P_0$, hence invariant under conjugation, giving (iv).
\end{proof}

\subsection{Action of $SL_2$ on spinors}
\label{Sec:SL2_on_spinors}

We now consider $SL_2\$$ actions on the spaces $S\HH \subset \HH^2$, which is simply by matrix-vector multiplication.
We denote all actions of $SL_2\$$ by a dot: for $A \in SL_2\$$ and $\kappa \in \HH^2$, $A.\kappa = A\kappa$.

This action preserves the bracket on $\HH^2$: given $\kappa_1, \kappa_2 \in \HH^2$, denoting by $(\kappa_1, \kappa_2)$ the $2 \times 2$ matrix with $\kappa_1, \kappa_2$ as columns as in \refsec{bracket}, by \reflem{pdet_preserved} we have
\begin{equation}
\label{Eqn:action_preserves_bracket}
\{ A.\kappa_1, A.\kappa_2 \} = \pdet(A\kappa_1, A \kappa_2) = \pdet(A(\kappa_1, \kappa_2)) = \pdet(\kappa_1,\kappa_2) = \{ \kappa_1, \kappa_2\}.
\end{equation}

Restricting to $S\HH \subset \HH^2$, the following lemma ensures that there is a well-defined action on $S\HH$. Since spinors are those $(\xi, \eta)$ whose quotient $\xi \eta^{-1}$ is a paravector, this is essentially the same result as that M\"{o}bius transformations from $SL_2\$$ preserve $\$\R^3 \cup \{\infty\}$, proved by Ahlfors in \cite{Ahlfors_Mobius85}, going back to Maass \cite{Maass_49} and Vahlen \cite{Vahlen_1902}.
\begin{lem}
\label{Lem:action_preserves_spinors}
If $\kappa \in S\HH$ and $A \in SL_2\$$ then $A \kappa \in S\HH$. Moreover, $SL_2\$$ acts transitively on $S\HH$.
\end{lem}

\begin{proof}
If $\kappa = (\xi, \eta) \in S\HH$ then $\xi^* \eta \in \$\R^3$. Letting $A$ have entries $a,b,c,d$ as in \refeqn{minimal_matrix}, we have $A.\kappa = (a \xi + b \eta, c \xi + d \eta)$. Then
\begin{align*}
(a \xi + b \eta)^* (c \xi + d \eta)
&= (\xi^* a^* + \eta^* b^*) (c \xi + d \eta) \\
&= \xi^* a^* c \xi + \xi^* a^* d \eta + \eta^* b^* c \xi + \eta^* b^* d \eta \\
&= \sigma(\xi^*) (a^* c ) + \xi^* (c^* b + 1) \eta + \eta^* b^* c \xi + \sigma(\eta^*)(b^* d) \\
&= \sigma(\xi^*)(a^* c) + \xi^* c^* b \eta + [\xi^* c^* b \eta]^* + \xi^* \eta + \sigma(\eta^*)(b^* d)
\end{align*}
In line 3 we used  $a^*d - c^* b = 1$. Now $a^* c, b^* d \in \$\R^3$, so $\sigma(\xi^*)(a^* c), \sigma(\eta^*)(b^* d) \in \$\R^3$.
The sum of a term and its $*$-conjugate lies in $\$\R^3$, so the above expression lies in $\$\R^3$, hence $A.\kappa \in S\HH$.

By \reflem{Vahlen_with_arbitrary_column} above, for any two spinors $\kappa_1, \kappa_2$, there are matrices $A_1, A_2 \in SL_2\$$ whose first columns are $\kappa_1$ and $\kappa_2$ respectively. Letting $\kappa_0 = (1,0)$, we have $A_1.\kappa_0 = \kappa_1$ and $A_2.\kappa_0 = \kappa_2$. Thus $A_2 A_1^{-1}$ sends $\kappa_1$ to $\kappa_2$.
\end{proof}

Parabolic translation matrices can be characterised by their action on spinors, as follows. See \cite{Zhang_97} for a review of eigenvalues of quaternionic matrices.
\begin{lem}
\label{Lem:parabolic_on_spinors}
Let $A \in SL_2\$$.
\begin{enumerate}
\item
If $A \in \P$ and $A \kappa = \kappa x$ for some $\kappa \in S\HH$ and $x \in \HH$ then $x = 1$.
\item
$A \in \P$ iff $A$ is not the identity and there exists some $\kappa \in S\HH$ such that $A \kappa = \kappa$.
\end{enumerate}
\end{lem}
Thus, the only right eigenvalue of a parabolic translation matrix is $1$; and parabolic matrices are precisely the non-identity matrices in $SL_2\$$ with $1$ as a right eigenvalue.

\begin{proof}
First we consider $P_0$. Note that $P_0 \kappa_0 = \kappa_0$, where $\kappa_0 = (1,0)$. Moreover, if $P_0 \kappa = \kappa x$ for some $\kappa = (\xi, \eta) \in S\HH$ and $x \in \HH$, then $\kappa = \kappa_0 y$ for some $y \in \HH^\times$, and $x=1$. To see this, $P_0 \kappa = \kappa x $ gives $(\xi + \eta, \eta) = (\xi x, \eta x)$. If $\eta \neq 0$ then we obtain $x=1$ so $\xi + \eta = \xi$, hence $\eta = 0$, a contradiction. Thus $\eta = 0$, so $\xi \neq 0$, and $\kappa_0 = (\xi, 0) = \kappa_0 \xi$ where $\xi \in \HH^\times$.

Now a general $A \in \P$ is given as $A=B P_0 B^{-1}$ where $B \in SL_2\$$. Letting $\kappa_1 = B^{-1} \kappa$, we have $P_0 \kappa_1 = \kappa_1 x$. So the above gives $x = 1$ (and in fact $\kappa_1 = \kappa_0 y$ for some $y \in \HH^\times$), proving (i).

Moreover, for $A \in \P$ written as $A = B P_0 B^{-1}$, letting $\kappa = B \kappa_0$, from $P_0 \kappa_0 = \kappa_0$ we have $A \kappa = 
\kappa$, giving the forwards direction of (ii). 

For the converse, suppose $A \in SL_2\$$ is not the identity and $A \kappa = \kappa$. First suppose $\kappa = \kappa_0$. Letting
\[
A = \begin{pmatrix} a & b \\ c & d \end{pmatrix},
\]
$A \kappa_0 = \kappa_0$ implies that $a = 1$ and $c=0$. Then $d^* a - b^* c = 0$ implies $d=1$, and $(b,d) = (b,1) \in S\HH$ implies $b \in \$\R^3$. So
\[
A = \begin{pmatrix} 1 & b \\ 0 & 1 \end{pmatrix}
\quad \text{where} \quad b \in \$\R^3,
\]
which lies in $\P$ by \refeqn{upper_triangular_unipotent_parabolic}.

Now suppose $A \kappa = \kappa$ for a general $\kappa \in S\HH$ and $A \neq 1$. Since $SL_2\$$ acts transitively on $S\HH$ (\reflem{action_preserves_spinors}), there exists $B \in SL_2\$$ be such that $B \kappa = \kappa_0$. Then $BAB^{-1} \kappa_0 
= \kappa_0$, and moreover $BAB^{-1} \neq 1$. So by the preceding paragraph $BAB^{-1} \in \P$. As $\P$ is closed under conjugation (\reflem{parabolic_facts}(iv)) then $A \in \P$.
\end{proof}

\subsection{Action of $SL_2$ on the tangent space to spinors}
\label{Sec:action_SL2_tangent_spinors}

The action of $A \in SL_2\$$ on $\S\HH$ is linear, so we obtain a map on tangent spaces $T_\kappa S\HH \To A_{A \kappa} S\HH$, which is just a restriction of the linear isomorphism $\HH^2 \To \HH^2$ given by $\kappa \mapsto A \kappa$. Moreover, this map respects right multiplication: for any $x \in \HH$ we have $A.(\kappa x) = A \kappa x$. 

Using \reflem{TSH}, we have orthogonal decompositions $T_\kappa S\HH = \kappa \R \oplus \kappa \II \oplus \check{\kappa} \$\R^3$ and $T_{A \kappa} S\HH = A \kappa \R \oplus A \kappa \II \oplus \widecheck{(A\kappa)}\$\R^3$. The action of $A$ preserves the first two summands, sending $\kappa \R \To A \kappa \R$ and $\kappa \II \To A \kappa \II$. However, its action on the third summand $\check{\kappa} \$\R^3$, the copy of the paravectors given by the section $s$ of \refdef{Z}, is more subtle.

Nonetheless, it follows from the preservation of the first two summands that for all $\kappa \in S\HH$, the action of $A$ yields an $\R$-linear map of real-3-dimensional vector spaces
\begin{equation}
\label{Eqn:A_on_quotient_spinors}
\frac{T_\kappa S\HH}{\kappa \HH}
\To
\frac{T_{A\kappa} S\HH}{A \kappa \HH},
\end{equation}
which by \reflem{TSH} are isomorphic as vector spaces to $\check{\kappa} \$\R^3$ and $\widecheck{(A \kappa)} \$\R^3$ respectively (induced by inclusion or projection). It will be useful later in \refsec{SL2_on_paravectors_etc} to understand this map.

These quotients form a vector bundle over $S\HH$, and as each quotient has an ordered basis and orientation given by the equivalence classes of $\check{\kappa}1, \check{\kappa}i, \check{\kappa}j$, we obtain a trivialisation of this bundle. Under their natural isomorphism, the orientations on $T_\kappa S\HH / \kappa \HH$ and $\check{\kappa}\$\R^3$ (\refsec{space_of_spinors}) agree.

The points in $T_\kappa S\HH / \kappa \HH$ are affine 4-planes in $T_\kappa S\HH$, of the form $\widecheck{\kappa} v + \kappa \HH = s_v (\kappa) + \kappa \HH$, over all $v \in \$\R^3$.  These affine 4-planes have a simple interpretation in terms of the bracket $\{ \cdot, \cdot \}$.
\begin{lem}
Let $\kappa \in S\HH$ and $\nu \in \HH^2$. Then $\nu \in s_v (\kappa) + \kappa \HH$ iff $\{ \kappa, \nu \} = - v |\kappa|^2$.
\end{lem}

\begin{proof}
By \refeqn{inner_product_with_sv}, $\{ \kappa, s_v (\kappa) \} = - v |\kappa|^2$, and by \reflem{nondegeneracy_of_spinor_form}, $\{ \kappa, \kappa x \} = 0$. Thus by bilinearity of the bracket, $\nu \in s_v (\kappa) + \kappa \HH$ implies $\{\kappa, \nu \} = - v |\kappa|^2$. Conversely, if $\{ \kappa, \nu \} = - v |\kappa|^2$ then $\{\kappa, \nu - s_v (\kappa) \} = 0$ so by \reflem{nondegeneracy_of_spinor_form} $\nu - s_v(\kappa) \in \kappa \HH$.
\end{proof}

We can then describe the action of $A \in SL_2\$$ on these affine 4-planes, or equivalently, on quotients $T_\kappa S\HH / \kappa \HH$.
\begin{lem}
\label{Lem:A_on_quotient_spinors}
Let $\kappa_0 \in S\HH$, 
$A \in SL_2\$$ and $\kappa_1 
 = A \kappa_0$. For each $v \in \$\R^3$, the action of $A$ restricts to a map
\[
 s_v (\kappa_0) |\kappa_0|^{-2} + \kappa_0 \HH 
\To 
 s_v (\kappa_1) |\kappa_1|^{-2} + \kappa_1 \HH 
\]
which are affine 4-planes in $T_{\kappa_0} S\HH$ and $T_{\kappa_1} S\HH$ respectively.
\end{lem}

\begin{proof}
By the previous lemma, the first affine 4-plane is precisely the set of $\nu \in T_{\kappa_0} S\HH$ such that $\{\kappa_0, \nu\} = - v$, and the second 4-plane is precisely the set of $\nu \in T_{\kappa_1} S\HH$ such that $\{\kappa_1, \nu \} = -v$. By \refeqn{action_preserves_bracket}, for any $\nu \in T_{\kappa_0} S\HH$, we have $\{\kappa, \nu\} = \{ A \kappa, A \nu \}$. So if $\{ \kappa_0, \nu \} = -v$ then $\{ \kappa_1, A \nu\} = -v$.
\end{proof}
We note that this statement can also be proved explicitly by calculating the component of $A \kappa \in T_{\kappa_1} S\HH = \kappa_1 \HH \oplus \check{\kappa}_1 \$\R^3$ in the the second summand using \reflem{explicit_decomposition_in_TSH}.

It will be useful to introduce an inner product on the quotients $T_\kappa S\HH / \kappa \HH$. As we have the orthogonal decomposition and isomorphism
\[
T_\kappa S\HH \cong \kappa \HH \oplus \check{\kappa}\$\R^3, \quad
\frac{T_\kappa S\HH}{\kappa \HH} \cong \check{\kappa}\$\R^3,
\]
we could define a norm on the quotient by minimising over equivalence classes in $T_\kappa S\HH$, and then define an inner product by the polarisation identity; or we could define an inner product on the quotient via its isomorphism with the subspace $\check{\kappa}\$\R^3$. As the decomposition is orthogonal and all spaces are positive definite, these approaches yield the same result, agreeing with the following definition.
\begin{defn}
\label{Def:inner_product_on_spinor_quotient}
Let $\kappa \in S\HH$. We define a positive definite inner product
\[
\left( \cdot, \cdot \right) \colon \frac{T_\kappa S\HH}{\kappa \HH} \times \frac{T_\kappa S\HH}{\kappa \HH} \To \R
\]
as follows. Each element of $T_\kappa S\HH / \kappa \HH$ is of the form $\check{\kappa} v + \kappa \HH$ for some unique $v \in \$\R^3$, and contains a unique representative $\check{\kappa} v \in \check{\kappa}\$\R^3 \subset T_\kappa S\HH$. Then for $v,w \in \$\R^3$ we define
\[
\left( \check{\kappa} v + \kappa \HH, \; \check{\kappa} w + \kappa \HH \right) = \langle \check{\kappa} v, \check{\kappa} w \rangle,
\]
where $\langle \cdot, \cdot \rangle$ is the inner product in $\HH^2$.
\end{defn}
By \reflem{paravector_first_conformal} we have 
\[
\left( \check{\kappa} v + \kappa \HH, \; \check{\kappa} w + \kappa \HH \right) = |\kappa|^2 \; v \cdot w.
\]

With respect to this inner product, the map \refeqn{A_on_quotient_spinors} is conformal, and we now calculate its scaling factor (\refdef{scaling_factor}).
\begin{lem}
\label{Lem:action_on_spinor_quotients_conformal}
Let $\kappa_0 \in S\HH$, $A \in SL_2\$$ and $\kappa_1 = A \kappa_0$. The map
\[
\frac{T_{\kappa_0} S\HH}{\kappa_0 \HH}
\To
\frac{T_{\kappa_1} S\HH}{\kappa_1 \HH}.
\]
induced by the action of $A$ is an orientation-preserving conformal $\R$-linear isomorphism with scale factor $|\kappa_0|^4/|\kappa_1|^4$. 
\end{lem}
In other words, for $v,w \in \$\R^3$,
\[
\left( A \check{\kappa}_0 v + \kappa_1 \HH, A \check{\kappa}_0 w + \kappa_1 \HH \right) = \frac{|\kappa_0|^4}{|\kappa_1|^4} \; \left(\check{\kappa}_0 v + \kappa_0 \HH, \check{\kappa}_0 w + \kappa_0 \HH \right)
= \frac{|\kappa_0|^4}{|\kappa_1|^4} \; \left\langle \check{\kappa}_0 v, \check{\kappa}_0 w \right\rangle.
\]

\begin{proof}
By \reflem{A_on_quotient_spinors}, the action of $A$ sends the element represented by $\check{\kappa}_0 v = s_v (\kappa_0)$ to the element represented by $\check{\kappa}_1 v |\kappa_0|^2 |\kappa_1|^{-2} = s_v (\kappa_1) |\kappa_0|^2 |\kappa_1|^{-2} $. So for $v,w \in \$\R^3$ we have
\begin{align*}
\left( A \check{\kappa}_0 v + \kappa_1 \HH, A \check{\kappa}_0 w + \kappa_1 \HH \right)
&= \left\langle \check{\kappa}_1 |\kappa_0|^2 |\kappa_1|^{-2} v, 
\check{\kappa}_1 |\kappa_0|^2 |\kappa_1|^{-2} w \right\rangle 
= \frac{|\kappa_0|^4}{|\kappa_1|^{4}} \left\langle \check{\kappa}_1 v, \check{\kappa}_1 w \right\rangle
\end{align*}
Thus we have the required conformal map.

Each quotient $T_\kappa S\HH / \kappa \HH$ has an oriented basis $(s_1 (\kappa), s_i (\kappa), s_j (\kappa))$, and $A$ sends such a basis to another, up to the positive scaling factor. So $A$ is orientation-preserving.
\end{proof}

\section{From quaternionic spinors to flags}
\label{Sec:spinor_to_flag}

\subsection{Paravector Hermitian matrices and Minkowski space}
\label{Sec:Hermitian_Minkowski}

We now define a map from the set $S\HH$ spinors to certain matrices, analogous to $2 \times 2$ Hermitian matrices, as mentioned earlier and generalising \cite{Mathews_Spinors_horospheres}. To begin, we define this set of matrices. 

\begin{defn}
Let $A$ be a $2 \times 2$ quaternionic matrix.
\begin{enumerate}
\item
$A$ is \emph{Hermitian} if $A = \overline{A}^T$.
\item
$A$ is a \emph{paravector matrix} if its entries are paravectors. 
\end{enumerate}
The set of Hermitian paravector matrices is denoted $\pH$.
\end{defn}
(A conjugate transpose is often denoted by a $*$, but that notation already being used, we write $\overline{A}^T$.) A matrix in $\pH$ is a real linear 
combination of the ``quaternionic Pauli matrices"
\[
\begin{pmatrix} 1 & 0 \\ 0 & 1 \end{pmatrix}, \quad
\begin{pmatrix} 0 & 1 \\ 1 & 0 \end{pmatrix}, \quad
\begin{pmatrix} 0 & i \\ -i & 0 \end{pmatrix}, \quad
\begin{pmatrix} 0 & j \\ -j & 0 \end{pmatrix}, \quad
\begin{pmatrix} 1 & 0 \\ 0 & -1 \end{pmatrix}.
\]
We can identify paravector Hermitian matrices with points in $(1+4)$-dimensional Minkowski space $\R^{1,4}$. 
Indeed, both $\pH$ and $\R^{1,4}$ are isomorphic 5-dimensional real vector spaces.
We take $\R^{1,4}$ to have coordinates $(T,W,X,Y,Z)$, metric $dT^2 - dW^2 - dX^2 - dY^2 - dZ^2$, denoted $\langle \cdot, \cdot \rangle$, and norm $|p|^2 = \langle p, p \rangle$. (So $|p|^2$ is positive, zero, or negative respectively as $p$ is timelike, lightlike, or spacelike; we leave $|p|$ undefined.) Two vectors $v,w$ in $\R^{1,4}$ are \emph{orthogonal} if $\langle v,w \rangle = 0$. For $v,w$ both timelike or both spacelike, the \emph{angle} $\theta$ between them is formed in the usual way, by $\cos \theta = \langle v, w \rangle / \sqrt{|v|^2 |w|^2}$.

We use a similar identification as in \cite{Mathews_Spinors_horospheres}, and as discussed in the introduction (\refsec{intro_spinors_paravectors}).
\[
\pH \cong \R^{1,4}, \quad
(T,W,X,Y,Z) \leftrightarrow \frac{1}{2} \begin{pmatrix} T+Z & W+Xi+Yj \\ W-Xi-Yj & T-Z \end{pmatrix}.
\]

For a matrix in $\pH$, as its entries are all paravectors, the determinant $\pdet$ of \refeqn{lambda_pdet} agrees with the usual determinant, since all entries are in $\$\R^3$, hence invariant under $*$. We may also take the trace as usual. For a point $p = (T,W,X,Y,Z) \in \R^{1,4}$ corresponding to $S \in \pH$, we have 
\begin{equation}
\label{Eqn:matrix_R14_correspondences}
\Tr S = T, \quad
4 \det S = \langle p, p \rangle = |p|^2.
\end{equation}
The light cone in $\R^{1,4}$ consists of $p$ with $\langle p,p \rangle = 0$, corresponding to $S$ with $\det S = 0$, and the the future light cone $L^+$ consists of all $x$ additionally satisfying $T>0$, corresponding to $\Tr S > 0$. 
The tangent space to $L^+$ at $p$ is defined by the equation $\langle x, p \rangle = 0$, so we have $T_p L^+ = p^\perp$.

We define the \emph{celestial sphere} $\S^+$ as the projectivisation of the light cone, i.e. the space of lightlike 1-dimensional subspaces $\ell \subset \R^{1,4}$. Intersection of $L^+$ with a 4-plane $T=T_0$ also yields a 3-sphere, $\S^+_{T_0}$, which is naturally diffeomorphic to $\S^+$ via projectivisation, so we call $\S^+_{T_0}$ the \emph{celestial sphere at constant $T$}. All $\S^+_T$ and $\S^+$ are diffeomorphic to $S^3$. If $p \in L^+$ has $T$-coordinate $T_0$ and $\ell = p \R$ then projection yields an isomorphism $T_\ell \S^+ \cong T_p \S^+_{T_0}$.

At each point $p \in L^+$, the line $\ell = p \R$ is tangent to $L^+$. The family of celestial spheres at constant $T$, and such lines, provides a product structure $L^+ \cong \S^3 \times \R$. The sphere $\S_T^+$ and line $p \R$ through each point are orthogonal, so we have an orthogonal decomposition of the tangent space
\begin{equation}
\label{Eqn:TL_decomposition}
T_p L^+ = p^\perp = \ell^\perp = p \R \oplus T_p \S_T^+ = \ell \oplus T_p \S_T^+
\end{equation}
and hence an isomorphism
\begin{equation}
\label{Eqn:celestial_sphere_light_cone_quotient}
T_p \S^+_T \cong \frac{T_p L^+}{p \R} = \frac{p^\perp}{p \R}
= \frac{\ell^\perp}{\ell}
\end{equation}
sending each vector to its equivalence class.

This isomorphism can in fact be regarded as an isometry, as follows.
Let $p = (T,W,X,Y,Z) \in L^+$, let $v,w \in T_p L^+ = p^\perp$ and let $\overline{v}, \overline{w}$ be their images in $T_p L^+ / p \R$. Any other representatives $\tilde{v}, \tilde{w} \in T_p L^+$ of $\overline{v}, \overline{w}$ satisfy $\tilde{v} = v+ap$, $\tilde{w} = w+bp$ where $a,b \in \R$, so as $p,v,w \in p^\perp$ we have $\langle \tilde{v}, \tilde{w} \rangle = \langle v+ap, w+bp \rangle = \langle v,w \rangle$. Hence there is a well-defined inner product on each $T_p L^+ / p \R = \ell^\perp / \ell$, induced by the Minkowski inner product. Sending each vector in $T_p L^+$ to its representative in $T_p L^+ / p \R$ yields the same result for the inner product, and hence the isomorphism \refeqn{celestial_sphere_light_cone_quotient} is an isometry.

Note that when two points $p_1, p_2 \in L^+$ are real multiples of each other, we have $p_1 \R = p_2 \R = \ell$, where $\ell$ is a lightlike line along $L^+$. Then we have $p_1^\perp = p_2^\perp = \ell^\perp$, so $T_{p_1} L^+ = T_{p_2} L^+ = \ell^\perp$. Similarly, $p_1^\perp / p_1 \R = p_2^\perp / p_2 \R = \ell^\perp / \ell$. Since each celestial sphere $\S^+_T$ is spacelike, the inner product on each $\ell^\perp/\ell$ is negative definite. The lines $\ell$ are the points of $\S^+$, we also have $T_\ell \S^+ \cong \ell^\perp / \ell$ also negative definite.

\subsection{Orientations in Minkowski space}
\label{Sec:orientations}

We will need to consider orientations on various spaces: Minkowski space $\R^{1,4}$, the future light cone $L^+$, each celestial sphere $\S^+_T$ at constant $T$, and quotient spaces from the light cone $T_p L^+ / p \R = \ell^\perp / \ell$.

We use the following conventions:
\begin{enumerate}
\item
\emph{Orientations of codimension-1 subspaces.}
Given a codimension-1 subspace $W$ of an oriented vector space $V$, a transverse vector $v$ to $W$ induces an orientation on $W$: an ordered basis $B$ for $W$ is positively oriented iff appending $v$ to $B$ yields a positively oriented basis for $V$.

The opposite convention is probably more standard (prepending $v$ to an oriented basis for $W$ yields an oriented basis for $V$), but the usefulness of this convention can be seen from the fact that $(1,i,j,k)$ is a standard oriented basis of $\HH$ over $\R$, and $(1,i,j)$ is a standard oriented basis for $\$\R^3$, identified with $\R^3$ via \refeqn{R3_R3}.
\item
\emph{Orientations of codimension-1 quotient spaces.}
Given a 1-dimensional subspace $W = \R w$ of an oriented vector space $V$, spanned by $w$, we obtain a natural orientation on the quotient $V/W$: an ordered basis $B$ for $V/W$ is positively oriented iff representatives  of $B$ (arbitrarily chosen, this does not affect the orientation), followed by $w$, yield an oriented basis for $V$.

Again, the opposite convention is probably more standard, but by this convention the quotient $\HH / k\R \cong \$\R^3$, obtains its standard orientation from that of $\HH$.
\end{enumerate}

In Minkowski space $\R^{1,4}$ with points denoted $p = (T,W,X,Y,Z)$ we have unit (i.e. norm $\pm 1$) vector fields $\partial_T, \partial_W, \partial_X, \partial_Y, \partial_Z$ in coordinate directions. We endow $\R^{1,4}$ with the standard orientation, so that this ordered basis is positively oriented. We define two radial vector fields: 
\[
\partial_r = T \partial_T + W \partial_W + X \partial_X + Y \partial_Y + Z \partial_Z,
\quad\text{and} \quad
\partial_- = W \partial_W + X \partial_X + Y \partial_Y + Z \partial_Z.
\]
Thus $\partial_r$ is nonzero at all $p \neq 0$, indeed is the position vector of $p$, and $\partial_-$ is nonzero on the complement of the $T$-axis, and is spacelike, pointing radially outward along 4-planes of constant $T$. 

Next, we define an orientation on each spacelike 4-plane $\Pi_{T_0}$ given by $T=T_0$. The vector field $\partial_T$ is normal to $\Pi_{T_0}$ and we endow it with the induced orientation, which is the same as the standard $\R^4$ with coordinates $(W,X,Y,Z)$.

We then define \emph{two} orientations, \emph{outward} and \emph{inward}, on $\S^+_{T_0}$, which is a codimension-1 submanifold of $\Pi_{T_0}$, bounding a 4-ball in $\Pi_{T_0}$, with outward normal vector field $\partial_-$. Thus we define the \emph{outward} orientation to be the one induced by $\partial_-$, and the \emph{inward} orientation induced by $-\partial_-$. Any basis of a tangent space to $\S^+$ is accordingly either oriented \emph{inward} or \emph{outward}. Explicitly, at $p \in \S^+_{T_0}$, an ordered basis $(B_1, B_2, B_3)$ of $T_p \S^+_{T_0}$ has outward orientation iff $(B_1, B_2, B_3, \partial_-)$ is a positively oriented basis of $\Pi_{T_0}$, iff $(B_1, B_2, B_3, \partial_-, \partial_T)$ is a positively oriented basis of $\R^{1,4}$.

Using the isomorphism \refeqn{celestial_sphere_light_cone_quotient} between $T_p \S^+_{T_0}$ and $T_p L^+ / p \R$ for $p \in L^+$, we also obtain inward and outward orientations on each $T_p L^+ / p \R = p^\perp / p \R = \ell^\perp/\ell$ where $\ell = p \R$.

Alternatively, $L^+$ also obtains an orientation as a codimension-1 submanifold of $\R^{1,4}$ with transverse vector field $\partial_-$, pointing out of the solid cone formed by future-pointed timelike vectors.  At any point $p \in L^+$, the 4-plane $T_p L^+ = p^\perp$ thus obtains an orientation. (Any element of $SO(1,4)^+$ sends $\partial_-$ to a vector also pointing out of the solid cone, so this orientation is Lorentz invariant.) The quotient $T_p L^+ / p \R = p^\perp / p \R$ then obtains an orientation as a quotient. With this orientation, an ordered basis $(\underline{B}_1, \underline{B}_2, \underline{B}_3)$ of $p^\perp / p \R$ is positively oriented iff $(B_1, B_2, B_3, \partial_r)$ forms a positively oriented basis of $p^\perp = T_p L^+$ (where each $B_m \in T_p L^+$ is an arbitrary representative of $\underline{B}_m$), iff $(B_1, B_2, B_3, \partial_r, \partial_-)$ forms a positively oriented basis of $\R^{1,4}$. This corresponds to the inward orientation defined above, because $(\partial_-, \partial_T) \mapsto (\partial_r, \partial_-)$ is orientation reversing.

\begin{example}
\label{Eg:orientation_at_p0}
At $p_0 = (1,0,0,0,1) \in \S^+_1$, we have $\partial_- = \partial_Z$. The ordered basis $(\partial_W, \partial_X, \partial_Y)$ of $T_{p_0} \S^+_1$ is outward oriented, since $(\partial_W, \partial_X, \partial_Y, \partial_Z, \partial_T )$ has positive orientation in $\R^{1,4}$.
\end{example}

\subsection{From spinors to the light cone}
\label{Sec:spinors_to_light_cone}

\begin{defn}
The map $\phi_1 \colon S\HH \To \pH \cong \R^{1,4}$ is defined by $\phi_1 (\kappa) = \kappa \overline{\kappa}^T$.
\end{defn}

Thus
\begin{equation}
\label{Eqn:matrix_for_phi1}
\phi_1 (\kappa) = \begin{pmatrix} \xi \\ \eta \end{pmatrix} \begin{pmatrix} \overline{\xi} & \overline{\eta} \end{pmatrix} = \begin{pmatrix} |\xi|^2 & \xi \overline{\eta} \\ \overline{\xi} \eta & |\eta|^2 \end{pmatrix},
\end{equation}
just as in the complex case. By definition $\phi_1 (\kappa) \in \pH$; as the trace is positive and determinant zero, so $\phi_1(\kappa) \in L^+$. Letting $p = \phi_1 (\kappa) = (T,W,X,Y,Z)$, comparing the above with \refeqn{R14_matrix}, we have
\begin{equation}
\label{Eqn:phi1_in_coords}
T = |\xi|^2 + |\eta|^2 = |\kappa|^2, \quad
W+iX+jY = 2 \xi \overline{\eta}, \quad
Z = |\xi|^2 - |\eta|^2.
\end{equation}

In fact the image of $\phi_1$ is precisely $L^+$. To see this, note that an arbitrary element of $\pH$ with positive trace and zero determinant can be written as
\[
\begin{pmatrix}
a & re^{u \theta} \\
r e^{-u \theta} & b
\end{pmatrix}
\]
where the off-diagonal elements are written in polar form. Here $a,b,r,\theta \in \R$, $a,b \geq 0$ with at least one of $a,b$ positive, $r \geq 0$, $u \in \$\R^3 \cap \II$ with $|u| = 1$, (if $r=0$ then $u,\theta$ are arbitrary), and $ab = r^2$. Such a matrix can be obtained as the image under $\phi_1$ of, say, $(\xi, \eta) = (\sqrt{a} e^{u \theta}, \sqrt{b})$, where $\sqrt{\cdot}$ denotes positive square root of a positive real. Note that as $re^{u\theta} \in \$\R^3$ we have $\sqrt{a} e^{u \theta} \in \$\R^3$ also, so this $(\xi, \eta)$ satisfies $\xi \overline{\eta} = \sqrt{ab} e^{u \theta} = r e^{u \theta}$, hence lies in $S\HH$.

Thus $\phi_1$ maps a 7-dimensional domain (topologically $S^3 \times S^3 \times \R$) to a 4-dimensional domain (topologically $S^3 \times \R$). We next describe the fibres of this map.
\begin{lem}
\label{Lem:phi1_fibres}
Let $\kappa_0 = (\xi_0, \eta_0), \kappa_1 = (\xi_1, \eta_1) \in S\HH$. Then $\phi_1 (\kappa_0) = \phi_1 (\kappa_1)$ iff $\kappa_0 = \kappa_1 \alpha$ for some unit $\alpha \in \HH$.
\end{lem}

\begin{proof}
By definition of $\phi_1$ we have $\phi_1 (\kappa_0 ) = \phi_1 ( \kappa_0 \alpha )$ for any unit $\alpha \in \HH$. Now suppose $\phi_1 (\kappa_0) = \phi_1 (\kappa_1)$. Comparing diagonal entries yields $|\xi_0| = |\xi_1|$ and $|\eta_0| = |\eta_1|$. Both cannot be zero; suppose without loss of generality $|\xi_0| \neq 0$. Then $\xi_0 = \xi_1 \alpha$ for some unit $\alpha \in \HH$. Comparing off-diagonal entries then $\xi_0 \overline{\eta_0} = \xi_1 \alpha \, \overline{\eta_0} = \xi_1 \, \overline{\eta_1}$, so $\alpha \overline{\eta_0} = \overline{\eta_1}$ and hence $\eta_0 \overline{a} = \eta_1$. Since $|\alpha| = 1$, $\alpha^{-1} = \overline{\alpha}$ so $\eta_0 = \eta_1 \alpha$ and $\kappa_0 = \kappa_1 \alpha$.
\end{proof}
Thus the fibres of $\phi_1$ are of the form $\kappa S^3$ over $\kappa \in S\HH$, where we regard $S^3$ as the unit quaternions.

In the complex case, the corresponding map $\phi_1$ is the cone on the Hopf fibration, and the restriction of $\phi_1$ to unit spinors $S^3 \subset \C^2$ is the Hopf fibration $S^3 \To S^2$. In the present case the quaternionic Hopf fibration arises.

The quaternionic Hopf fibration is the map $S^7 \To S^4$ with $S^3$ fibres, which sends unit elements $(\xi, \eta) \in \HH^2$ with $|\xi|^2 + |\eta|^2 = 1$ to $\xi \eta^{-1} \in \HH \cup \{\infty\} \cong \R^4 \cup \{\infty\} \cong S^4$. Two points $(\xi_0, \eta_0), (\xi_1, \eta_1) \in S^7$ lie in the same fibre iff $(\xi_0, \eta_0) = (\xi_1, \eta_1)\alpha$ for some unit quaternion $\alpha$, just as for $\phi_1$.

The map $\phi_1$ and the Hopf fibration essentially agree on their common domain $S\HH \cap S^7$ of unit spinors. (Here a unit spinor $\kappa$ is one satisfying $|\kappa| = 1$ as in \refeqn{H2_norm}, which lies in $S^7$ via the identification \refeqn{H2_R8}.) Indeed, $S\HH \cap S^7$ consists precisely of those $(\xi, \eta) \in S^7$ which the Hopf fibration sends to $\$\R^3 \cup \{\infty\} \subset \HH \cup \{\infty\}$, which is an equatorial $S^3 \subset S^4$. For such spinors, $\phi_1$ sends them to have $T$ coordinate $1$, so the image lies on the celestial sphere at $T=1$, i.e. $\S^+_1 \cong S^3$. Moreover $W+Xi+Yj = 2\xi \overline{\eta}$, and $Z$ is given by $Z = |\xi|^2 - |\eta|^2$. Composing with stereographic projection, which sends $S^3 \To \R^3 \cup \{\infty\}$, which we regard as a map $\S^+_1 \To \$\R^3 \cup \{\infty\}$ and from the boundary of the hyperboloid model to the upper half space model of $\hyp^4$, $\phi_1 (\kappa)$ ends up at precisely at $(W+Xi+Yj)/(T-Z) = \xi \overline{\eta} / |\eta|^2 = \xi \eta^{-1}$. (See \refeqn{boundary_hyperboloid_to_upper} in \refsec{H4_models} below, where we consider stereographic project explicitly.)

Thus, $\phi_1$ composed with stereographic projection precisely equals the Hopf map on unit spinors, $S\HH \cap S^7 \To S^3$. Topologically this is a map $S^3 \times S^3 \To S^3$, which is the restriction of the quaternionic Hopf fibration $S^7 \To S^4$ to an equatorial $S^3$ subset of the base $S^4$. It follows from the above that this restriction of the Hopf fibration is a trivial $S^3$ bundle over $S^3$.

\subsection{Tangent space of spinors and derivative of $\phi_1$}
\label{Sec:deriv_phi1}

Following the 3D case, we extend $\phi_1$ to a map to flags by including tangent data. Given a tangent vector $\nu \in T_\kappa S\HH$, we write $D_\kappa \phi_1(\nu)$ for the derivative of $\phi_1$ at $\kappa$ in the direction $\nu$. We then have, for real $t$, 
\[
\phi_1 \left( \kappa + t \nu \right) 
= \left( \kappa + t \nu \right) \left( \overline{\kappa} + t \overline{\nu} \right)^T 
= \kappa \overline{\kappa}^T + \left( \kappa \overline{\nu}^T + \nu \overline{\kappa}^T \right) t + \nu \overline{\nu}^T t^2
\]
so the derivative is given by
\begin{equation}
\label{Eqn:derivative_phi1}
D_\kappa \phi_1 (\nu) = \left. \frac{d}{dt} \phi_1 \left( \kappa + t \nu \right) \right|_{t=0} = \kappa \overline{\nu}^T + \nu \overline{\kappa}^T.
\end{equation}

We now use the structure of the tangent bundle to $S\HH$ from \reflem{TSH}, and $L^+$ from \refeqn{TL_decomposition}, so that if $\phi_1 (\kappa) = p = (T,W,X,Y,Z)$ then we have
\[
D_\kappa \phi_1 \colon T_\kappa S\HH \To T_p L^+, 
\quad \text{i.e.} \quad
\kappa \HH \oplus \check{\kappa} \$\R^3 \To p \R \oplus T_p \S^+_T.
\]
As we now see, $D_\kappa \phi_1$ behaves nicely on these summands.

\begin{prop}
\label{Prop:Derivs_props}
Let $\kappa = (\xi, \eta) \in S\HH$ and $\phi_1 (\kappa) = p = (T,W,X,Y,Z)$ as above.
\begin{enumerate}
\item
$D_\kappa \phi_1$ maps $\kappa \R$ isomorphically to $p \R$. Precisely, $D_\kappa \phi_1 (\kappa) = 2 p$.
\item
$D_\kappa \phi_1$ maps $\check{\kappa} \$\R^3$ isomorphically onto $T_{p} \S^+_T$. 
\item
$\ker D_\kappa \phi_1 = \kappa \II$. 
\end{enumerate}
\end{prop}
Thus, $D_\kappa \phi_1$ is surjective onto $T_p L^+$, mapping $\kappa \R$ onto $p \R$ and $\check{\kappa} \$\R^3$ onto $T_p \S^+_T$. Thus the image of $\kappa$ forms a basis for $p \R$, and the images of $\check{\kappa}, \check{\kappa}i, \check{\kappa} j$ form a basis for $T_p \S^+_T$. A basis for the kernel is given by $\kappa i, \kappa j, \kappa k$.  The kernel $\kappa \II$ is also the tangent space to the fibres of $\phi_1$, as described in \reflem{phi1_fibres}.

\begin{proof}
Taking $\nu = \kappa$ we have
\[
D_\kappa \phi_1 (\nu)
= \kappa \overline{\nu}^T + \nu \overline{\kappa}^T
= 2 \kappa \overline{\kappa}^T = 2 \phi_1 (\kappa)
\]
which is nonzero and proportional to $p = \phi_1 (\kappa)$, so spans $p \R$, proving (i).

Then taking $\nu = \check{\kappa} \, v = (\eta', -\xi') v$ for $v \in \$\R^3$ we have
\[
D_\kappa \phi_1 (\nu)
= \kappa \overline{\nu}^T + \nu \overline{\kappa}^T
= \begin{pmatrix} \xi \\ \eta \end{pmatrix} \overline{v} \begin{pmatrix} \eta^* & -\xi^* \end{pmatrix} + \begin{pmatrix} \eta' \\ - \xi' \end{pmatrix} v \begin{pmatrix} \overline{\xi} & \overline{\eta} \end{pmatrix}
\]
The trace of the result, i.e. its $T$-coordinate, is
\[
\xi \, \overline{v} \, \eta^* - \eta \, \overline{v} \, \xi^* + \eta' v \, \overline{\xi} - \xi' v \, \overline{\eta}
= a - a^* + \overline{a} - a'
\quad \text{where } a = \xi \, \overline{v} \, \eta^*
\]
which is zero by \reflem{conjugation_combination}. Hence $D_\kappa \phi_1 (\nu) \in T_p S_T^+$. We claim that when $v \neq 0$, $D_\kappa \phi_1 (\nu) \neq 0$. 

To see this, suppose to the contrary that there exists some $0 \neq v \in \$\R^3$ such that $D_\kappa \phi_1 ( \check{\kappa} v ) = 0$. Then we have $\kappa \overline{\nu}^T = - \nu \overline{\kappa}^T = -\check{\kappa} (v \overline{\kappa}^T)$, an equality of quaternionic matrices factorised as nonzero $2 \times 1$ and $1 \times 2$ vectors. Now by \reflem{factorisation_fact}, there exists a $1 \times 2$ row vector $\tau = (\alpha, \beta) \in \HH^2$ such that precisely one of $\tau \kappa$ and $\tau \check{\kappa}$ is zero, so precisely one of $\tau \kappa \overline{\nu}^T$ and $-\tau \check{\kappa}(v \overline{\kappa}^T)$ is zero. But these two matrices are equal, and we have a contradiction.

Thus, when $v \neq 0$, $D_\kappa \phi_1 (\check{\kappa} v)$ is a nonzero vector in $T_p \S^+_T$. Since $\check{\kappa} \$\R^3$ and $T_p \S^+_T$ are both 3-dimensional, $D_\kappa \phi_1$ must map them to each other isomorphically. This proves (ii). 

From (i) and (ii), $D_\kappa \phi_1$ is surjective from $T_\kappa S\HH$ onto $T_p L^+$. As these spaces have real dimension $7$ and $4$ respectively, we have $\dim \ker D_\kappa \phi_1 = 3$.
Now we calculate directly that $D_\kappa \phi_1 (\nu) = 0$ when $\nu = (\xi, \eta) u = \kappa u$ with $u$ pure imaginary:
\[
D_\kappa \phi_1 (\nu)
= \kappa \overline{\nu}^T + \nu \overline{\kappa}^T
= \kappa \overline{u} \overline{\kappa}^T + \kappa u \overline{\kappa}^T
= \kappa (\overline{u} + u) \overline{\kappa}^T.
\]
Since $u$ is imaginary, $\overline{u} = -u$, we obtain $D_\kappa \phi_1 (\nu) = 0$. Hence $\kappa \II \subseteq \ker D_\kappa \phi_1$, and as the kernel has dimension 3, this inclusion is in fact equality, as desired. 
\end{proof}

To form multiflags, we will take derivatives of $\phi_1$ in certain directions of $\check{\kappa} \$\R^3$, which map to directions along $\S^+_T$. Recall we defined the section $s_v$ of $TS\HH$ for each $v \in \$\R^3$ in \refdef{Z}, which allows us to access these tangent vectors. We will see next that $\check{\kappa} \$\R^3 \subset T_\kappa S\HH$ maps conformally onto $T_p \S^+_T$.

We then have the derivative in the direction $s_v$ as
\begin{align*}
\label{Eqn:DkappaZkappa}
D_\kappa \phi_1 ( s_v \kappa ) 
= \kappa \, \overline{(s_v \kappa)}^T + (s_v \kappa) \, \overline{\kappa}^T
= \kappa \overline{(J \kappa')v}^T + (J \kappa') v \overline{\kappa}^T
= -\kappa \overline{v} \kappa^{*T} J + J \kappa' v \overline{\kappa}^T,
\end{align*}
using $J$ as in \refeqn{J_eqn}.
When the spinors are complex and $v = i$, this reduces to the expression $\kappa \kappa^T (Ji) + (Ji) \overline{\kappa} \overline{\kappa}^T$ from (2.4) of \cite{Mathews_Spinors_horospheres}.

\begin{example}
\label{Eg:Dphi1_at_10}
At $\kappa_0 = (1,0)$ we have 
$
\phi_1 (\kappa_0) = p_0  
= (1,0,0,0,1) \in \S^+_1.
$
The relevant tangent spaces are $T_{\kappa_0} S\HH = \kappa_0 \HH \oplus \check{\kappa}_0 \$\R^3$, where $\check{\kappa}_0 = (0,-1)$, and $T_{p_0} L^+ = p_0^\perp$, which is spanned by $p_0$ and the basis $\partial_W, \partial_X, \partial_Y$ of $T_{p_0} \S^+_1$. Then $D_{\kappa_0} \phi_1 (\kappa_0) = 2p_0$, $D_{\kappa_0} (\kappa_0 \II) = 0$, 
$D_{\kappa_0} (s_1 (\kappa_0)) 
= 2 \partial_W$,
$D_{\kappa_0} (s_i (\kappa_0)) 
= 2 \partial_X
$,
$D_{\kappa_0} (s_j (\kappa_0)) 
= 2 \partial_Y$.
So the oriented basis $s_1 (\kappa), s_i (\kappa), s_j (\kappa)$ of $\$\R^3$ (\refsec{space_of_spinors}) maps under $D_{\kappa_0} \phi_1$ to the basis $2 \partial_W, 2 \partial_X, 2 \partial_Y$ of $T_{p_0} \S^+_1$. As in \refeg{orientation_at_p0}, this is an outward oriented basis.
\end{example}

\subsection{Conformality on paravectors}
\label{Sec:conformal_paravector}

We saw in \refprop{Derivs_props} that at $\kappa \in S\HH$, mapping to $p \in L^+$ under $\phi_1$ the derivative $D_\kappa \phi_1$ restricts to an isomorphism $\widehat{\kappa} \$\R^3 \To T_p \S^+_T$. And as we saw in \refsec{paravectors_in_TSH}, there is a natural identification of the paravectors with $\check{\kappa} \$\R^3$, given by the map $s(v \cdot) \colon \$\R^3 \cong \check{\kappa} \$\R^3$ of \refeqn{conformal_paravectors_in_TSH}. Thus we have a composition of linear isomorphisms
\begin{equation}
\label{Eqn:paravectors_to_celestial_sphere}
\$\R^3 \stackrel{s(\cdot, \kappa)}{\To} \check{\kappa} \$\R^3 \stackrel{D_\kappa \phi_1}{\To} T_p \S^+_T
\end{equation}
We saw in \reflem{paravector_first_conformal} that the first and second spaces have positive definite inner products, and the first map is conformal. The third space $T_p \S^+_T$ is a spacelike linear subspace of $\R^{1,4}$, so has a (negative definite) inner product and norm, as in \refsec{Hermitian_Minkowski}. We now show this composition is conformal. From \reflem{paravector_first_conformal}, the first map has scaling factor $|\kappa|^2$; we show the second has scaling factor $-4|\kappa|^2$.

\begin{prop}
\label{Prop:paravectors_conformal}
For any $v,w \in \$\R^3$ and $\kappa \in S\HH$,
\begin{equation}
\label{Eqn:both_scaling_factors}
\langle D_\kappa \phi_1 ( \check{\kappa} v ), D_\kappa \phi_1 ( \check{\kappa} w ) \rangle 
= -4 |\kappa|^2 \; \langle \check{\kappa} v, \check{\kappa} w \rangle
= -4 |\kappa|^4 \; v \cdot w.
\end{equation}
\end{prop}

The proof requires the following calculation. 
\begin{lem}
\label{Lem:derivative_det_miracle}
For $v \in \$\R^3$ and $\kappa \in S\HH$,
\[
\det \left( D_\kappa \phi_1 \left( \check{\kappa} v \right) \right)
= - |v|^2 |\kappa|^4.
\]
\end{lem}
Here the derivative is considered as a matrix in $\pH$. 

\begin{proof}
The derivative is given by
\begin{align*}
D_\kappa \phi_1 \left( \check{\kappa} v \right) 
&= \kappa \, \overline{\left( \check{\kappa} v \right)}^T + \check{\kappa}  v \, \overline{\kappa}^T 
= 
\begin{pmatrix} \xi \\ \eta \end{pmatrix} 
\overline{v}
\begin{pmatrix} \eta^* & - \xi^* \end{pmatrix}
+
\begin{pmatrix} \eta' \\ - \xi' \end{pmatrix}
v
\begin{pmatrix} \overline{\xi} & \overline{\eta} \end{pmatrix} \\
&=
\begin{pmatrix}
\xi \overline{v} \eta^* + \eta' v \overline{\xi} 
& - \xi \overline{v} \xi^* +  \eta' v \overline{\eta} \\
\eta \overline{v} \eta^* -\xi' v \overline{\xi} 
& - \eta \overline{v} \xi^* - \xi' v \overline{\eta}.
\end{pmatrix}
\end{align*}
As this matrix is paravector Hermitian, $\det$ or $\pdet$ yield the same result, and we calculate
\begin{align*}
\pdet \left( D_\kappa \phi_1 \left( \check{\kappa} v \right) \right)
&=
- \left( \eta \overline{v} \xi^* + \xi' v \overline{\eta} \right)
\left( \eta \overline{v} \xi^* + \xi' v \overline{\eta}  \right)
- \left( \eta \overline{v} \eta^* - \xi' v \overline{\xi} \right)
\left( - \xi \overline{v} \xi^* + \eta' v \overline{\eta} \right) \\
&=
- \left( \eta \overline{v} \xi^* \right)^2 - 2 |\xi|^2 |\eta|^2 |v|^2
- \left( \xi' v \overline{\eta} \right)^2 
+ \left( \eta \overline{v} \eta^* \right) \left( \xi \overline{v} \xi^* \right) - |\xi|^4 |v|^2 - |\eta|^4 |v|^2 + \left( \xi' v \overline{\xi} \right) \left( \eta' v \overline{\eta} \right) \\
&=
\eta \overline{v} \left( \eta^* \xi - \xi^* \eta \right) \overline{v} \xi^*
- \left( |\xi|^2 + |\eta|^2 \right)^2 |v|^2
+ \xi' v \left( \overline{\xi} \eta' - \overline{\eta} \xi' \right) v \overline{\eta} \\
&= - |\kappa|^4 |v|^2
\end{align*}
In the first line we used 
$v^* = v$. In the final line we used $(\xi, \eta) \in S\HH$, so that $\xi^* \eta, \; \overline{\eta} \xi' \in \$\R^3$ (\reflem{spinor_condition}, \refeqn{spinor_conditions}). As elements of $\$\R^3$ are invariant under $*$, we have $\xi^* \eta = \eta^* \xi$ and $\overline{\eta} \xi' = \overline{\xi} \eta'$.
\end{proof}

\begin{proof}[Proof of \refprop{paravectors_conformal}]
After \reflem{paravector_first_conformal} it is sufficient to prove that, for $v,w \in \$\R^3$ and $\kappa \in S\HH$,
\begin{equation}
\label{Eqn:paravectors_conformal_2}
\langle D_\kappa \phi_1 ( \check{\kappa} v ), D_\kappa \phi_1 ( \check{\kappa} w ) \rangle 
= -4 |\kappa|^4 \; v \cdot w.
\end{equation}
To see this, we first recall \refeqn{matrix_R14_correspondences} that if a matrix $S \in \pH$ corresponds to $p \in \R^{1,4}$ then $4 \det S = |p|^2$. Thus from \reflem{derivative_det_miracle}, we have
\begin{equation}
\label{Eqn:miracle_in_norms}
\left| D_\kappa \phi_1 \left( s_v (\kappa) \right) \right|^2
= - 4 |v|^2 |\kappa|^4.
\end{equation}
Now we apply the polarisation identity on both sides. In $\$\R^3$ we obtain $4 v \cdot w = |v+w|^2 - |v-w|^2$. In $\R^{1,4}$, using real-linearity of the derivative we have
\[
4 \left\langle D_\kappa \phi_1 \left( \check{\kappa} v \right), \; D_\kappa \phi_1 \left( \check{\kappa} w \right) \right\rangle
=
\left| D_\kappa \phi_1 \left( \check{\kappa} (v+w) \right) \right|^2
- \left| D_\kappa \phi_1 \left( \check{\kappa} (v-w) \right) \right|^2.
\]
Applying these polarisation identities to \refeqn{miracle_in_norms} then gives \refeqn{paravectors_conformal_2} as desired.
\end{proof}

We can also consider the effect of $D \phi_1$ quotients $T_\kappa S\HH / \kappa \HH \cong \check{\kappa}\$\R^3$, studied in \refsec{SL2_on_spinors}. This quotient has the inner product defined in \refdef{inner_product_on_spinor_quotient}.

Now we consider the derivative $D_\kappa \phi_1$ at $\kappa \in S\HH$. as applied to the quotient $T_\kappa S\HH / \kappa \HH$. Letting $\phi_1 (\kappa) = p$ lie on $\S^+_T \subset L^+$, and $\ell = p \R$, then the map $D_\kappa \phi_1 \colon T_\kappa S\HH \To T_p L^+$ sends $\kappa \R$ to $\ell$, and $\kappa \II$ is the kernel, so $D_\kappa \phi_1$ yields a well defined map of quotients
\begin{equation}
\label{Eqn:derivative_on_quotients}
D_\kappa \phi_1 \colon 
\frac{T_\kappa S\HH}{\kappa \HH} \To \frac{T_p L^+}{p \R} 
= \frac{\ell^\perp}{\ell},
\quad \text{isomorphic to} \quad
\check{\kappa}\$\R^3 \To T_p \S^+_T 
\end{equation}
The quotient $\ell^\perp/\ell$ is isometric to $T_p \S^+$ and has a well defined negative definite inner product as discussed in \refsec{Hermitian_Minkowski}, and orientations as discussed in \refsec{orientations}. On $T_\kappa S\HH / \kappa \HH \cong \check{\kappa}\$\R^3$ we have the positive definite inner product of \refdef{inner_product_on_spinor_quotient}, and orientation formed by $s_1(\kappa),s_i(\kappa),s_j(\kappa)$ as discussed in \refsec{SL2_on_spinors}. 

We then have the following.
\begin{lem}
\label{Lem:spinor_quotient_conformal}
The map \refeqn{derivative_on_quotients} is conformal with scaling factor $-4 |\kappa|^2$, and is orientation-preserving with respect to the outward orientation on $\S^+$.
\end{lem}

\begin{proof}
As discussed in \refsec{SL2_on_spinors}, every element of $T_\kappa S\HH / \kappa \HH$ has a unique representative of the form $s_\kappa (v)$, for some $v \in \$\R^3$, and for $v,w \in \$\R^3$ we have $(s_\kappa (v) + \kappa \HH, s_\kappa (w) + \kappa \HH) = \langle s_\kappa(v), s_\kappa (w)\rangle = |\kappa|^2 \, v \cdot w$, where $\langle \cdot, \cdot \rangle$ is the inner product on $\HH^2$ and $\cdot$ is the dot product on paravectors.

The conformality statement is now just recalling that the second map $D_\kappa \phi_1$ of \refeqn{paravectors_to_celestial_sphere} has scaling factor $-4|\kappa|^2$. Explicitly, by \refprop{paravectors_conformal} we have 
\[
\langle D_\kappa \phi_1 (s_\kappa (v)), D_\kappa \phi_1 (s_\kappa (w)) \rangle
= -4 |\kappa|^2 \langle s_v (\kappa), s_w (\kappa) \rangle
= -4 |\kappa|^2 (s_\kappa (v) + \kappa \HH, s_\kappa (w) + \kappa \HH).
\]
To see that the map is orientation-preserving, note that at each $\kappa \in S\HH$, the map \refeqn{derivative_on_quotients} sends a standard oriented basis represented by $s_1(\kappa), s_i(\kappa), s_j(\kappa)$ to a basis of $T_{\phi(\kappa)} \S^+_T$. Since $S\HH$ is connected, $\phi_1$ is continuous, and a continuously varying set of oriented bases always has the same orientation, it suffices to verify the statement at a single $\kappa$. From \refeg{Dphi1_at_10}, at $\kappa_0 = (1,0)$ the standard basis is sent to an outward oriented basis of the celestial sphere.
\end{proof}

\subsection{Flags}
\label{Sec:flags}

We can now define flags. As in \cite{Mathews_Spinors_horospheres}, all flags are oriented of signature $(1,2)$, i.e. of the form $\{0\} = V_0 = V_1 \subset V_2$, where $\dim V_1 = 1$, $\dim V_2 = 2$, and $V_1/V_0 = V_1$ and $V_2/V_1$ are endowed with orientations. All vector spaces and dimensions are over $\R$. The following definition is identical to \cite{Mathews_Spinors_horospheres}.
\begin{defn}
A \emph{pointed oriented flag}, or just \emph{flag}, consists of a point $p \in L^+$ and an oriented flag $\{0\} \subset V_1 \subset V_2$ in $\pH \cong \R^{1,4}$ of signature $(1,2)$, such that
\begin{enumerate}
\item $V_1 = p \R$, the \emph{flagpole}, is future-oriented, and
\item $V_2$ is a tangent 2-plane to $L^+$ at $p$, i.e. $V_2 \subset T_p L^+$.
\end{enumerate}
We call $p$ the \emph{basepoint} and say the flag is \emph{based} at $p$. The set of flags is denoted $\F$, and the set of flags based at $p$ is denoted $\F_p$.
\end{defn}
A flag can be recovered from the data of $p$ and the relatively oriented $V_2$, thus we can denote a flag by a pair $(p, V_2)$. 
For $p \in L^+$ and $v \in T_p L^+$, we denote by $[[p,v]]$ the flag based at $p$ with $V_2$ spanned by $p$ and $v$ and $V_2 / p \R$ oriented by (the equivalence class of) $v$. 

As in \cite{Mathews_Spinors_horospheres}, $V_2$ contains no timelike vectors, and $p \R$ generates the unique 1-dimensional lightlike subspace of $V_2$. Since $T_p L^+ = p^\perp$, we have $p \R \subset V_2 \subset p^\perp$. 

Two flags described as $[[p,v]]$, $[[p',v']]$ are equal if and only if $p=p'$ 
and there exist real $a,b,c$ such that $ap+bv+cv' = 0$, where $b,c$ (necessarily nonzero) have opposite sign.

The set $\F_p$ of flags based at $p$ naturally corresponds to the set of oriented lines in $p^\perp / p \R$. Indeed, from a flag $(p,V)$, the quotient $V / p \R$ is a line in $p^\perp / p \R$, which obtains an orientation from the flag orientation. Conversely, an oriented line $\ell + p \R \in p^\perp / p \R$ lifts to a 2-plane $V = \ell + p \R \subset p^\perp$ such that $V / p \R$ has an orientation, so $(p, V)$ is a flag. 
The isometry $T_p \S^+_T \cong p^\perp / p \R$ of \refeqn{celestial_sphere_light_cone_quotient}, for any $T>0$, thus provides a bijective correspondence between $\F_p$ and oriented lines tangent to $\S^+_T$ at $p$, or equivalently, to unit tangent vectors in $T_p \S^+_T$. We have a conformal structure on $\F_p$ using the following fact, which will also be useful in the sequel.
\begin{lem}
\label{Lem:projection_along_light}
Let $p \in L^+$ and $\Pi$ be a vector subspace of $\R^{1,4}$ such that $p \R \subseteq \Pi \subseteq p^\perp$. Let $\pi \colon \Pi \To \frac{\Pi}{p\R}$ be the projection. Then $\pi$ is a well-defined map of inner product spaces, and we have the following:
\begin{enumerate}
\item
If $v_1, v_2, w_1, w_2 \in \Pi$ satisfy $\pi(v_1) = \pi(v_2)$ and $\pi(w_1) = \pi(w_2)$ then $\langle v_1, w_1 \rangle = \langle v_2, w_2 \rangle$.
\item
For any $v,w \in \Pi$, $\langle v,w \rangle = \langle \pi(v), \pi(w) \rangle$.
\end{enumerate}
\end{lem}
In particular, if $\pi(v_1) = \pi(v_2)$ then $|v_1|^2 = |v_2|^2$, and $|v|^2 = |\pi(v)|^2$. The argument is essentially the same as the one in \refsec{Hermitian_Minkowski} showing \refeqn{celestial_sphere_light_cone_quotient} is an isometry.
\begin{proof}
We have $v_1 = v_2 + xp$ and $w_1 = w_2 + yp$ for some $x,y \in \R$, so $\langle v_1, w_1 \rangle = \langle v_2 + xp, w_2 + yp \rangle = \langle v_2, w_2 \rangle$, using $\Pi \subseteq p^\perp$. Thus $\langle \pi(v), \pi(w) \rangle$ is well defined and equal to $\langle v_1, w_1 \rangle = \langle v_2, w_2 \rangle$.
\end{proof}

\begin{lem}
\label{Lem:flag_angle_well_defined}
Let $[[p,v]],[[p,w]] \in \F_p$. Then the angle between $v$ and $w$ is independent of the choice of $v,w$ in the descriptions of these flags.
\end{lem}
Note the angle is well defined since the vectors $v, w$ must be spacelike.

\begin{proof}
Let $v_1, v_2$ be two choices of $v$, and $w_1, w_2$ two choices of $w$. Then $v_1 = ap + b v_2$, where $a,b \in \R$ and $b>0$; similarly $w_1 = cp + dw_2$, where $c,d \in \R$ and $d>0$. As all flags lie in $p^\perp$, by the previous lemma $|v_1|^2 = b^2 |v_2|^2$, $|w_1|^2 = d^2 |w_2|^2$, and $\langle v_1, w_1 \rangle = bd \langle v_2, w_2 \rangle$, the cosines of the angles are equal:
\[
\frac{ \langle v_1, w_1 \rangle }{ \sqrt{|v_1|^2 |w_1|^2} } 
= \frac{ bd \langle v_2, w_2 \rangle }{ \sqrt{ b^2 d^2 |v_2|^2 |w_2|^2 }}
=
\frac{ \langle v_2, w_2 \rangle}{ \sqrt{|v_2|^2 |w_2|^2} }.
\]
\end{proof}
\begin{defn}
\label{Def:flag_angle}
The \emph{angle} between two flags $[[p,v]],[[p,w]]$ is the angle between $v$ and $w$.
\end{defn}

Topologically, each $\F_p \cong UT_p \S^+_T \cong S^2$, and $\F$ is a bundle over $L^+ \cong S^3 \times \R$ with $S^2$ fibres $\F_p$. 
Then $\F \cong UTS^3 \times \R$, where $UTS^3$ is the unit tangent bundle of $S^3$.  We regard $p \in L^+ \cong S^3 \times \R$, and then the 2-plane provides an oriented line in $T_p S^3$, corresponding to a unit vector in $UT_p S^3$. Equivalently, $\F \cong TS^3 \setminus S^3$, the complement of the zero section in $TS^3$. Since $S^3$ (like all 3-manifolds) is parallelisable we have $TS^3 \cong S^3 \times \R^3$, so $\F \cong S^3 \times S^2 \times \R$.

The following lemma generalises lemma 2.8 of \cite{Mathews_Spinors_horospheres}, relating flags to the bracket, and showing when derivatives in different directions yield the same flag.
\begin{lem}
\label{Lem:when_flags_equal}
For $\kappa \in S\HH$, $\nu \in T_\kappa S\HH$, and $v \in \$\R^3$, the following are equivalent:
\begin{enumerate}
\item
$\{ \kappa, \nu \}$ is a negative real multiple of $v$;
\item
$\nu = \kappa x + b s_v (\kappa)$ where $x \in \HH$ and $b$ is real positive;
\item
$[[\phi_1 (\kappa), D_\kappa \phi_1 (\nu) ]] = [[\phi_1 (\kappa), D_\kappa \phi_1 \left( s_v (\kappa) \right) ]]$.
\end{enumerate}
\end{lem}

\begin{proof}
From \refeqn{inner_product_with_sv} above we have $\{\kappa, s_v (\kappa) \}$ is a negative real multiple of $v$, and from \reflem{nondegeneracy_of_spinor_form} $\{\kappa, \kappa x \} = 0$. Thus (ii) implies (i). For the converse, from \reflem{nondegeneracy_of_spinor_form} $\{\kappa, \alpha \} = \{ \kappa, \beta \}$ implies $\alpha - \beta = \kappa x$ for some $x \in \HH$.

To see that (ii) implies (iii), let $x = x_0 + x_1$ where $x_0 \in \R$ and $x_1$ is imaginary. Then by \refprop{Derivs_props}, $D_\kappa \phi_1 (\kappa x) = D_\kappa \phi_1 ( \kappa x_0) 
= 2 x_0 \phi_1 (\kappa)$, a real multiple of $\phi_1 (\kappa)$. So $D_\kappa \phi_1 (\nu)$ is equal to a real multiple of $\phi_1 (\kappa)$, plus $D_\kappa \phi_1 (b s_v (\kappa)) = b D_\kappa \phi_1 (s_v(\kappa))$, where $b$ is positive. Thus the half-plane spanned by $\phi_1 (\kappa)$ and positive multiples of $D_\kappa \phi_1 (\nu)$ is equal to the half-plane spanned by $\phi_1 (\kappa)$ and positive multiples of $D_\kappa \phi_1 (s_v (\kappa))$. Hence $[[\phi_1 (\kappa), D_\kappa \phi_1 (\nu)]] = [[\phi_1 (\kappa), D_\kappa \phi_1 (s_v (\kappa))]]$.

For the converse, if the two flags are equal, then $D_\kappa \phi_1 (\nu ) = a \phi_1 (\kappa) + b D_\kappa \phi_1 (s_v \kappa)$, where $a$ is real and $b$ positive. Since $D_\kappa \phi_1 (\kappa) = 2 \phi_1 (\kappa)$ we have $\nu - \kappa \frac{a}{2}  - b s_v (\kappa) \in \ker D_\kappa \phi_1$, so by \refprop{Derivs_props} is equal to $\kappa c$ for some $c \in \II$. Thus $\nu = \kappa ( \frac{a}{2} + c ) + b s_v (\kappa)$, so (ii) holds.
\end{proof}

\subsection{Multiflags}
\label{Sec:multiflags}

We can now make use of the notion of angle between flags from \refdef{flag_angle} to define a multiflag.
\begin{defn}
\label{Def:multiflag}
A \emph{multiflag} is a triple $(p, V^i, V^j)$ where $p \in L^+$, and $(p, V^i), (p,V^j)$ are orthogonal flags.
The flags $(p, V^i)$ and $(p, V^j)$ are called the \emph{$i$-flag} and \emph{$j$-flag} respectively.
We say the multiflag is \emph{based} at $p$. The set of multiflags is denoted $\MF$, and the set of multiflags based at $p$ is denoted $\MF_p$.
\end{defn}
If the $i$- and $j$-flags in a multiflag are described as $[[p, v^i]]$ and $[[p,v^j]]$ respectively, we denote the multiflag by $[[p,v^i,v^j]]$. By \reflem{flag_angle_well_defined}, the orthogonality of the flags means $\langle v^i, v^j \rangle = 0$.

In a multiflag, since the $i$-flag and $j$-flag have the same basepoint $p$, they have the same 1-plane $p \R$, future-oriented.

If two multiflags $[[p,v^i,v^j]]$, $[[p',w^i,w^j]]$ are equal then $p=p'$ and we have equalities $[[p,v^i]] = [[p,w^i]]$ and $[[p,v^j]]=[[p,w^j]]$, so $a^i p + b^i v^i + c^i w^i = 0$ and $a^j p + b^j v^j + c^j w^j = 0$, with each triple $a^\bullet,b^\bullet,c^\bullet$ real and $b^\bullet,c^\bullet$ of opposite sign.

Topologically, $\MF_p$ is diffeomorphic to $SO(3)$. As we have seen, flags based at $p$ correspond to unit vectors in $T_p \S^+_T$. The $i$-flag and $j$-flag of a multiflag correspond to two orthonormal vectors in $\S^+_T \cong S^3$. This pair of vectors extends to a unique (right-handed orthonormal) frame in $T_p \S^+_T$. Thus $\MF_p$ is diffeomorphic to the space of frames at $p$ in $\S^+_T$, which is diffeomorphic to $SO(3)$ in the standard way. That is, $\MF_p \cong \Fr_p \S^+_T$, where $\Fr \S^+_T$ is the bundle of (oriented orthonormal) frames on $\S^+_T$.

Then $\MF$ is a bundle over $L^+ \cong S^3 \times \R$ with fibres $\MF_p \cong \Fr_p \S^+_T \cong SO(3)$. Indeed, $\MF \cong \Fr S^3 \times \R$. A point of $S^3$ describes a ray in $L^+$, a frame there describes the 2-planes of the flags from its first two vectors, and the $\R$ factor fixes the basepoint $p$. By parallelisability $\Fr S^3 \cong S^3 \times SO(3)$, so $\MF \cong S^3 \times SO(3) \times \R$.

The following map to flags generalises the construction of flags in \cite{Mathews_Spinors_horospheres} and by Penrose--Rindler \cite{Penrose_Rindler84}.
\begin{defn}
The map $\phi_1 \colon S\HH \To \MF$ is defined as
\[
\Phi_1 (\kappa) = [[ \phi_1 (\kappa), D_\kappa \phi_1 (s_i \kappa), D_\kappa \phi_1 (s_j \kappa) ]].
\]
\end{defn}
Thus $\phi_1$ provides the flagpole, and its derivatives in the directions given by the sections $s_i, s_j$ yield the $i$- and $j$-flags.

It follows from \refprop{Derivs_props} and \refdef{Z} that the derivatives of $\phi_1$ in the $s_i (\kappa)$ and $s_j (\kappa)$ directions are nonzero, and from \refprop{paravectors_conformal} that these derivatives are orthogonal. Indeed, the flag directions are given by derivatives of $\phi_1$ in the directions corresponding to paravectors $i,j$ under the section $s$ of \refdef{Z} and the conformal maps \refeqn{paravectors_to_celestial_sphere} discussed in \refsec{conformal_paravector}. So $\Phi_1$ is well defined.

\begin{example}
\label{Eg:Phi1_of_k0}
Let us calculate $\Phi_1 ( \kappa_0 )$ where $\kappa_0 = (1,0)$. From \refeg{Dphi1_at_10} we have $\phi_1  (\kappa_0) = (1,0,0,0,1) = p_0$, and $D_\kappa \phi_1 (s_i \kappa_0) = 2 \partial_X$, $D_{\kappa_0} \phi_1 (s_j \kappa_0) = 2 \partial_Y$.  So we have $\Phi_1 (1,0) = [[p_0, \partial_X, \partial_Y]]$. 
\end{example}

\subsection{Decorated ideal points}
\label{Sec:decorated_ideal_points}

In this section, we discuss how a multiflag based at $p \in L^+$ is equivalent to a conformal identification of $\$\R^3$ with the quotient $T_p L^+ / p \R$, as follows. These are analogous to the decorations we later give on horospheres, for ideal points of $\hyp^4$, and hence we call them ``ideal decorations". Recall from \refsec{Hermitian_Minkowski} that the celestial sphere $\S^+$ is the projectivised light cones, so its points can be regarded as lightlike 1-dimensional subspaces $\ell \subset \R^{1,4}$; such an $\ell$ also corresponds to a point on the ideal boundary of $\hyp^4$ in the hyperboloid model.
\begin{defn}
\label{Def:decorated_ideal_point}
Let $\ell \in \S^+$. A \emph{decoration} on $\ell$ is a conformal orientation-preserving $\R$-linear isomorphism
\[
\psi \colon \$\R^3 \To \frac{\ell^\perp}{\ell},
\]
with respect to the outward orientation on $\ell^\perp/\ell$.
A \emph{decorated ideal point} is such a pair $(\ell, \psi)$. The set of all decorated ideal points is denoted $\S^{+D}$, and the set of all decorations on $\ell$ is denoted $\S^{+D}_\ell$.
\end{defn}
Here $\$\R^3$ are the paravectors, identified with $\R^3$ 
as in \refeqn{R3_R3}, endowed with its standard orientation and dot product, agreeing with the inner product on $\HH$ (\refsec{dot_cross_paravector}). 
The space $\ell^\perp/\ell$ has the negative definite inner product induced from $\R^{1,4}$, of which it is sub-quotient, as discussed in \refsec{Hermitian_Minkowski}.
It has an outwards and inwards orientation as discussed in \refsec{orientations}; we choose the outward orientation because of \refeg{orientation_at_p0} and \reflem{spinor_quotient_conformal}, but the choice is irrelevant to the equivalence with multiflags. The key fact is that once $\psi(i), \psi(j)$ are given (necessarily orthogonal and equal in norm), $\psi$ is determined.
As the inner products on $\$\R^3$ and $\ell^\perp/\ell$ are respectively positive and negative definite, the scaling factor of $\Psi$ (\refdef{scaling_factor}) must be negative.

The set of conformal linear orientation-preserving isometries between 3-dimensional definite vector spaces is diffeomorphic to $SO(3) \times \R$, with the $SO(3)$ acting transitively on frames, so each $\S^{+D}_\ell \cong SO(3) \times \R$. In fact, as $\ell^\perp/\ell \cong T_\ell \S^+$ as in \refsec{Hermitian_Minkowski}, we have $\S^{+D}_\ell \cong \Fr_\ell \S^+ \times \R$, where $\Fr \S^{+}$ is the frame bundle of $\S^+$. Then $\S^{+D}$ is a bundle over $\S^+$ with fibres $\S^{+D}_\ell \cong \Fr_\ell \S^+ \times \R$, so $\S^{+D} \cong \Fr \S^+ \times \R$.

\begin{prop}
\label{Prop:multiflags_ideal_decorations}
There is a smooth bijective correspondence $\Psi \colon \MF \To \S^{+D}$ given as follows.
\begin{enumerate}
\item
Given $[[p, v^i, v^j]] \in \MF$, let $p$ have $T$-coordinate $T_0$, and let $a^i, a^j > 0$ be such that $a^i v^i$ and $a^j, v^j$ have Minkowski norm $-4T_0^2$. Then
\[
[[p, v^i, v^j]] \mapsto (p \R, \psi)
\]
where $\psi$ is uniquely defined by $\psi(i) = a^i v^i + p \R$ and $\psi(j) = a^j v^j + p \R$.
\item
Given $(\ell, \psi) \in \S^{+D}$, let $K<0$ be the scale factor of $\psi$. Then
\[
(\ell, \psi) \mapsto [[ p, \widetilde{\psi(i)}, \widetilde{\psi(j)} ]],
\]
where $p$ is the point on $\ell$ whose $T$-coordinate $T_0$ satisfies $K = -4T_0^2$, and $\widetilde{\psi(i)}, \widetilde{\psi(j)} \in \ell^\perp$ are arbitrary lifts of $\psi(i), \psi(j) \in \ell^\perp/\ell$ to $\ell^\perp = T_p L^+$.
\end{enumerate}
\end{prop}

\begin{proof}
Consider the multiflag $[[p, v^i, v^j]]$. Let $\ell = p \R$. Flags at $p$ correspond to oriented lines in $\ell^\perp / \ell$, and the equivalence classes of $v^i$ and $v^j$ in $\ell^\perp/\ell$ orient the lines corresponding to the $i$-flag and $j$-flag. Note $v^i, v^j$ are both spacelike, as are their equivalence classes in $\ell^\perp/\ell$.

As $T_0 > 0$ and $v^i, v^j$ are spacelike, there are unique $a^i, a^j>0$ such that $a^i v^i, a^j v^j$ have Minkowski norm $-4T_0^2$. Hence $a^i v^i + \ell, a^j v^j + \ell \in \ell^\perp/\ell$ also have norm $-4T_0^2$. As $\ell^\perp / \ell$ is oriented, 3-dimensional and negative definite, there exists a unique $v^1 \in \ell^\perp\ell$, also of Minkowski norm $-4T_0^2$, such that $(a^i v^i + \ell, a^j v^j + \ell, v^1)$ from an oriented basis. There is a unique orientation-preserving conformal $\psi$ such that $\psi(i) = a^i v^i + \ell$ and $\psi(j) = a^j v^j + \ell$, and by choice of $v^1$, this $\psi$ satisfies $\psi(1) = v^1$. Since $\psi$ sends $1,i,j$ of unit norm to elements of norm $-4T_0^2$, $\psi$ has scale factor $K = -4T_0^2$.

From $(\ell, \psi)$ we can set $p$ as the point on $\ell$ with $T$-coordinate $T_0>0$ such that the scale factor $K$ of $\psi$ is $K = -4T_0^2$. Then $\psi(i)$ spans a ray in $\ell^\perp/\ell = T_p L^+ / p \R$ giving the $i$-flag, and $\psi(j)$ spans the $j$-flag. Taking arbitrary lifts $\widetilde{\psi(i)}, \widetilde{\psi(j)} \in \ell^\perp$ then the $i$-flag is $[[p, \widetilde{\psi(i)}]]$ and the $j$-flag is $[[p, \widetilde{\psi(j)}]]$. As $\psi$ is conformal and $i,j \in \$\R^3$ are orthogonal, $\psi(i),\psi(j)$ are orthogonal, so the $i$-flag and $j$-flag are orthogonal, and we have a multiflag as claimed.

To see these correspondences are mutually inverse, suppose $[[p,v^i,v^j]]$ maps to $(\ell, \psi)$, so $\ell = p \R$, $\psi$ has scaling factor $K = -4T_0^2$, and $\psi(i) = a^i v^i + p \R$ and $\psi(j) = a^j v^j + p \R$ as above. Let $(\ell, \psi)$ map to $[[p', \widetilde{\psi(i)}, \widetilde{\psi(j)}]]$. Since $p'$ lies on $p \R$ and has $T$-coordinate $T'$ such that $K = - 4T_0^2 = -4T'^2$ and $T,T'>0$, we have $p = p'$. Since $\psi(i) = a^i v^i + p \R$ and $a^i > 0$, any lift $\widetilde{\psi(i)} \in \ell^\perp$ yields the same flag as $v^i$; similarly $\widetilde{\psi(j)}$ yields the same flag as $v^j$. Thus $[[p,v^i,v^j]] = [[p',\widetilde{\psi(i)},\widetilde{\psi(j)}]]$.

Similarly, suppose $(\ell, \psi)$ maps to $[[p,v^i,v^j]]$, where $\psi$ has scaling factor $K$, the $T$-coordinate $T_0$ of $p$ satisfies $K = -4T_0^2$ and $v^i,v^j \in \ell^\perp$ generate the same flags as $\psi(i),\psi(j) \in \ell^\perp/\ell$. Let $[[p,v^i,v^j]]$ map to $(\ell', \psi')$. Then $\ell' = p \R$ so $\ell' = \ell$. And $\psi'$ has scaling factor $-4T_0^2$, so sends $i,j$ to elements of $\ell^\perp/\ell$ of norm $-4T_0^2$, generating the flags $[[p,v^i]]$ and $[[p,v^j]]$, just like $\psi$. Thus $\psi = \psi'$ and so $(\ell, \psi) = (\ell', \psi')$. So the correspondences are mutually inverse, and they clearly depend smoothly on the input data.
\end{proof}

We have previously seen a decorated ideal point: the map $D_\kappa \phi_1 \circ s(\kappa \cdot)$ of \refeqn{paravectors_to_celestial_sphere}, which maps $\$\R^3 \To \widecheck{\kappa} \$\R^3$, as arises naturally in the tangent bundle to spinors, then $\widecheck{\kappa} \$\R^3 \To T_p \S^+_T$ via $D_\kappa \phi_1$. The latter space is isometric to $\ell^\perp/\ell$ by \refeqn{celestial_sphere_light_cone_quotient}. By \refprop{multiflags_ideal_decorations}, there is a corresponding flag: we show it is precisely $\Phi_1 (\kappa)$. The seemingly arbitrary $-4T_0^2$ in \refprop{multiflags_ideal_decorations} was chosen in order to produce this equality.
\begin{lem}
For any $\kappa \in S\HH$, let $p = \phi_1 (\kappa)$, $\ell = p \R$ and 
\[
\psi \colon \$ \R^3 \To T_p \S^+_T \cong \frac{\ell^\perp}{\ell}, \quad
\psi(v) = D_\kappa \phi_1 \left( s_v (\kappa) \right).
\]
Then $(\ell, \psi)$ is the decorated ideal point corresponding to $\Phi_1 (\kappa) \in \MF$ under $\Psi$. 
\end{lem}

\begin{proof}
Let $p$ have $T$-coordinate $T_0$. By \refeqn{phi1_in_coords} we have $T_0 = |\xi|^2 + |\eta|^2 = |\kappa|^2$. 
By \refprop{paravectors_conformal}, since $s_v (\kappa) = \check{\kappa} v$, $\psi$ has scaling factor $K = -4|\kappa|^4 = -4T_0^2$. Since $|i|^2 = |j|^2 = 1$, we have $|\psi(i)|^2 = |\psi(j)|^2 = -4 T_0^2$; in other words, $|D_\kappa \phi_1 (s_i \kappa)|^2 = |D_\kappa \phi_1 (s_j \kappa)|^2 = - 4T_0^2$. Moreover, by \refprop{Derivs_props}, $D_\kappa \phi_1 (s_i \kappa)$ and $D_\kappa \phi_1 (s_i \kappa)$ both lie in $T_p \S^+_{T_0} \subset \ell^\perp$.

Now $\Phi_1 (\kappa) = [[p, D_\kappa \phi_1 (s_i \kappa), D_\kappa \phi_1 (s_j \kappa)]]$. Under $\Psi$, this multiflag corresponds to the line $p \R = \ell$, together with the conformal map $\$\R^3 \To \ell^\perp/\ell$, with scaling factor $K = -4T_0^2$ which sends $i$ to a generator of the $i$-flag in $\ell^\perp/\ell$ and $j$ to a generator of the $j$-flag in $\ell^\perp/\ell$. As $D_\kappa \phi_1 (s_i \kappa), D_\kappa \phi_1 (s_j \kappa) \in \ell^\perp$ have precisely the correct norm and generate these flags, they must correspond precisely to $\Psi(i), \Psi(j) \in \ell^\perp/\ell$ under the isometry $T_p \S^+_{T_0} \cong \ell^\perp/\ell$. Thus the conformal map is precisely $\psi$ as claimed.
\end{proof}

\begin{example}
\label{Eg:decorated_ideal_point_of_k0}
In \refeg{Phi1_of_k0}, for $\kappa_0 = (1,0)$, we calculated $\Phi_1 (\kappa_0) = [[p_0, \partial_X, \partial_Y]]$ where $p_0 = (1,0,0,0,1) \in \S^+_1$, and in \refeg{Dphi1_at_10} we calculated that $D_{\kappa_0} \phi_1$ sends $s_i (\kappa_0) \mapsto 2 \partial_X$ and $s_j (\kappa_0) \mapsto 2 \partial_Y$. Let the corresponding ideal point of $\hyp^4$ be $\ell_0 = p_0 \R$, so $\ell_0^\perp$ has basis $\partial_W, \partial_X, \partial_Y, p_0$, and as in \refeg{orientation_at_p0}, the quotient $\ell_0^\perp/\ell \cong T_{p_0} \S^+_1$ has outward oriented basis represented by $\partial_W, \partial_X, \partial_Y$. 
The corresponding decorated ideal point is $(\ell_0, \psi_0)$ where $\psi_0 \colon \$\R^3 \To \ell/\ell^\perp$ sends $i \mapsto 2\partial_X$, $j \mapsto 2\partial_Y$, and $1 \mapsto 2\partial_W$.
\end{example}

\subsection{$SL_2$ action and equivariance}
\label{Sec:SL2_on_paravectors_etc}

We now consider the action of $SL_2\$$ on $\pH \cong \R^{1,4}$ and $\MF \cong \S^{+D}$, and the equivariance of $\phi_1$ and $\Phi_1$. 

We studied the action on $S\HH$ in \refsec{SL2_on_spinors}, including on tangent vectors in $T_\kappa S\HH = \kappa \R \oplus \kappa \II \oplus \check{\kappa} \$\R^3$. We showed that the action preserves the summands $\kappa \R$ and $\kappa \II$, on which it acts straightforwardly, and explicitly described its action on the quotients $T_\kappa S\HH / \kappa \HH \cong \check{\kappa}\$\R^3$. 

We now also know the effect of $D \phi_1$ on these summands, as it sends a spinor $\kappa \in S\HH$ to a point $p \in \S^+_T \subset L^+$. 
From \refprop{Derivs_props}, $D_\kappa \phi_1$ sends the $\kappa \R$ direction on $S\HH$ maps to the $p \R$ direction at $p \in L^+$, $\kappa \II$ is the kernel, and $\check{\kappa}\$\R^3$ is sent to $T_p \S^+_T$. Moreover, by \refprop{paravectors_conformal} and \reflem{spinor_quotient_conformal}, the map from $\check{\kappa}\$\R^3$ or equivalently $T_\kappa S\HH / \kappa \HH$ to $T_p \S^+_T$ is orientation-preserving and conformal. The flags of $\Phi_1$ are spanned by $p \R$ and the directions in which $D \phi_1$ sends $s_i (\kappa), s_j (\kappa) \in \check{\kappa}\$\R^3$.

The group $SL_2\$$ acts on $\pH \cong \R^{1,4}$: for a point in $\R^{1,4}$ corresponding to a Hermitian matrix $S \in \pH$, $A \in SL_2\$$ acts by $A.S = AS \overline{A}^T$. This is clearly a group action on $2 \times 2$ Hermitian matrices; we show it preserves paravector Hermitian matrices in the following lemma.
\begin{lem}
If $S \in \pH$ and $A \in SL_2\$$ then $A.S = AS \overline{A}^T \in \pH$. Further, $\phi_1$ is equivalent with respect to this action.
\end{lem}

\begin{proof}
For $\kappa \in S\HH$ we have $A \kappa \in S\HH$ (\reflem{action_preserves_spinors})
\[
\phi_1 (A \kappa) = (A \kappa) \overline{(A \kappa)}^T
= A \kappa \overline{\kappa}^T \overline{A}^T
= A \phi_1 (\kappa) \overline{A}^T
= A.\phi_1 (\kappa).
\]
We claim that the action of $A$ on $\pH \cong \R^{1,4}$ preserves $L^+$: any $p \in L^+$ can be written as $p = \phi_1 (\kappa)$ for some $\kappa \in S\HH$, and then the above calculation shows $A.p = \phi_1 (A \kappa) \in L^+$.

Since vectors in $L^+$ span $\R^{1,4}$ and the action of $A$ on $\pH \cong \R^{1,4}$ is linear, the action of $A$ must preserve $\R^{1,4}$. Thus for any $S \in \pH$ we have $A.S \in \pH$ also.
\end{proof}

Since the action of $A$ preserves $L^+$, it lies in $O(1,4)^+$. (In fact, it lies in $SO(1,4)^+$, as is well known; this also follows from arguments below.) 
We also note that $SL_2\$$ acts transitively on $L^+$, since $\phi_1$ has image $L^+$ and the action of $SL_2\$$ on $S\HH$ is transitive (\reflem{action_preserves_spinors}).

The action of $SL_2\$$ on $L^+$, via its derivative, extends to an action of each $A \in SL_2\$$ on each tangent space to $L^+$. As $A$ acts linearly, we also denote the derivative by $A$. Precisely, if $p \in L^+$ then $A$ sends $p \mapsto A.p \in L^+$ and sends $T_p L^+ \To T_{A.p} L^+$. 
Since the action is by linear maps in $O(1,4)^+$, this is an isometry of tangent spaces.

The equivariance of $\phi_1$ yields the following equivariance property on its derivative
\[
A. D_\kappa \phi_1 (\nu) = D_{A \kappa} \phi_1 (A \nu),
\]
for $\nu \in T_\kappa S\HH$. Indeed, $A.D_\kappa \phi_1 (\nu) = A.(\kappa \overline{\nu}^T + \nu \overline{\kappa}^T) 
= (A \kappa) \overline{(A \nu)}^T + (A \nu) \overline{(A \kappa)}^T = D_{A \kappa} \phi_1 (A\nu)$.

\begin{lem}
The action of $A \in SL_2\$$ on $L^+$ and $TL^+$ above yields an action of $SL_2\$$ on the space $\MF$ of multiflags, given by
\[
[[p, v^i, v^j]] \mapsto [[A.p, A.v^i, A.v^j]]
\]
\end{lem}

\begin{proof}
First, we show that $[[A.p, A.v^i, A.v^j]]$ is in fact a multiflag. Since $p \in L^+$ we have $A.p \in L^+$, and as $v^i, v^j \in T_p L^+$ we have $A.v^i, A.v^j \in T_{A.p} L^+$. Now $\langle v^i, v^j \rangle = 0$, and as the action of $A$ lies in $O(1,4)^+$ we have $\langle A.v^i, A.v^j \rangle = 0$. So indeed $[[A.p, A.v^i, A.v^j]] \in \MF$.

To show that the map is well defined, suppose $[[p, v^i, v^j]] = [[p, w^i, w^j]]$. Then $a^i p + b^i v^i + c^i w^i = 0$ and $a^j p + b^j v^j + c^j w^j = 0$, with each triple $a,b,c$ and $b,c$ of opposite sign. Applying $A$ to these equations we have $a^i A.p + b^i A.v^i + c^i A.w^i = 0$ and $a^j A.p + b^j A.v^j + c^j A.w^j = 0$. Thus $[[A.p, A.v^i, A.v^j]] = [[A.p, A.w^i, A.w^j]]$.
\end{proof}

\begin{lem}
The map $\Phi_1 \colon S\HH \To \MF$ is $SL_2\$$-equivariant.
\end{lem}

\begin{proof}
We have $\Phi_1 (\kappa) = [[ \phi_1 (\kappa), D_\kappa \phi_1 (s_i \kappa), D_\kappa \phi_1 (s_j \kappa) ]]$ so we must show the following multiflags coincide:
\begin{align*}
\Phi_1 (A \kappa) &= [[ \phi_1 (A \kappa), D_{A\kappa} \phi_1 (s_i (A \kappa)), D_{A \kappa} \phi_1 (s_j (A \kappa)) ]], \\
A.\Phi_1 (\kappa) &= [[ A.\phi_1 (\kappa), A.D_\kappa \phi_1 (s_i \kappa), A.D_\kappa \phi_1 (s_j \kappa) ]]
= [[ \phi_1 (A \kappa), D_{A \kappa} \phi_1 ( A s_i (\kappa)), D_{A \kappa} \phi_1 (A s_j (\kappa)) ]],
\end{align*}
where we used equivariance of $\phi_1$ and its derivative.

Now by \refeqn{action_preserves_bracket}, we have $\{ A \kappa, A s_i (\kappa) \} = \{ \kappa, s_i (\kappa) \}$ and $\{ A \kappa, A s_j (\kappa) \} = \{ \kappa, s_j (\kappa) \}$, and by \refeqn{inner_product_with_sv}, these are negative multiples of $i$ and $j$ respectively. By \reflem{when_flags_equal} we then have equalities of flags $[[\phi_1 (A \kappa), D_{A \kappa} \phi_1 (A s_i (\kappa))]] = [[\phi_1 (A \kappa), D_{A \kappa} (s_i (A \kappa))]]$ and $[[\phi_1 (A \kappa), D_{A \kappa} \phi_1 (A s_j (\kappa))]] = [[\phi_1 (A \kappa), D_{A \kappa} (s_j (\kappa))]]$. Thus the two multiflags are equal.
\end{proof}

Being linear and preserving $L^+$, the action of $A$ preserves the rays along $L^+$. So in mapping $p \mapsto Ap$, the tangent direction $p \R \subset T_p L^+$ is mapped to the tangent direction $A p \R \subset T_{Ap} L^+$. The action of $A$ thus yields an action on the quotient spaces,
\[
\frac{T_p L^+}{p\R} \To \frac{T_{A.p} L^+}{A.p\R}.
\]
These are isometric to tangent spaces $T_p \S^+_T$, as in \refeqn{celestial_sphere_light_cone_quotient}, with inwards and outwards orientations as in \refsec{orientations}, and by \reflem{spinor_quotient_conformal}, $D_\kappa \phi_1$ maps the quotients $T_\kappa S\HH / \kappa \HH$ onto them in conformal and orientation-preserving fashion. 

For $\kappa_0 \in S\HH$ and $A \in SL_2\$$, let $\kappa_1 = A \kappa_0$ and $p_0 = \phi_1 (\kappa_0)$, $p_1 = \phi_1 (\kappa_1)$. By equivariance of $\phi_1$ we have $A.p_0 = p_1$. Then we have the following diagram of maps.
\[
\begin{array}{ccc}
\frac{T_{\kappa_0} S\HH}{\kappa_0 \HH} & \stackrel{D_{\kappa_0} \phi_1}{\To} & \frac{T_{p_0} L^+}{p_0 \R} \\
A \downarrow && \downarrow A \\
\frac{T_{\kappa_1} S\HH}{\kappa_1 \HH} & \stackrel{D_{\kappa_1} \phi_1}{\To} & \frac{T_{p_1} L^+}{p_1 \R}
\end{array}
\]
\begin{lem} 
With respect to the outward orientation on each $T_p L^+ / p \R$, all maps in the above diagram are orientation-preserving $\R$-linear conformal isomorphisms. Moreover, the diagram is ``conformally commutative'': for $\nu \in T_\kappa S\HH / \kappa \HH$ we have
\[
D_{\kappa_1} \phi_1 (A \nu) = \frac{|\kappa_0|^2}{|\kappa_1|^2} A.(D_{\kappa_0} \phi_1 (\nu)).
\]
\end{lem}

\begin{proof}
By \reflem{spinor_quotient_conformal}, the top and bottom horizontal maps are conformal and orientation-preserving, with scaling factors $-4|\kappa_0|^2$ and $-4|\kappa_1|^2$ respectively.
By \reflem{action_on_spinor_quotients_conformal}, the left vertical map is conformal and orientation-preserving, with scaling factor $|\kappa_0|^4 |\kappa_1|^{-4}$.
As $A$ acts by linear maps in $O(1,4)^+$, the right vertical map is an isometry.

The top left space has an orthogonal oriented basis represented by $(s_1(\kappa_0), s_i(\kappa_0), s_j (\kappa_0))$, which maps under $D_{\kappa_0} \phi_1$ to an orthogonal oriented basis of $T_{p_0} L^+$; the images of $s_i (\kappa_0)$ and $s_j (\kappa_0)$, together with $p_0 \R $, span the $i$-flag and $j$-flag of $\Phi_1 (\kappa_0)$ respectively. Taking a quotient by $p_0 \R$ sends the $i$-flag and $j$-flag of $\Phi_1 (\kappa_0)$ to oriented lines (the ``$i$-line'' and ``$j$-line'') directed by the equivalence class of $D_{\kappa_0} \phi_1 (s_i (\kappa_0))$ and $D_{\kappa_0} \phi_1 (s_j (\kappa_0))$. Under $A$, these map to the equivalence classes of $D_{A \kappa_0} \phi_1 (A s_i (\kappa_0)) = D_{\kappa_1} \phi_1 (A s_i (\kappa_0))$ and $D_{A \kappa_0} \phi_1 (A s_j  (\kappa_0)) = D_{\kappa_1} \phi_1 (A s_j (\kappa_0))$ in $T_{p_1} L^+/p_1 \R$.

On the other hand, the orthogonal basis of the top left space maps under $A$ to an oriented orthogonal oriented basis of $T_{\kappa_1} S\HH/\kappa_1 \HH$, which by \reflem{A_on_quotient_spinors} is represented by a real multiple of $(s_1 (A \kappa_0), s_i (A \kappa_0), s_j (A \kappa_0)) = (s_1 (\kappa_1), s_i (\kappa_1), s_j (\kappa_1))$. Under $D_{\kappa_1} \phi_1$, these map to an oriented equal-length oriented basis of $T_{p_1} L^+/p_1 \R$, such that the $i$-flag and $j$-flag of $\Phi_1 (A \kappa_0) = \Phi_1 (\kappa_1)$ project in this quotient to $i$- and $j$-lines directed by $D_{A \kappa_0} \phi_1 ( s_i (A \kappa_0)) = D_{\kappa_1} \phi_1 (s_i (\kappa_1))$ and $D_{A \kappa_0} \phi_1 (s_j (A \kappa_0)) = D_{\kappa_1} \phi_1 (s_j (\kappa_1))$. 

By equivariance of $\Phi_1$, these $i$- and $j$-lines in $T_{p_1} L^+ / p_1 \R$ yield the same flags, so they must be positive multiples of each other. A similar argument with the derivatives in the $s_1(\kappa_0)$ direction shows that, either way around the diagram, the standard basis of the top left space represented by $(s_1 \kappa_0, s_i \kappa_0, s_j \kappa_0)$ is sent to a positive multiple of the standard basis of the bottom right space represented by $(D_{\kappa_1} \phi_1 (s_1 \kappa_1), D_{\kappa_1} \phi_1 (s_i \kappa_1), D_{\kappa_1} \phi_1 (s_j \kappa_1))$. Since the scaling factors are all known, the overall factor is $|\kappa_0|^2/|\kappa_1|^2$ as claimed.

Moreover, as the maps commute up to a positive factor, and as 3 of the 4 maps are orientation-preserving, the remaining map must be orientation preserving also. (This shows directly that $A$ acts by orientation-preserving linear maps, i.e. maps in $SO(1,4)^+$.)
\end{proof}

There is also a natural action of $SL_2\$$ on decorated ideal points, as follows. Let $(\ell, \psi) \in \S^{+D}$. Since $\ell$ is a line in $L^+$, $A$ sends $\ell$ to another line $A.\ell$ in $L^+$, and as $A$ acts on $\R^{1,4}$ by linear maps in $SO(1,4)^+$, it also sends the 4-plane $\ell^\perp$ to another 4-plane $A.\ell^\perp$. (Note $(A.\ell)^\perp = A.(\ell^\perp)$.)
Thus $A$ yields a well-defined map $\ell^\perp / \ell \To A. \ell^\perp / A.\ell = A.(\ell^\perp/\ell)$. Since $A$ acts by an element of $SO(1,4)^+$, this action of $A$ is orientation-preserving. Composing with $\psi \colon \$\R^3 \To \ell^\perp / \ell$ yields a map $A. \psi \colon \$\R^3 \To A.(\ell^\perp / \ell)$. Since $\psi$ is orientation-preserving, $\R$-linear and conformal, so is $A.\psi$. We can thus define $A.(\ell, \psi) = (A.\ell, A.\psi)$.

\begin{lem}
The correspondence $\Psi \colon \MF \To \S^{+D}$ of \refprop{multiflags_ideal_decorations} is $SL_2\$$-equivariant. 
\end{lem}

\begin{proof}
Both actions are essentially induced by the action of $SL_2\$$ on $\R^{1,4}$. Let $[[p, v^i, v^j]]$ be a multiflag corresponding to a decorated ideal point $(\ell, \psi)$. So $\ell = p \R$ and $\psi \colon \$\R^3 \To \ell^\perp/\ell$ sends $i,j$ to the equivalence classes of $v^i,v^j$ respectively.

Now consider $A.[[p,v^i,v^j]] = [[A.p,A.v^i,A.v^j]]$ and $A.(\ell, \psi) = (A.\ell, A.\Psi)$. Then we have $A.\ell = A.p \R$, and $A.\Psi(i), A.\Psi(j)$ are the equivalence classes of $A.v^i, A.v^j$ respectively. So the multiflag and decorated ideal point again correspond, and we have equivariance.
\end{proof}

\begin{example}
\label{Eg:equivariance_of_flags_at_10}
We consider $\Phi_1 (\kappa)$ for $\kappa = (x,0) = \kappa_0 x$, where $\kappa_0 = (1,0)$ and $|x| = 1$, i.e. $x \in S^3$. By \reflem{phi1_fibres} and \refeg{Dphi1_at_10} these are precisely the $\kappa$ such that $\phi_1(\kappa) = \phi_1 (\kappa_0) = p_0 = (1,0,0,0,1) \in \S^+_1$.
We have
\begin{equation}
\label{Eqn:S3_matrices}
\kappa = A \kappa_0
\quad \text{where} \quad
A = \begin{pmatrix} x & 0 \\ 0 & x' \end{pmatrix} \in SL_2\$,
\end{equation}
since $\pdet A 
= x^* x'
= 1$. At $p_0$, the tangent space to $\S^+_1$ is the $WXY$ 3-plane.
From \refeg{Phi1_of_k0} and equivariance we have $\Phi_1 (\kappa) = A.\Phi_1 (\kappa_0) = [[p_0, A.\partial_X, A.\partial_Y]]$. We compute
\[
A.\partial_X = A \partial_X \overline{A}^T
=
\begin{pmatrix} x & 0 \\ 0 & x' \end{pmatrix}
\begin{pmatrix} 0 & i \\ -i & 0 \end{pmatrix} 
\begin{pmatrix} \overline{x} & 0 \\ 0 & x^* \end{pmatrix} \\
= \begin{pmatrix} 0 & xix^* \\ -x'i\overline{x} & 0 \end{pmatrix} \\
= \begin{pmatrix} 0 & \sigma(x)(i) \\ \overline{\sigma(x)(i)} & 0 \end{pmatrix},
\]
and $A.\partial_Y$ is the same matrix with $i$ replaced with $j$.
Thus $A.\partial_X, A.\partial_Y$, which together with $p_0$ span the $i$- and $j$-flags of $\Phi_1 (\kappa)$, are obtained from $\partial_X, \partial_Y$ by the rotation $\sigma(x)$ in the $WXY$ 3-plane.  As $x$ varies over $S^3$, the multiflag $\Phi_1(\kappa)$ rotates via the representation $\sigma$ of \refdef{rho}, which provides a homomorphism and double cover $S^3 \To SO(3)$ (\refsec{actions_on_paravectors}), and hence $\Phi_1 (\kappa)$ rotates through all multiflags based at $p_0$. Moreover, $\Phi(\kappa) = \Phi(\kappa_0)$ precisely when $x = \pm 1$, i.e. $\kappa = \pm \kappa_0$.

If we restrict to complex $x = e^{i \theta}$, then $\sigma(x)$ is a rotation of $2\theta$ about $-ik = j$, so the $j$-flag remains constant at $[[p_0, \partial_Y]]$, while the $i$-flag rotates in the $WX$-plane through $2\theta$.
So on the copy of $\R^{1,3}$ given by $Y = 0$, multiflags reduce to the flags of \cite{Mathews_Spinors_horospheres} and \cite{Penrose_Rindler84}.
\end{example}

As $SL_2\$$ acts transitively on $L^+$, and by the above example acts transitively on the space $\MF_{p_0}$ of multiflags at $p_0$, then $SL_2\$$ acts transitively on $\MF$, with the stabiliser of each $p \in L^+$ being an $S^3 \cong \Spin(3)$ subgroup  conjugate to the group of matrices in \refeqn{S3_matrices} above. It follows that $\Phi_1$ is surjective onto $\MF$, and $\Phi(\kappa_1) = \Phi(\kappa_2)$ iff $\kappa_1 = \pm \kappa_2$. Topologically, $\Phi_1 \colon S\HH \To \MF$ is a map $S^3 \times S^3 \times \R \To S^3 \times SO(3) \times \R$ (\reflem{topology_of_SH}, \refsec{multiflags}), which is a double cover.

\section{Horospheres and decorations}
\label{Sec:Minkowski_horospheres}

We now turn to hyperbolic geometry, and define spaces of horospheres $\Hor$ and decorated horospheres $\Hor^D$, and maps $\phi_2, \Phi_2$
of the commutative diagram \refeqn{main_thm_diagram}.

\subsection{Hyperbolic geometry in general}
\label{Sec:hyp_geom_general}

The hyperboloid model $\hyp^4$ of hyperbolic 4-space is
\[
\hyp^4 = \left\{ x = (T,W,X,Y,Z) \in \R^{1,4} \; \mid \; \langle x,x \rangle = 1, \; T>0 \right\}.
\]
Thus $\hyp^4$ is a spacelike codimension-1 submanifold of $\R^{1,4}$, with tangent spaces $T_x \hyp^4 = x^\perp$.
The boundary at infinity $\partial \hyp^4$ given by the celestial sphere $\S^+ \cong S^3$, whose points are the lines $\ell$ in $L^+$. 

The isometry group $\Isom^+ \hyp^4$ is isomorphic to $SO(1,4)^+$ which, acting linearly on $\R^{1,4}$ in the standard way, preserving $L^+$ and $\hyp^4$. This group acts on the celestial sphere $\S^+$ by conformal orientation-preserving maps, which are the orientation-preserving M\"{o}bius transformations of $\S^+ \cong S^3$.

We saw in \refsec{SL2_on_paravectors_etc} that $SL_2\$$ acts on $\R^{1,4}$ by linear maps in $SO(1,4)^+$. One can check that the resulting map $SL_2\$ \To SO(1,4)^+$ has kernel $\{\pm 1\}$. Indeed, as discussed in \refsec{Clifford_properties}, this is a spin double cover, and we have isomorphisms $PSL_2\$ \cong SO(1,4)^+ \cong \Isom^+ \hyp^4$. So $SL_2\$$ is also the spin double cover of the isometry group $\Isom^+ \hyp^4$, and we may regard its elements as \emph{spin isometries} of $\hyp^4$; we denote this group $\Isom^S \hyp^4$.

\subsection{Horospheres and their geometry}
\label{Sec:horospheres_geometry}

As in \cite{Penner87} for $\R^{1,2}$ and \cite{Mathews_Spinors_horospheres} for $\R^{1,3}$, horospheres $\h$ in $\hyp^4 \subset \R^{1,4}$  are given by intersections with affine hyperplanes $\Pi$ of the form $\{ x \mid \langle x, p \rangle = 1 \}$ where $p \in L^+$. Such planes have (Minkowski) normal $p$ lightlike and hence we call them lightlike. As $\h = \hyp^4 \cap \Pi$, the tangent space to $\h$ at a point $q$ is $q^\perp \cap p^\perp$. The line $\ell = p \R$ is a point of $\S^+$ and we say $\ell$ is the \emph{centre} of $\h$. 
The following definition follows \cite{Mathews_Spinors_horospheres}.

\begin{defn}
The set of horospheres in $\hyp^4$ is denoted $\Hor$.

The map $\phi_2 \colon L^+ \To \Hor$ sends $p \in L^+$ to the horosphere with equation $\langle x, p \rangle = 1$.

The map $\phi_2^\partial \colon L^+ \To \partial \hyp^4$ sends $p$ to the point at infinity of $\phi_2 (p)$.
\end{defn}

As in other dimensions, $\phi_2$ is a diffeomorphism, with both $L^+$ and $\partial \hyp^4$ diffeomorphic to $S^3 \times \R$. The map $\phi_2^\partial$ can be regarded as projectivisation $L^+ \To S^3$.

As $SL_2\$$ acts on $\hyp^4$ via hyperbolic isometries, it also acts on $\partial \hyp^4$ and $\Hor$. As in other dimensions, a horosphere $\phi_2 (p)$ with equation $\langle x, p \rangle = 1$ is sent to a horosphere with equation $\langle A^{-1}.x, p \rangle = 1$, which is equivalent to $\langle x, A.p \rangle = 1$, so $A.\phi_2 (p) = \phi_2 (A.p)$ and $\phi_2$ is $SL_2\$$-equivariant.

Indeed, the fact that $\phi_2$ is an $SL_2\$$-equivariant diffeomorphism implies that $A \in SL_2\$$ fixes a point $p \in L^+$ iff it fixes the corresponding horosphere $\h = \phi_2 (p)$.

As is well known, a horosphere is isometric to a Euclidean space. We can see this directly as follows.
\begin{lem}
\label{Lem:projection_for_horosphere}
Let $p \in L^+$ and $\h = \phi_2 (p)$.
Then the projection $\pi \colon \R^{1,4} \To \frac{\R^{1,4}}{p \R}$ restricts to an isometry
$\h \To \frac{\Pi}{p \R}$, which is isometric to $\frac{T_p L^+}{p \R}$, and $\E^3$ with negative definite inner product.
\end{lem}
Note here that $\Pi$ is an affine 4-plane, so $\Pi/p \R$ is not a quotient of vector spaces, but an affine 3-plane in $\R^{1,4}/p \R$, parallel to $p^\perp / p \R$, whose points are $x + p \R$ over $x \in \Pi$. Here $\E^3$ denotes $\R^3$ with the Euclidean metric.

\begin{proof}
Each ray parallel to $p \R$ in $\Pi $ intersects $\hyp^4$ at a single point of $\h$, 
so the restriction of $\pi$ to $\h$ yields a diffeomorphism onto $\Pi/p \R$. At each $q \in \h$, we have $T_q \h = q^\perp \cap p^\perp$. As $\Pi$ is parallel to $p^\perp$, each tangent space of $\Pi$ is $p^\perp$, and hence each tangent space of $\Pi / p \R$ is $p^\perp / p \R$. Then the derivative of $\pi|_\h \colon \h \To \Pi/p \R$ at each point $q \in \h$ is a linear isomorphism $q^\perp \cap p^\perp \To p^\perp / p \R$, which is a restriction of the projection $p^\perp \To p^\perp/p \R$. By \reflem{projection_along_light}, this map preserves inner products, hence is an isometry. As $\Pi$ is parallel to $p^\perp$, we have an isometry $\Pi/p \R \cong p^\perp/p \R = T_p L^+ / p \R$. This spacelike 3-dimensional vector space is isometric to $\E^3$ with negative definite inner product.
\end{proof}

\begin{example}
\label{Eg:phi_2_p0}
We compute $\phi_2 (p_0)$ where $p_0 = (1,0,0,0,1)$ as in \refeg{Phi1_of_k0}. This horosphere $\h_0$ is the intersection of $\hyp^4$ with the plane $\Pi_0$ defined by $\langle x, p_0 \rangle = 1$, i.e. $T-Z = 1$. It contains the point $q_0 = (1,0,0,0,0)$, at which the tangent space $q_0^\perp \cap p_0^\perp$ has basis $\partial_W, \partial_X, \partial_Y$, which also represents a basis of $p_0^\perp/p_0 \R$.
\end{example}

Now we consider hyperbolic isometries which preserves a horosphere $\h$; or equivalently, preserve a point $p \in L^+$ such that $\phi_2 (p) = \h$.  Precisely, let
\[
\B_\h = \{ A \in SL_2\$ \mid A.\h = \h\}, 
\quad
\P_\h = \left\{ A \in \P \cup \{1\} \mid A.\h = \h,  \right\}.
\]
Recall $\P$ is the set of parabolic translation matrices of \refdef{parabolic}, and $\P \cup \{1\}$ adjoins the identity matrix. Let $\underline{\B_\h}$, $\underline{\P_\h}$ be their images in $PSL_2\$$. As stabilisers of a group action, $\B_\h$ and $\underline{\B_\h}$ are subgroups of $SL_2\$$ and $PSL_2 \$ $respectively.

If the action of $A \in SL_2\$$ takes $\h$ to another horosphere $\h'$, then we have $\B_{h'} = A \B_\h A^{-1}$ and, as $\P$ is invariant under conjugation (\reflem{parabolic_facts}(iv)), $\P_{\h'} = A \P_\h A^{-1}$. As $SL_2\$$ acts transitively on $L^+$ (\refsec{SL2_on_paravectors_etc}), then $SL_2\$$ acts transitively on $\Hor$. So for any two horospheres $\h,\h'$, the the groups $\B_\h, \B_{\h'}$ are isomorphic, indeed conjugate subgroups, as are $\P_\h, \P_{\h'}$. 

We can also describe $\B_\h$ and $\P_\h$ in terms of spinors.
\begin{lem}
\label{Lem:parabolic_preserve_spinor}
Suppose $\kappa \in S\HH$, $p = \phi_1 (\kappa) \in L^+$ and $\h = \phi_2 (p)$. 
Then $A \in \P_\h$ iff $A.\kappa = \kappa$, and $A \in \B_\h$ iff $A.\kappa = \kappa x$ for some $x \in \HH$ such that $|x| = 1$.
\end{lem}
In other words, $\B_\h$ consists of those matrices in $SL_2\$$ which have $\kappa$ as an eigenvector with right eigenvalue a unit quaternion, and $\P_\h$ the subset with eigenvalue $1$. In particular, as the stabiliser of $\kappa$, this implies that $\P_\h$, and hence also $\overline{\P_\h}$,  is a group.

\begin{proof}
We have $A \in \B_\h$ iff $A.p = p$, which by equivariance of $\phi_1$ is equivalent to $\phi_1 (A. \kappa) = \phi_1 (\kappa)$. By \reflem{phi1_fibres}, this is equivalent to $A.\kappa = \kappa x$ where $|x| = 1$.

If $A \in \P_\h$ then we again have $A.\kappa = \kappa x$, but then by \reflem{parabolic_on_spinors}(i) we have $x=1$. Conversely, if $A.\kappa = \kappa$ then by equivariance of $\phi_1$ we have $A.p = p$, so $A.\h = \h$. Moreover, by \reflem{parabolic_on_spinors}(ii), $A = 1$ or $A \in \P$. As $A \in \P \cup \{1\}$ and $A.\h = \h$ then $A \in \P_\h$.
\end{proof}

\begin{lem}
\label{Lem:action_of_parabolics}
For any horosphere $\h$,
$\B_\h \cong \Isom^S \E^3 \cong \R^3 \rtimes \Spin(3)$, and $\underline{\B_\h} \cong \Isom^+ \E^3 \cong \R^3 \rtimes SO(3)$. Moreover, $\P_\h$ and $\underline{\P_\h}$ are their respective $\R^3$ translation subgroups.
\end{lem}

Here in standard fashion we regard $\Isom^+ \E^3 = \R^3 \rtimes SO(3)$, so that $\Isom^S \E^3 = \R^3 \rtimes \Spin(3)$.

\begin{proof}
Since all $\B_\h \cong \B_{\h'}$, $\underline{\B_\h} \cong \underline{\B_{\h'}}$, $\P_\h \cong \P_{\h'}$, $\underline{\P_\h} \cong \underline{\P_{\h'}}$ are isomorphic via conjugation by an isometry from $\h$ to $\h'$, it suffices to consider the single horosphere $\h_0 =  \phi_2 (p_0) = \phi_2 \circ \phi_1 (\kappa_0)$ of \refeg{Dphi1_at_10} and \refeg{phi_2_p0}. We now consider the spinors $\kappa$ such that $\phi_1 (\kappa) = p_0$. From \reflem{phi1_fibres} and \refeg{Dphi1_at_10}, $\phi_1 (\kappa) = p$ iff $\kappa = \kappa_0 x$ where $|x| = 1$. By equivariance of $\phi_1$ and $\phi_2$, $\B_\h$ is the set of $A \in SL_2\$$ that preserve the set of spinors $\{(x,0) \mid |x| = 1 \}$. 
Thus
\[
\B_{\h_0} = \left\{ \begin{pmatrix} a & b \\ 0 & d \end{pmatrix} \in SL_2\$ \; \mid \; |a| = 1 \right\}.
\]
The requirement $A \in SL_2\$$ implies that $\pdet A = 1$ and that $ab^*, b^* d \in \$\R^3$. From $\pdet A = 1$ we have $a^* d = 1$, so $|d| = 1$ and $d = a^{-1*} = a'$. So $A \in \B_{\h_0}$ yields the M\"{o}bius transformation
\[
v \mapsto (av+b)d^{-1} = ava^* + ba^* = \sigma(a)(v) + ba^*, \quad
v \in \$\R^3.
\]
Since $|a| = 1$ and $ba^* \in \$\R^3$, these are precisely the orientation-preserving Euclidean isometries of $\$\R^3$, with $\sigma(a)$ acting as a rotation via $\sigma \colon S^3 \To SO(3)$. Precisely, identifying $S^3 \subset \HH^2$ with $\Spin(3)$ and $\$\R^3$ with $\R^3$, we have
\[
\B_{h_0} \stackrel{\cong}{\To} \Isom^S \E^3 = \R^3 \rtimes \Spin(3), \quad
\begin{pmatrix} a & b \\ 0 & a' \end{pmatrix}
\mapsto \left( ba^*, a \right),
\]
which double covers an isomorphism $\underline{\B_{h_0}} \To \Isom^+ \E^3 = \R^3 \rtimes SO(3)$ sending the image of the matrix above to $(ba^*, \sigma(a))$.
Considering \reflem{parabolic_conditions}, the subgroup $\P_{\h_0}$ of $\B_{\h_0}$ thus consists of those matrices of the form of \reflem{parabolic_conditions}(iii) with $c=0$, together with the identity matrix, i.e.
\begin{equation}
\label{Eqn:Ph0}
\P_{\h_0} = 
\left\{ \begin{pmatrix} 1 & aa^* \\ 0 & 1 \end{pmatrix} 
\; \mid \; a \in \HH \right\}
=
\left\{ \begin{pmatrix} 1 & v \\ 0 & 1 \end{pmatrix} \; \mid \; v \in \$\R^3 \right\},
\end{equation}
the latter equality by \reflem{paravector_square_root}. Thus $\P_{\h_0} \cong \$\R^3 \cong \R^3$, precisely as the translation subgroup of $\Isom^S \R^3$. Similarly $\underline{\P_{\h_0}}$ is the $\R^3$ translation subgroup of $\underline{\B_{\h_0}}$.
\end{proof}

\begin{lem}
\label{Lem:parabolic_identity_on_quotient}
Let $\h = \phi_2 (p)$. 
If $A \in \P_\h$ then $A$ acts as the identity on $p^\perp / p \R$.
\end{lem}
To see that this makes sense, note that $A \in \P_\h$ preserves $p$, hence $p\R$, and also preserves $p^\perp$, hence has a well-defined action on $p^\perp/p\R$.

\begin{proof}
We first argue that, by transitivity, it suffices to prove the statement for one horosphere. Given two horospheres $\h, \h'$, there is a $B \in SL_2\$$ such that $B.\h = \h'$, and $\P_{\h'} = B \P_\h B^{-1}$. Letting $\h = \phi_2 (p)$  and $\h' = \phi_2 (q)$ then $B.p = q$. So $B$ yields isomorphisms $p \R \To q \R$ and $p^\perp \To q^\perp$, hence an isomorphism $(p^\perp/p\R) \To q^\perp/q\R$. If $A \in \P_\h$ acts as the identity on $p^\perp / p\R$, then $BAB^{-1}$ acts as the identity on $q^\perp / q \R$, and conversely. 

Thus we consider the horosphere $\h_0 = \phi_2 (p_0)$ of \refeg{phi_2_p0}, and the group $\P_{\h_0}$ of \refeqn{Ph0}. The 4-plane $p_0^\perp$ has basis $p, \partial_W, \partial_X, \partial_Y$, so $p^\perp / p \R$ has basis represented by $\partial_W, \partial_X, \partial_Y$. To find the action of $A \in \P_{\h_0}$ on $p_0^\perp / p_0 \R $ it thus suffices to find the action on points of the form $(0,W,X,Y,0)$. Letting $S \in \pH$ correspond to $(0,W,X,Y,0)$ and $w = W+Xi+Yj \in \$\R^3$ we have
\begin{align*}
AS\overline{A}^T = 
\begin{pmatrix} 1 & v \\ 0 & 1 \end{pmatrix}
\begin{pmatrix} 0 & w/2 \\ \overline{w}/2 & 0 \end{pmatrix}
\begin{pmatrix} 1 & 0 \\ \overline{v} & 1 \end{pmatrix} 
&=
\frac{1}{2} \begin{pmatrix} 
(v\overline{w} + w\overline{v}) & w \\ \overline{w} & 0
\end{pmatrix}
= \begin{pmatrix} v \cdot w & w/2 \\ \overline{w}/2 & 0 \end{pmatrix}
\end{align*}
thus $(0,W,X,Y,0) \mapsto 
(0,W,X,Y,0) + (v\cdot w) p_0$. So the action of $A$ on $p_0^\perp / p_0 \R $ is the identity. 
\end{proof}

\begin{lem}
\label{Lem:parabolic_translation_on_quotient}
If $A \in \P_\h$ and $\Pi$ is the 4-plane given by $\langle x, p \rangle = 1$, then $A$ acts on $\Pi/p \R$ as a Euclidean translation. Moreover, $\P_\h$ is isomorphic to the $\R^3$ group of translations of $\Pi/p \R$.
\end{lem}

\begin{proof}
By a similar transitivity argument as in \reflem{parabolic_identity_on_quotient}, it suffices to prove the statement for $\h_0$ of \refeg{phi_2_p0}; let $p_0$, $\Pi_0$, $q_0$ be as there. The plane $\Pi_0$ is a translate of $p_0^\perp$, namely $\Pi_0 = p_0^\perp + q_0$. Taking a general element $A \in \P_{\h_0}$ as in \refeqn{Ph0}, and writing $v \in \$\R^3$ as $v = W+Xi+Yj$ with $W,X,Y \in \R$, we compute
\[
A.q_0 =
\begin{pmatrix} 1 & v \\ 0 & 1 \end{pmatrix}
\begin{pmatrix} 1/2 & 0 \\ 0 & 1/2 \end{pmatrix}
\begin{pmatrix} 1 & 0 \\ \overline{v} & 1 \end{pmatrix} 
=
\frac{1}{2} 
\begin{pmatrix} 1 + |v|^2 & v \\ \overline{v} & 1 \end{pmatrix}.
\]
Thus $q_0 \mapsto (1 + \frac{1}{2}|v|^2, W,X,Y, \frac{1}{2} |v|^2) = q_0 + \frac{1}{2} |v|^2 p_0 + (0,W,X,Y,0)$, which projects to $q_0 + (0,W,X,Y,0)$ in $\Pi_0 / p_0 \R $. Since $\Pi_0 = q_0 + p_0^\perp$ we have $\frac{\Pi}{p_0 \R} = (q_0 + p_0 \R) + \frac{p_0^\perp}{p_0 \R}$. A general point in $\Pi_0 /p_0 \R $ is of the form $q_0 + r + p_0 \R$, where $r \in p_0^\perp$. By \reflem{parabolic_identity_on_quotient}, $A$ acts as the identity on $p_0^\perp/p \R$, so $A.(q_0 + r + \R p_0) = q_0 + (0,W,X,Y,0) + r + p_0 \R$. In other words, $A$ acts by translation by $(0,W,X,Y,0)$ on $\Pi_0/p_0 \R$. Thus as $v$ varies over all $W+Xi+Yj \in \$\R^3$, we obtain an isomorphism between $\P_{\h_0}$ and the $\R^3$ group of translations of $\Pi_0/p_0 \R$.
\end{proof}

\begin{prop}
\label{Prop:parabolics_are_translations}
If $A \in \P_\h$ then $A$ acts on $\h$ by a Euclidean translation. Moreover, $\P_\h$ is isomorphic to the $\R^3$ group of translations of $\h$.
\end{prop}

\begin{proof}
We have $\h \subset \Pi$, and by \reflem{projection_for_horosphere}, the projection $\h \To \Pi/p \R$ is an isometry, where $\h = \phi_2 (p)$. This projection is along the lines of $p \R$. Any $A \in \P_\h$ fixes $p$, so acts on $p \R$ as the identity. Thus the actions of $A$ on $\h$ and $\Pi/p \R$ are equivariant with respect to this isometry. By \reflem{parabolic_translation_on_quotient}, $\P_\h$ acts on $\Pi/p \R$ by translations, and is isomorphic to the $\R^3$ translation group of $\Pi/p \R$; so under the isometry $\Pi / p \R \To \h$, $\P_\h$ acts on $\h$ by translations, and is isomorphic to the $\R^3$ group of parabolic translations of $\h$.
\end{proof}

\subsection{From multiflags to horospheres and line fields}
\label{Sec:multiflags_to_horospheres}

Suppose we now have a multiflag $(p, V^i, V^j) = [[p,v^i,v^j]]$, so the $i$-flag $[[p, v^i]]$ has 2-plane $V^i = p \R + \R v^i$ an the $j$-flag has 2-plane $V^j = p \R + \R v^j$. The horosphere $\phi_2 (p)$, being the intersection of $\hyp^4$ with the 4-plane $\langle x, p \rangle = 1$, has tangent space at a point $q \in \phi_2 (p)$ given by $T_q \phi_2 (p) = q^\perp \cap p^\perp$. Thus the $i$-flag and $j$-flag 2-planes intersect the tangent the tangent space to the horosphere in
\[
T_q \phi_2 (p) \cap V^i = V^i \cap q^\perp \cap p^\perp = V^i \cap q^\perp, \quad
T_q \phi_2 (p) \cap V^j = V^j \cap q^\perp \cap p^\perp = V^j \cap q^\perp,
\]
since $V^i, V^j$ are subspaces of $p^\perp$.

The intersection of the 2-plane $V^i$ and the 4-plane $q^\perp$ is not 2-dimensional, since $V^i$ contains the lightlike $p$, while $q^\perp$ is spacelike; nor can it be 0-dimensional, as both lie in $\R^{1,4}$; hence the intersection is $1$-dimensional and spacelike. Moreover, since the intersection does not lie in the lightlike direction $p \R$, then the orientation on $V^i/p \R$ provided by the $i$-flag determines an orientation on the intersection. 

Thus the $i$-flag of the multiflag $(p, V^i, V^j)$ naturally provides a 1-dimensional, spacelike, oriented tangent line at each point of the horosphere $\phi_2 (p)$, i.e. and oriented line field, which we call the \emph{$i$-line field}. A similar argument applies to the $j$-flag, which provides a \emph{$j$-line field} $T \phi_2 (p) \cap V^j$ on the horosphere.

In a multiflag, the $i$-flag and $j$-flag are orthogonal, as in \refdef{flag_angle}. At any $q \in \phi_2 (p)$, the $i$-line and $j$-line lie in $V^i, V^j$ respectively, and are both transverse to the lightlike $p \R$, so by \reflem{flag_angle_well_defined}, they are orthogonal. Hence a multiflag naturally produces a horosphere decorated with two orthogonal line fields, and we make the following definitions.

\begin{defn} \
\begin{enumerate}
\item
$\Hor^L$ is the set of triples $(\h, L^i, L^j)$ where $\h \in \Hor$, and $L^i, L^j$ are oriented tangent line fields on $\h$ which are everywhere orthogonal.
\item
The map $\Phi_2 \colon \MF \To \Hor^L$ sends a multiflag $(p, V^i, V^j)$ to the horosphere $\phi_2 (p)$, with the $i$-line field $T \phi_2 (p) \cap V^i$ and $j$-line field $T \phi_2 (p) \cap V^j$. 
\end{enumerate}
\end{defn}
The group $SL_2\$$ acts on $\Hor^L$ via $SO(1,4)^+$: such a map acts as an orientation-preserving isometry of $\hyp^4$, sending horospheres to horospheres, with its derivative sending line fields on horospheres to line fields on horospheres. The actions of $SL_2\$$ on both $\MF$ and $\Hor^L$ are both via the linear maps of $A \in SO(1,4)^+$ on $\R^{1,4}$. We have seen $\phi_2$ is equivariant, sending horospheres $\h \to \h'$ correspondingly to points $p \to p'$ on the light cone; moreover the relatively oriented 2-planes $V^i, V^j$, whose intersections with $\h$ determine its oriented line fields, are sent to relatively oriented 2-planes whose intersections with $\h'$ determine its oriented line fields. So $\Phi_2$ is $SL_2\$$-equivariant.

As in the 3-dimensional case, we now show that the line fields obtained from multiflags are parallel. As in \cite{Mathews_Spinors_horospheres}, an oriented line field on a horosphere $\h$ is \emph{parallel} if it is invariant under Euclidean translations, i.e. parabolic translations fixing $\h$. Thus, as discussed in the introduction in \refdef{decorated_horosphere}, we define a \emph{decorated horosphere} to be an $(\h, L^i, L^j) \in \Hor^L$ where $L^i, L^j$ are both parallel. $L^i$ and $L^j$ are called the \emph{$i$-decoration} and \emph{$j$-decoration} respectively. The set of decorated horospheres is denoted $\Hor^D$.

In other words, a decoration on a horosphere $\h$ consists of a pair of orthogonal parallel oriented line fields. As horospheres have Euclidean geometry, to describe a decoration on a horosphere, it suffices to give two orthogonal directions at one point. Given an orientation on $\h$, the two orthogonal directions uniquely define an oriented orthonormal basis, so a decoration is equivalent to a triple of oriented orthogonal parallel line fields.

\begin{prop}
\label{Prop:line_fields_parallel}
The image of $\Phi_2$ lies on $\Hor^D$.
\end{prop}

\begin{proof}
Let $(p, V^i, V^j) \in \MF$. As $\Phi_1$ is surjective onto multiflags, we have $(p, V^i, V^j) = \Phi_1 (\kappa)$ for some spinor $\kappa$. Let $\Phi_2 (p, V^i, V^j) = (\h, L^i, L^j) \in \Hor^L$. So $\h = \phi_2 (p)$, and we must show that the oriented line fields $L^i = T\h \cap V^i$ and $L^j = T\h \cap V^j$ on $\h$ are parallel.

By \reflem{projection_for_horosphere}, we have an isometry $\h \cong \Pi / p \R \cong p^\perp / p \R$, where $\Pi$ is the affine 3-plane given by $\langle x, p \rangle = 1$. induced by projection. The 2-plane $V^i$ satisfies $p \R \subset V^i \subset T_p L^+ = p^\perp$. At a point $q \in \h$, we have the 3-plane $T_q \h = q^\perp \cap p^\perp \subset p^\perp$, and the 2-plane $V^i \subset p^\perp$, which intersect in the line $L^i$. Under projection along $p \R$, the 3-plane $T_q \h$ maps to the Euclidean 3-plane $p^\perp / p \R$, and the 2-plane $V^i$ maps to the line $V^i / p \R$. Their intersection $L^i$ is transverse to $p \R$, so maps to a line in $p^\perp / p \R$, and as $L^i \subset V^i$, this projection must be the same as $V^i / p \R$.

Thus, under the isometry $T_q \h \To p^\perp / p \R$ induced by projection along the lines parallel to $p$, $L^i$ maps to $V^i / p \R$, and similarly $L^j$ maps to $V^j / p \R$. In particular, in $p^\perp / p \R$ these two line fields appear as constant line fields (there is no dependence on $q$).

Now 
by \refprop{parabolics_are_translations}, the group $\P_\h$ acts on $\h$ as the $\R^3$ group of Euclidean translations of $\h$. Moreover, 
by \reflem{parabolic_translation_on_quotient}, $\P_\h$ acts on $\Pi/p \R$ as the $\R^3$ group of Euclidean translations on $\Pi / p \R$. These two actions are equivariant with respect to the projection isometry $\h \To \Pi/p \R$. Since $\P_\h$ acts by translations in $\Pi / p \R$, in which the line fields $V^i / p \R$ and $V^j / p \R$ are constant, $\P_\h$ preserves these line fields. These line fields are the images of $L^i, L^j$ under the isometry $\h \To \Pi / p \R$, so $\P_\h$ preserves $L^i$ and $L^j$.

As $\Phi_1,\Phi_2$ are $SL_2\$$-equivariant, and as each $A \in \P$ fixes $\kappa$, each $A \in \P$ also fixes the multiflag $(\h, L^i, L^j)$. As $\P_\h$ acts transitively on $\h \cong p^\perp / p \R$, then $L^i$ and $L^j$ are parallel.
\end{proof}

Let us also calculate what $\Phi_2$ does, when considered as a map $\S^{+D} \to \Hor^D$. Recall a multiflag $(p, V^i, V^j) \in \MF$ corresponds to a decorated ideal point $(\ell, \psi) \in \S^{+D}$, where $\ell = p \R$ and $\psi \colon \$\R^3 \To \ell^\perp / \ell$ is conformal and orientation-preserving. Writing $(p, V^i, V^j)$ as $[[p, v^i, v^j]]$ where $p$ has $T$-coordinate $T_0$ and $v^i, v^j$ have Minkowski norm $-4T_0^2$, $\psi$ is given by $\psi(i) = v^i + p \R$ and $\psi(j) = v^j + p \R$. On the corresponding horosphere $\h = \phi_2 (p)$, we have the isometry $\h \To \ell^\perp/\ell$ given by projection. 

Now $\Phi_2$ sends this multiflag or decorated ideal point to $(\h, L^i, L^j) \in \Hor^D$, where $\h = \phi_2 (p)$, and the orthogonal line fields $L^i, L^j$ are defined by intersection with $V^i$ and $V^j$. As we have seen, under the isometry $\h \To \ell^\perp/\ell$, these line fields are directed by $v^i$ and $v^j$. So the line fields $L^i$ and $L^j$ are directed by $\psi(i)$ and $\psi(j)$ respectively. The horosphere $\h$ can be obtained by letting the scale factor of $\psi$ be $-4T_0^2$, then taking $p$ to be the point on $\ell$ with $T$-coordinate $T_0$; we then have $\h = \phi_2 (p)$.

\begin{lem}
\label{Lem:Phi_2_diffeo}
$\Phi_2$ is a diffeomorphism $\MF \To \Hor^D$.
\end{lem}

\begin{proof}
We have seen $\Phi_2$ has image in $\Hor^D$. As $\phi_2$ is a diffeomorphism, every horosphere $\h$ arises in the image of $\Phi_2$. For a given basepoint $p$ and corresponding horosphere $\h$, 
the multiflags $(p, V^i, V^j)$ based at $p$ correspond bijectively with decorations on $\h$, so $\Phi_2$ is a bijection; indeed both are diffeomorphic to $SO(3)$. Both $\MF$ and $\Hor^D$ are naturally diffeomorphic to $S^3 \times \R \times SO(3)$, and $\Phi_2$ provides a diffeomorphism between them.
\end{proof}

\subsection{From hyperboloid to disc and upper half space models}
\label{Sec:H4_models}

We now proceed from the hyperbolic model to the upper half space model $\U$ of $\hyp^4$, denoted as
\begin{equation}
\label{Eqn:U}
\U = \left\{ (w,x,y,z) \in \R^4 \; \mid \; z>0 \right\}
\quad \text{with metric} \quad
ds^2 = \frac{dw^2 + dx^2+dy^2+dz^2}{z^2}.
\end{equation}
The plane $z=0$ is identified with $\$\R^3$ via $(w,x,y) = w+xi+yj$, and $\partial \U$ with $\$\R^3 \cup \infty$. In $\U$, horospheres centred at $\infty$ appear as horizontal 3-planes; we call the $z$-coordinate of this plane the \emph{height} of the horosphere. Horospheres centred at other points appear as 3-spheres tangent to $\$\R^3$; we call the maximum of $z$ on the 3-sphere the \emph{Euclidean diameter} of the horosphere, achieved at its \emph{north pole}. A decoration on a horosphere $\h$ can be described by numbers in $\$\R^3$, using the following identification.
\begin{defn}
\label{Def:U-identification}
Let $\h$ be a horosphere in $\U$. Let $p$ be a point of $\h$ at which $T_p \h$ is the $wxy$-plane. Then the \emph{paravector $\U$-identification} is the $\R$-linear isomorphism $\$\R^3 \To T_p \h$ which sends $1,i,j$ to the vectors $\partial_w, \partial_x, \partial_y \in T_p \h$ respectively.
\end{defn}
The point $p$ in the above definition is the north pole of $\h$, if the centre of $\h$ is not $\infty$; if $\h$ is centred at $\infty$ then $p$ can be any point on $\h$.

If $\h$ has centre $\infty$, it appears as a plane parallel to $\$\R^3$, and we can describe a parallel oriented line field by a nonzero element of $\$\R^3$, well defined up to positive multiples, using the paravector $\U$-identification. A decoration $(L^i, L^j)$ is then specified by two nonzero orthogonal elements of $\$\R^3$, again each well defined up to positive multiples. If $\h$ is centred elsewhere, its tangent plane at its north pole is parallel to $\$\R^3$, and so we can \emph{north pole specify} an oriented line field by a nonzero element of $\$\R^3$ up to positive multiples, again using the paravector $\U$-identification; and we can specify a decoration by two nonzero orthogonal elements of $\$\R^3$, each up to positive multiples. 

We also have the disc model $\Disc$, given by the standard unit ball in $\R^4$, with boundary $\partial \Disc = S^3$. Just as in 3 dimensions, there are standard isometries $\hyp^4 \To \Disc \To \U$, where $\Disc$ is the conformal disc model. The first of these is given by
\[
\hyp^4 \To \Disc, \quad
(T,W,X,Y,Z) \mapsto \frac{1}{1+T}(W,X,Y,Z).
\]
On the boundaries, these are given by
\begin{equation}
\label{Eqn:hyp_models_translations}
\partial\hyp^4 \To \partial\Disc, \quad (T,W,X,Y,Z) \mapsto \frac{1}{T} \left( W,X,Y,Z \right)
\quad \text{and} \quad
\partial \Disc \To \partial \U, \quad (w,x,y,z) \mapsto \frac{w+xi+yj}{1-z}
\end{equation}
where points of $\partial \hyp^4$ are represented by points in $L^+$.
The composition $\partial\hyp^4 \To \partial\U$ sends
\begin{align}
\label{Eqn:boundary_hyperboloid_to_upper}
(T,W,X,Y,Z) \mapsto
\frac{W + Xi + Yj}{T-Z}.
\end{align}
The group $SL_2\$$ acts via isometries on each model, equivariantly with respect to these maps.
A spinor $\kappa = (\xi, \eta) \in S\HH$ maps under $\phi_1$ to
$p = (T,W,X,Y,Z) \in L^+$ given by \refeqn{phi1_in_coords}.
In particular, $W+Xi+Yj = 2 \xi \overline{\eta}$ and $T-Z = 2 |\eta|^2$. Under $\phi_2$, $p$ maps to a horosphere centred at $p \in \partial \hyp^4$, which in $\U = \$\R^3 \cup \{\infty\}$ is given by
\[
\frac{W+Xi+Yj}{T-Z}
= \frac{2 \xi \overline{\eta}}{2 |\eta|^2} 
= \xi \eta^{-1}
\]
A matrix $A \in SL_2\$$ with entries $a,b,c,d$ as in \refdef{SL2H} sends $(\xi, \eta) \mapsto (a \xi + b\eta, c \xi + d \eta)$, so by equivariance, the action of $A$ on $\partial \U$ sends
\[
\xi \eta^{-1} \mapsto (a\xi + b \eta)(c\xi + d\eta)^{-1}
= \left( a \xi \eta^{-1} + b \right) \left( c \xi \eta^{-1} + d \right)^{-1},
\]
which is the standard action by M\"{o}bius transformations, as in \refeqn{Mobius_from_SL2H} or \refeqn{Mobius_from_Clifford_for_quaternions}.

We can now prove \refthm{main_thm_4}, giving the explicit description of $\Phi_2 \circ \Phi_1 (\xi, \eta)$.

\begin{proof}[Proof of \refthm{main_thm_4}]
We have already seen that $\h$ has centre $\xi \eta^{-1}$. We prove the above using transitivity from $\kappa_0 = (1,0)$. Let $\kappa_0, p_0, \h_0, q_0$ be as in 
\refeg{Phi1_of_k0} and 
 \refeg{phi_2_p0}.

So, let $\kappa = \kappa_0 = (1,0)$. From \refeg{Phi1_of_k0} we have $\Phi_1 (\kappa_0) = [[p_0, \partial_X, \partial_Y]]$. 
Under $\Phi_2$, by \refeg{phi_2_p0}, this maps to $\h_0$, which is centred at $p_0$ and passes through $q_0$.
At $q_0$ we have $T_{q_0} \hyp^4$ is the $WXYZ$ 4-plane, and $T_{q_0} \h_0$ is the $WXY$ 3-plane. The $i$-decoration on $\h_0$ at $q_0$ is thus $V^i \cap T_{q_0} \h_0$, where $V^i$ is spanned by $\partial_X$ and $p_0$. Thus the $i$-decoration at $q_0$ is $\partial_X$, and similarly the $j$-decoration at $q_0$ is $\partial_Y$.

In $\Disc$, this corresponds to a decorated horosphere centred at $(0,0,0,1)$, passing through $(0,0,0,0)$, and at that point having $i$-decoration $(0,1,0,0)$ and $j$-decoration $(0,0,1,0)$. In $\U$, this corresponds to a horosphere centred at $\infty$, passing through $(0,0,0,1)$, at that point having $i$-decoration $(0,1,0,0) \sim i \in \$\R^3$ and $j$-decoration $(0,0,1,0) \sim j \in \$\R^3$. Thus $\Phi_2 \circ \Phi_1 (\kappa_0)$ is centred at $\infty$, has height $1$, and $i$- and $j$-decorations specified by $i = \sigma(1)(i)$ and $j = \sigma(1)(j)$ respectively. Thus the proposition holds for $\kappa = (1,0)$.

We now consider the following matrices and actions on spinors
\[
\begin{pmatrix} 0 \\ 1 \end{pmatrix}
= \begin{pmatrix} 0 & -1 \\ 1 & 0 \end{pmatrix}
\begin{pmatrix} 1 \\ 0 \end{pmatrix}, \quad
\begin{pmatrix} \xi \\ 0 \end{pmatrix}
= \begin{pmatrix} \xi & 0 \\ 0 & \xi^{-1*} \end{pmatrix}
\begin{pmatrix} 1 \\ 0 \end{pmatrix}, \quad
\begin{pmatrix} \xi \\ \eta \end{pmatrix}
= \begin{pmatrix} \eta^{-1*} & \xi \\ 0 & \eta \end{pmatrix}
\begin{pmatrix} 0 \\ 1 \end{pmatrix}.
\]
So $\Phi_2 \circ \Phi_1 (0,1)$ is obtained from that of $\kappa_0$ by applying $v \mapsto -v^{-1}$, which in $\U$ is a half turn in the 2-plane bounded by $\R \cup \{\infty\}$; the resulting horosphere has Euclidean diameter $1$, $i$-decoration $i 
= \sigma(1)(i)$, and $j$-direction $j 
= \sigma(1)(j)$. 
Similarly, $\Phi_2 \circ \Phi_1 (\xi, 0)$ is obtained from $\Phi_2 \circ \Phi_1 (1,0)$ by applying $v \mapsto \xi v (\xi^{-1*})^{-1} = \xi v \xi^* = \sigma(\xi)(v)$, which is a rotation about $0$ given by $\sigma(\xi/|\xi|)$, followed by a dilation of $|\xi|^2$, hence has centre $\infty$, height $|\xi|^2$, $i$-direction $\sigma(\xi)(i)$, 
and $j$-direction $\sigma(\xi)(j)$. 
Similarly, when $\eta \neq 0$, $\Phi_2 \circ \Phi_1 (\xi, \eta)$ is obtained from $\Phi_2 \circ \Phi_1 (0, 1)$ by applying $v \mapsto (\eta^{-1*} v + \xi)\eta^{-1} = \sigma(\eta^{-1*})(v) + \xi \eta^{-1}$, hence has centre $\xi \eta^{-1}$, Euclidean diameter $|\eta|^{-2}$, $i$-direction $\sigma(\eta^{-1*})(i)$, 
and $j$-direction $\sigma(\eta^{-1*})(j)$. 
\end{proof}

From this description, it is straightforward to see that when we multiply a $\kappa \in S\HH$ by a real $r>0$, the corresponding decorated horosphere is translated by hyperbolic distance $2 \log r$ towards its centre, and the decoration specification is unchanged in the upper half space model. Similarly, when we multiply $\kappa$ on the right by a unit in $\HH$, the horosphere is unchanged, but the decorations rotate.

\section{Spin decorations and lambda lengths}
\label{Sec:spin_decorations}

We now define spin decorations on horospheres in $\hyp^4$, and define quaternionic lambda lengths, generalising \cite{Mathews_Spinors_horospheres}.

\subsection{Orientations, frames, paravector identifications}
\label{Sec:orientations_frames_spin}

We consider $\hyp^4$ to be oriented in a standard fashion, with vectors in the $w,x,y,z$ directions in $\U$ of \refeqn{U} forming an oriented basis. This agrees with the orientation induced on the hyperboloid model from $\R^{1,4}$ by the transverse vector field $\partial_T$ discussed in \refsec{orientations}.

A horosphere $\h \subset \hyp^4$ has two normal directions: \emph{outward} of $\h$ towards its centre, and \emph{inward} into $\hyp^4$. Denote the outward and inward unit normal vector fields on $\h$ by $N^{out}, N^{in}$. 

We consider orthonormal frames in $\hyp^4$ with this orientation, which we simply refer to as \emph{frames}. The set of frames $\Fr$ forms a principal $SO(4)$-bundle $\Fr \To \hyp^4$, and its spin double (universal) cover is a principal $\Spin(4)$ bundle $\Fr^S \To \hyp^4$. Points of $\Fr$ are frames, and we call points of $\Fr^S$ \emph{spin frames}.

The orientation-preserving isometry group $\Isom^+ \hyp^4 \cong PSL_2\$$ acts simply transitively on $\Fr$, and indeed via a choice of base frame we may identify $\Fr \cong PSL_2\$$. Similarly, the spin isometry group $\Isom^S \hyp^4 \cong SL_2\$$ acts simply transitively on $\Fr^S$ and via a choice of base spin frame we may identify $\Fr^S \cong SL_2\$$. Two matrices $\pm A \in SL_2\$$ represent the same element of $PSL_2\$$, hence correspond to the two spin frames lifting a common frame, related by $2\pi$ rotation. 
We can regard elements of $SL_2\$$ as spin isometries of $\hyp^4$, or equivalently, since they form the universal cover of $\Isom^+ \hyp^4$, as homotopy classes of paths of isometries starting at the identity.

A decoration $(L^i, L^j)$ on a horosphere $\h$ can be normalised to a pair or orthogonal parallel unit tangent vector fields $(v^i, v^j)$ on $\h$, and we can then construct frame fields along $\h$.
\begin{defn}
\label{Def:inward_outward_frame_field}
Let $v = (v^i, v^j)$ be a decoration on $\h$.
\begin{enumerate}
\item
The \emph{inward frame field} of $v$ is the unique frame field $F^{in}$ on $\h$ of the form $(v^{1,in}, v^i, v^j, N^{in})$ for some $v^{1,in}$.
\item
The \emph{outward frame field} of $v$ is the unique frame field $F^{out}$ on $\h$ of the form $(v^{1,out}, v^i, v^j, N^{out})$ for some $v^{1,out}$.
\end{enumerate}
\end{defn}
Note that the first vector $v^{1,in}, v^{1,out}$ in each frame is uniquely determined by the requirement that frames are oriented. Both $v^{1,in}, v^{1,out}$ are tangent to $\h$ and orthogonal to $v^i$ and $v^j$. Indeed $v^{1,in} = - v^{1,out}$. Both $F^{in}$ and $F^{out}$ are sections of $\Fr \To \hyp^4$ over $\h$, differing by a half-turn in the 2-plane orthogonal to $v^i,v^j$. The ordering of the elements in each frame is based on the idea that the frame provides ``$1,i,j,k$ directions" and thus provides an isomorphism of the tangent space to $\hyp^4$ with the quaternions, with the tangent space to $\h$ identified with the paravectors.

The first three elements of each frame $(v^{1,in/out}, v^i, v^j, v^{1,out})$ forms a basis for $T \h$ at each point of $\h$. The notation $v^i, v^j, v^1$ is intended to suggest the real basis $i,j,1$ of paravectors $\$\R^3$. Indeed, it gives us a way to identify the tangent space of a decorated horosphere with $\$\R^3$, as follows.
\begin{defn}
\label{Def:paravector_identification}
Let $p$ be a point on a horosphere $\h$ with a decoration $(v^i, v^j)$. 
\begin{enumerate}
\item
The \emph{inward paravector identification} is the $\R$-linear isomorphism $\$\R^3 \To T_p \h$ which sends $1,i,j$ to the vectors $v^{1,in},v^i,v^j$ of $F^{in}(p)$ respectively.
\item
The \emph{outward paravector identification} is the $\R$-linear isomorphism $\$\R^3 \To T_p \h$ which sends $1,i,j$ to the vectors $v^{1,out},v^i,v^j$ of $F^{out}(p)$ respectively.
\end{enumerate}
\end{defn}

Thus, each vector in $T_p \h$ can be described as a paravector, in two different ways. The two frames $F^{in}, F^{out}$ have the same $i$- and $j$-directions, but reversed direction of $1$. So the two paravectors are related by the following lemma.
 \begin{lem}
\label{Lem:inward_outward_identification}
Let $p, \h, (v^i, v^j)$ be as above, and let $w \in T_p \h$. 
Let $w_{in}, w_{out} \in \$\R^3$ be the paravectors identified with $w$ using the inward and outward paravector identifications respectively. Then $w_{in} = - \overline{w_{out}}$.
\end{lem}
Note the operation $w \mapsto -\overline{w}$ takes a general element $a+bi+cj$ of $\$\R^3$ to $-a+bi+cj$.

\begin{proof}
The inward paravector identification sends $(1,i,j) \mapsto (v^{1,in},v^i,v^j)$, and the outward sends $(1,i,j) \mapsto (v^{1,out}, v^i,v^j)$. Since $v^{1,in} = -v^{1,out}$, we observe that $w_{in}, w_{out}$ have the same imaginary part, but opposite real part, i.e. $w_{in} = - \overline{w_{out}}$.
\end{proof}

\subsection{Spin decorations and multiflags}
\label{Sec:spin_decorations_multiflags}

We now have several ways to describe a decoration on a horosphere: oriented line fields $(L^i, L^j)$ as in the original \refdef{decorated_horosphere}; unit vector fields $(v^i, v^j)$ as in \refsec{orientations_frames_spin}; and we also observe that a decoration is specified by its inwards or outward frame fields $F^{in}, F^{out}$. Henceforth, we will even-handedly denote a decoration by this pair of frame fields $F = (F^{in}, F^{out})$.

Following \cite{Mathews_Spinors_horospheres} we make the following definitions.
\begin{defn}
An \emph{outward (resp. inward) spin decoration} on $\h$ is a continuous lift of an outward (resp. inward) frame field from $\Fr$ to $\Fr^S$.
\end{defn}
Spin decorations come in pairs $W = (W^{in}, W^{out})$, each associated to the other. To describe rotations in a 2-plane we use the convention discussed in \refsec{actions_on_paravectors}, rotating one vector towards another.
\begin{defn} \
\label{Def:associated_spin_decorations}
\begin{enumerate}
\item
Let $W^{out}$ be an outward spin decoration on $\h$. The \emph{associated inward spin decoration} is the spin decoration obtained by rotating $W^{out}$ by $\pi$ in the 2-plane orthogonal to $v^i,v^j$, from $N^{out}$ towards $v^{1,out}$. 
\item
Let $W^{in}$ be an inward spin decoration on $\h$. The \emph{associated outward spin decoration} is the spin decoration obtained by rotating $W^{in}$ by $\pi$ in the 2-plane orthogonal to $v^i,v^j$, from $v^{1,in}$ towards $N^{in}$. 
\item
A \emph{spin decoration} on $\h$ is a pair $W = (W^{in}, W^{out})$ of associated inward and outward spin decoration. The set of spin-decorated horospheres is denoted $\Hor^S$.
\end{enumerate}
\end{defn}
Note the convention here is essentially opposite to that of \cite{Mathews_Spinors_horospheres}, as we order the vectors in frames differently here.

Taking a basepoint in $\Fr$ to be the outward frame of $\Phi_1 \circ \Phi_1 (1,0)$ at $(0,0,0,1) \in \U$, and a lift in $\Fr^S$, the correspondence $\Fr \cong PSL_2\$$ identifies the parabolic translation subgroup $\underline{\P_{\h_0}}$ of \refsec{horospheres_geometry} and \refeqn{Ph0} with the frames of the outward frame field of $\Phi_2 \circ \Phi_1 (1,0)$, and $\P_{\h_0}$ with the spin frames of an outward spin decoration lifting $\Phi_2 \circ \Phi_1 (1,0)$. We thus have identifications
\[
\frac{PSL_2\$}{\underline{\P_{\h_0}}} \cong \Hor^D, \quad
\frac{SL_2\$}{\P_{\h_0}} \cong \Hor^S.
\]
Since $\Hor^D \cong S^3 \times SO(3) \times \R$, we have $\pi_1 (\Hor^D) \cong \Z/2$, and a nontrivial loop is given by rotating a decoration through $2\pi$ in a fixed horosphere. In the double cover, a rotation through $4\pi$ is required to form a loop. Indeed, $\Hor^S$ is the double, universal cover of $\Hor^D$.

We now lift the maps $\Phi_1, \Phi_2$ to maps $\widetilde{\Phi_1}, 
\widetilde{\Phi_s}$ to obtain a diagram
\[
\begin{array}{ccccc}
S\HH & \stackrel{\widetilde{\Phi_1}}{\To} & \MF^S & \stackrel{\widetilde{\Phi_2}}{\To} & \Hor^S \\
& \stackrel{\Phi_1}{\searrow} & \downarrow & & \downarrow \\
& & \MF & \stackrel{\Phi_2}{\To} & \Hor^D
\end{array}
\]
as follows. We have diffeomorphisms $S\HH \cong S^3 \times S^3 \times \R$ and $\MF \cong \Hor^D \cong S^3 \times SO(3) \times \R$. 
Moreover $\Phi_1$ is a double universal cover and $\Phi_2$ a diffeomorphism. 
We have $\pi_1(\MF) \cong \Z/2$, with a nontrivial loop given by fixing a flagpole and rotating the two flags of a multiflag simultaneously about a common orthogonal direction through $2\pi$. We denote the double (universal) cover $\MF^S$, and its elements \emph{spin multiflags}. Choosing basepoints for the lifts arbitrarily, we obtain lifts $\widetilde{\Phi_1}, \widetilde{\Phi_2}$ making the above diagram commute, and remaining $SL_2\$$-equivariant.

\begin{proof}[Proof of \refthm{main_thm_1}]
We have explicitly defined $\widetilde{\Phi_1}, \widetilde{\Phi_2}$ and shown they are $SL_2\$$-equivariant diffeomorphisms of $S^3 \times S^3 \times \R$. They provide the desired bijections.
\end{proof}

\subsection{Quaternionic distances and lambda lengths}
\label{Sec:quaternionic_lambda}

Before defining lambda lengths, we define a notion of quaternionic distance between spin frames, generalising \cite{Mathews_Spinors_horospheres}. 

Let $p$ be a point on an oriented geodesic $\gamma$ in $\hyp^4$. A frame $F = (f_1, f_2, f_3, f_4)$ at $p$ is \emph{adapted} to $\gamma$ if $f_4$ is positively tangent to $\gamma$. A spin frame $\widetilde{F}$ at $p$ is adapted to $\gamma$ if it lifts a frame adapted to $\gamma$.

Now consider two points $p_1, p_2$ on $\gamma$, with frames $F^n = (f_1^n, f_2^n, f_3^n, f_4^n)$ adapted to $\gamma$ at each $p_n$. Parallel translation along $\gamma$ from $p_1$ to $p_2$ takes $F^1$ to a frame ${F'}^1 = ({f'}_1^1, {f'}_2^1, {f'}_3^1, {f'}_4^1)$ at $p_2$ adapted to $\gamma$ and agreeing with $F^2$ in its final vector, ${f'}_4^1 = f_4^2$. This translation has signed distance $\rho$. A rotation $R$ in the 3-plane $\Pi \subset T_{p_2} \hyp^4$ orthogonal to $\gamma$ (or ${f'}_4^1 = f_4^2$) at $p_2$ then moves $F'^1$ to $F^2$. The 3-plane $\Pi$ has an orientation induced by the normal $\gamma$  and the ambient orientation on $\hyp^4$ of \refsec{orientations_frames_spin}, using the conventions of \refsec{orientations}. This rotation $R$ is around some axis in $\Pi$, by some angle $\theta$, measured in the usual right-handed way, using the orientation on $\Pi$. 

Both $({f'}_1^1, {f'}_2^1, {f'}_3^1)$ or $(f_1^2, f_2^2, f_3^2)$ are oriented bases of $\Pi$. We can use such either basis to identify $\Pi$ with $\$\R^3$, identifying the three basis elements with $(i,j,1)$ respectively. In this way, the four elements of the original frame are treated like the ``$1$, $i$, $j$, normal=$k$" elements of a frame field, as in \refdef{inward_outward_frame_field}, and the identifications of $\Pi$ with $\$\R^3$ agree with the paravector identifications of \refdef{paravector_identification}.

Under such an identification, $R$ is given by a rotation of $\theta \in \R / 2 \pi Z$ about an oriented axis in the direction of a unit vector $v \in \$\R^3$. Although we have two possible choices for this identification $\Pi \cong \$\R^3$, in fact we obtain the same description of $R$.
In other words, the vector $v$ has the same expression as a paravector, whether we use the basis from ${F'}^1$ or $F^2$ for the paravector identification. 
This follows from the following elementary linear algebra fact.

\begin{lem}
\label{Lem:rotation_coordinates_invariant}
Suppose $B_1, B_2$ are two oriented bases of an oriented real inner product space $\V$, and $R$ is the linear map taking $B_1$ to $B_2$. Let $R_i$ be the matrix of $R$ with respect to the basis $B_i$. Then $R_1 = R_2$. Moreover, if $\V$ is 3-dimensional and $R$ is a rotation with  axis $\ell$, then any point $p \in \ell$ has the same coordinates with respect to $B_1$ and $B_2$.
\end{lem}

\begin{proof}
Let $C$ be the change of basis matrix from $B_1$ to $B_2$, so the $i$th column of $C$ expresses the $i$th vector of $B_2$ in terms of the basis $B_1$. Then for any endomorphism $L$ of $\V$, the matrices $L_1, L_2$ of $L$ with respect to $B_1, B_2$ are related by $L_2 = C^{-1} L_1 C$. For the endomorphism $R$, we have $R_1 = C$. Thus $R_2 = R_1^{-1} R_1 R_1 = R_1$.

Now if $R$ is a rotation with axis $\ell$ and $p \in \ell$, let $p_i$ be the vector of coordinates of $p$ with respect to $B_i$. Then we have $p_1 = C p_2$. But as $p$ lies on the axis of $R$ we have $R p = p$, so $R_1 p_1 = p_1$ and $R_2 p_2 = p_2$. Thus $p_1 = C p_2 = R_1 p_2 = R_2 p_2 = p_2$.
\end{proof}

Thus the unit $v \in \$\R^3$ describing the oriented axis of $R$ is the same for the two choices of identification $\Pi \cong \$\R^3$. (Note that $v,\theta$ and $-v,-\theta$ describe the same rotation.) The above equally applies to spin frames, the only difference then is $\theta \in \R / 4 \pi \Z$.

In the above we considered translating from $p_1$ to $p_2$, then rotating via $R$ in $\Pi \subset T_{p_2} \hyp^4$. But instead we could have rotated at $p_1$, then translated to $p_2$. We show that even in this case we obtain the same $v \in \$\R^3$ and $\theta$.

Precisely, let ${F'}^2 = ({f'}_1^2, {f'}_2^2, {f'}_3^2, {f'}_4^2)$ be the frame at $p_1$ obtained by parallel translation of $F^2$ from $p_2$ to $p_1$ along $\gamma$. Then ${F'}^2$ agrees with $F^1$ in its final vector, ${f'}_4^2 = f_4^1$. A rotation $R_1$ in the 3-plane $\Pi_1 \subset T_{p_1} \hyp^4$ orthogonal to $\gamma$ (or $f_4^1 = {f'}_4^2$) at $p_1$ takes $F^1$ to ${F'}^2$, and then translation $T_\rho$ along $\gamma$ by signed distance $\rho$ from $p_1$ to $p_2$ takes ${F'}^2$ to $F^2$. The isometries $T_\rho, R_1, R$ satisfy $T_\rho \circ R_1 = R \circ T_\rho$, so $R_1, R$ are conjugate, hence have the same angle $\theta$. Moreover, $T_\rho$ takes the oriented axis of $R_1$ to the oriented axis of $R$. Let $v_1$ be the unit vector directing the axis of $R_1$. Using $F^1$ or ${F'}^2$, we have an identification $\Pi_1 \cong \$\R^3$ and regard $v_1 \in \$\R^3$. By the above lemma, whichever identification we choose, we obtain the same $v_1 \in \$\R^3$.
As $T_\rho$ sends $F^1$ to ${F'}^1$, and ${F'}^2$ to $F^2$, the expressions of $v_1$ and $v$ in $\$\R^3$ are the same in these bases, and so $v_1 = v \in \$\R^3$. 

\begin{defn}
Let $F^1, F^2$ be frames, or spin frames, adapted to a common oriented geodesic. The \emph{quaternionic distance} from $F^1$ to $F^2$ is $\rho + \theta v k$.
\end{defn}
Note that as $v \in \$\R^3$ we have $vk \in \II$. Also, $v,k$ and $-v,-k$ yield the same result for the complex distance. However, the quaternionic distance between frames is only well defined modulo $2\pi vk$; between spin frames, it is well defined modulo $4\pi vk$. Since $e^{\pi v k} = -1$ and $e^{2 \pi vk} = 1$, we can obtain well-defined quaternions by taking exponentials, and this is how we define lambda lengths, as in \refeqn{lambda_length}.

Now consider horospheres $\h_1, \h_2$ with centres $z_1, z_2 \in \$\R^3 \cup \{\infty\}$. Let $\gamma_{12}$ the oriented geodesic from $z_1$ to $z_2$, and $p_i = \gamma_{12} \cap \h_i$. Decorations $F_i$ on $\h_i$ yield frames $F^{in}_i (p_i), F^{out}_i (p_i)$, and spin decorations $W_i$ yield associated spin frames $W^{in}_i (p_i), W^{out}_i (p_i)$. The frames $F^{in}_1 (p_1), F^{out}_2 (p_2)$ and spin frames $W^{in}_1 (p_1), W^{out}_2 (p_2)$ are adapted to $\gamma_{12}$.
\begin{defn}
Let $\h_1, \h_2$ be horospheres.
\begin{enumerate}
\item
If $F_i$ is a decoration on $\h_i$, then the \emph{lambda length} from $(\h_1, F_1)$ to $(\h_2, F_2)$ is
\[
\lambda_{12} = \exp \left( \frac{d}{2} \right)
\]
where $d = \rho + \theta v k$ is the quaternionic distance from $F^{in}_1 (p_1)$ to $F^{out}_2 (p_2)$.
\item
If $W_i$ is a spin decoration on $\h_i$, then the \emph{lambda length} from $(\h_1, W_1)$ to $(\h_2, W_2)$ is
\[
\lambda_{12} = \exp \left( \frac{d}{2} \right)
\]
where $d = \rho + \theta v k$ is the quaternionic distance from $W^{in}_1 (p_1)$ to $W^{out}_2 (p_2)$.
\end{enumerate}
When $\h_1, \h_2$ have common centre, we set $\lambda_{12} = 0$. 
\end{defn}
With decorations, $d$ is well defined modulo $2\pi vk$, so $\lambda_{12}$ is only well defined up to sign. With spin decorations, $d$ is well defined modulo $4 \pi vk$, so $\lambda_{12}$ is a well defined quaternion, and lambda lengths provide a function $\Hor^S \times \Hor^S \To \HH$, which as in the 3-dimensional case is continuous.
By definition, quaternionic distance between spin frames, and lambda length between spin-decorated horospheres, are invariant under the action of spin isometries.

Lambda lengths can be understood to arise naturally: for the representation $\sigma$ of \refdef{rho}, $\sigma(\lambda)$ essentially achieves a lambda length of $\lambda$, as follows. Consider the spin isometry $\phi$ of $\U$ which relates frames related by quaternionic distance $d = \rho + \theta v k$, along the geodesic $\gamma$ from $0$ to $\infty$. The 3-planes orthogonal to $\gamma$ are parallel to the $\$\R^3$ at infinity in $\U$ and we identify them with $\$\R^3$ accordingly. Then $\phi$ rotates the $\$\R^3$ at infinity by angle $\theta$ along an oriented axis directed by the unit paravector $v$, then translates by $\rho$ along the geodesic from $0$ to $\infty$. The isometry $\phi$ is given by the M\"{o}bius transformation of $\partial \U = \$\R^3 \cup \{\infty\}$ which fixes $\infty$, hence is given by a linear map of $\$\R^3$. This linear map is the composition of dilation by $e^\rho$, and rotation by $\theta$ about $v$. By \refprop{rho_rotation_dilation}, this linear map is $\sigma(\lambda)$ where $\lambda = e^{\rho/2} e^{\theta vk} = \exp(d/2)$. In other words, the M\"{o}bius transformation which achieves lambda length $\lambda$ is given by $\sigma(\lambda)$.

\subsection{Antisymmetry of lambda lengths}
\label{Sec:antisym}

In this section we prove the following proposition, generalising the 3-dimensional case where $\lambda$ is antisymmetric, $\lambda_{12} = - \lambda_{21}$.

Let $\h_\bullet, z_\bullet, \gamma_{\bullet \bullet}, p_\bullet$ be as above in \refsec{quaternionic_lambda}.
\begin{prop}
\label{Prop:lambda_length_antisymmetric}
Let $(\h_1, W_1), (\h_2, W_2)$ be spin-decorated horospheres, and let $\lambda_{ij}$ be the lambda length from $(\h_i, W_i)$ to $(\h_j, W_j)$. Then $\lambda_{12} = - \lambda_{21}^*$.
\end{prop}

\begin{proof}
If $h_1, \h_2$ have common centre then the result is clear, so assume not. Let $d_{ij}$ be the quaternionic distance from $W_i^{in}$ to $W_j^{out}$ from $p_i$ to $p_j$ along $\gamma_{ij}$. Let $Y_1^{out}$ be the spin frame at $p_1$ obtained by a rotation of $2\pi$ from $W_1^{out}$, so $Y_1^{out}, W_1^{out}$ are the two spin lifts of $F_1^{out}$.

As in the 3-dimensional case, the spin isometry $\phi$ which takes $W_1^{in} (p_1)$ to $W_2^{out} (p_2)$ also takes $Y_1^{out} (p_1)$ to $W_2^{in} (p_2)$. This spin isometry is the composition of a translation by signed distance $\rho$ along $\gamma_{12}$, followed by a rotation $R$ by angle $\theta \in \R / 4 \pi \Z$ about some unit vector $v \in T_{p_2} \h_2$. Using inward and outward paravector identifications from the decoration of $W_2$ (i.e. the frames $W_2^{in}(p_2)$ and $W_2^{out}(p_2)$), $v$ is identified with unit paravectors $v_{in}, v_{out} \in \$\R^3$. By \reflem{inward_outward_identification}, $v_{in} = - \overline{v_{out}}$. Since the spin isometry is from $W_1^{in}(p_1)$ to $W_2^{out}(p_2)$ and the rotation $R$ is at $p_2$, we use the outward paravector identification of $v$ in the quaternionic distance. Thus the quaternionic distance $d_{12}$ from $W_1^{in}(p_1)$ to $W_2^{out}(p_2)$ is given by $\rho + \theta v_{out} k$. 

The inverse $\phi^{-1}$ of this spin isometry takes $W_2^{in} (p_2)$ to $Y_1^{out} (p_1)$, and is the composition of the rotation $R^{-1}$ about $v$ in $T_{p_2} \h_2$, followed by a translation by signed distance $\rho$ on the oppositely oriented geodesic $\gamma_{21}$. 
As the basis $W_2^{in}(p_2)$ has opposite orientation to the basis $W_2^{out}(p_2)$, the rotation $R^{-1}$ has the same angle $\theta \in 4 \pi \Z$ about $v$ with respect to $W_2^{in}(p_2)$ as $R$ does about $v$ with respect to $W_2^{out}(p_2)$. 
Adding a $2\pi$ to the rotation gives the spin isometry from $W_2^{in} (p_2)$ to $W_1^{out} (p_1)$. 
As the rotation is performed at $p_2$, we use the inward paravector identification of $v$ in the quaternionic distance.
Thus $d_{21} = \rho + (\theta + 2\pi) v_{in} k$.
Note that here we have considered the rotation before the translation, but by the discussion of \refsec{quaternionic_lambda} this still yields the correct quaternionic distance.

Hence the quaternionic distances $d_{12}, d_{21}$ have the same real part $\rho$, and their imaginary parts are $\theta v_{out} k$ and $(\theta + 2\pi) v_{in} k = -(\theta + 2\pi) \overline{v_{out}} k$ respectively. Thus
\[
\lambda_{12} = \exp \left( \frac{\rho + \theta v_{out} k}{2} \right),
\quad
\lambda_{21} = \exp \left( \frac{\rho - (\theta + 2\pi) \overline{v_{out}} k}{2} \right) 
= -\exp \left( \frac{\rho -\theta \overline{v_{out}}k}{2} \right)
\]
Now we note that for $v \in \$\R^3$ we have $-\overline{v} k = (vk)^*$. For $(vk)^* = k^* v^* = -kv$ and from \reflem{quaternion_geometry_facts} we have $-kv = -\overline{v}k$.
Thus
\begin{align*}
\lambda_{21} = -\exp \left( \frac{\rho + \theta (v_{out} k)^*}{2} \right)
= - \exp \left( \left( \frac{\rho + \theta v_{out} k}{2} \right)^* \right)
= - \left( \exp \left( \frac{\rho + \theta v_{out} k}{2} \right) \right)^*
= - \lambda_{12}^*.
\end{align*}
\end{proof}

\subsection{Pseudo-determinant and lambda length}
\label{Sec:lambda_lengths}

Let $\kappa_1, \kappa_2$ be spinors, $\kappa_1 = (\xi_1, \eta_1)$, $\kappa_2 = (\xi_2, \eta_2)$.
Let $(\h_1, W_1), (\h_2, W_2)$ be the corresponding spin-decorated horospheres, i.e. $\Phi_2 \circ \Phi_1 (\kappa_m) = (\h_m, W_m)$ for $m=1,2$.
Let $\lambda_{12}$ be the lambda length from $(\h_1, W_1)$ to $(\h_2, W_2)$.
We now prove \refthm{main_thm_2}, that 
$\lambda_{12} = \{\kappa_1, \kappa_2\}$.

\begin{proof}[Proof of \refthm{main_thm_2}]
First, note $\{\kappa_1, \kappa_2 \} = 0$ precisely when $\kappa_1 = \kappa_2 x$ for some $x \in \HH$ (\reflem{nondegeneracy_of_spinor_form}), which occurs precisely when the corresponding horospheres $\h_1, \h_2$ have the same centre, which occurs precisely when $\lambda_{12} =0$. So we may assume this is not the case.

Second, following \cite{Mathews_Spinors_horospheres}  we prove the result when $\kappa_1 = (1,0)$ and $\kappa_2 = (0,1)$, showing $\lambda_{12} = 1$. Then $\h_1$ is centred at $\infty$ with height $1$, $\h_2$ is centred at $0$ with Euclidean diameter $1$, so $\h_1, \h_2$ are tangent at $p = (0,0,0,1) \in \U$, where $W_1, W_2$ both have $i$-direction specified by $i$, and $j$-direction specified by $j$. Hence $W_1^{in}, W_2^{out}$ have coincident frames. Recall that elements of $SL_2\$ = \Isom^s \hyp^4$ can be regarded as homotopy classes of paths in $PSL_2\$ \cong \Isom^+ \hyp^4$ starting at the identity. 
Consider
\[
A = \begin{pmatrix} 0 & -1 \\ 1 & 0 \end{pmatrix} \in SL_2\$
\quad \text{represented by} \quad
M_t = \pm \begin{pmatrix} \cos t & - \sin t \\ \sin t & \cos t \end{pmatrix} \in PSL_2\$, \quad t \in [0, \pi/2].
\]
Note all $M_t \in PSL_2\R \subset PSL_2\$$.
We observe $A.\kappa_1 = \kappa_2$, so by equivariance of $\Phi_1$ and $\Phi_2$ we have $A.(\h_1, W_1) = (\h_2, W_2)$. Now $M_t$ is the isometry of $\hyp^4$ which rotates the $wz$-plane by $2t$ about $p$, from the $z$-direction $\partial_z$ towards the $w$-direction $\partial_w$, fixing the 2-plane in the $xy$-directions at $p$. The frame $W_1^{in}$ at $p$ projects to $(v_1^{1,in}, v_1^i, v_1^j, N_1^{in}) = (-\partial_w, \partial_x, \partial_y, -\partial_z)$, which rotates via $M_t$ from $N_1^{in}$ towards $v_1^{1,in}$ through $\pi$ to arrive at $W_2^{in}$. This spin frame $W_2^{in}$ projects to $(v_2^{1,in}, v_2^i, v_2^j, N_2^{in}) = (\partial_w, \partial_x, \partial_y, \partial_z)$. Applying \refdef{associated_spin_decorations}, the associated outward spin frame $W_2^{out}$ is then obtained by rotating $W_2^{in}$ by $\pi$ from $v_2^{1,in}$ towards $N_2^{in}$ . These two rotations are inverses, and so $W_1^{in} = W_2^{out}$, the quaternionic distance from $W_1^{in}$ to $W_2^{out}$ is $0$, and $\lambda_{12} = 1$.

Third, we prove the result when $\kappa_1 = (1,0)$ and $\kappa_2 = (0,D)$ where $D \in \HH^\times$, showing that $\lambda_{12} = D$. The oriented common perpendicular $\gamma_{12}$ runs from $\infty$ to $0$, intersecting $\h_1$ at $p_1 = (0,0,0,1)$ and $\h_2$ at $p_2 = (0,0,0,|D|^{-2})$. The translation distance along $\gamma_{12}$ from $p_1$ to $p_2$ is $\rho = 2 \log |D|$. Both $T_{p_1} \h_1$ and $T_{p_2} \h_2$ are the $wxy$ 3-plane, so we can identify them with $\$\R^3$ using the $\U$-identification of \refdef{U-identification}, $(1,i,j) \mapsto (\partial_w, \partial_x, \partial_y)$). The inward paravector identification of $W_1^{in}$ at $p_1$ is $(1,i,j) \mapsto (-\partial_w, \partial_x, \partial_y)$, 
as calculated in the previous case. Thus a tangent vector given by $v \in \$\R^3$ in the $\U$-identification is given by $-\overline{v}$ in the inward paravector identification of $W_1^{in}$.

We now consider the rotation of frames from $W_1^{in}$ to $W_2^{out}$. Let $D= |D| e^{u \theta}$ where $u$ is unit imaginary and $\theta \in [0, \pi]$. Consider
\[
A = \begin{pmatrix} D^{-1*} & 0 \\ 0 & D \end{pmatrix} \in SL_2\$, 
\quad \text{represented by} \quad
M_t = 
\pm \begin{pmatrix} |D|^{-t} e^{-tu^* \theta} & 0 \\ 0 & |D|^t e^{tu\theta} \end{pmatrix}
\]
over $t \in [0,1]$. All $M_t$ are of the form \reflem{elementary_vahlen_properties}(ii) up to sign, hence lie in $PSL_2\$$.
Now $A.(0,1) = (0,D)$, so by equivariance, $A$ sends $(\h_2, W_2)$ from the previous case to $(\h_2, W_2)$ here, and is realised by the family of isometries given by $M_t$. Each $M_t$ yields the M\"{o}bius transformation 
\begin{align*}
z \mapsto 
\left( |D|^{-t} e^{-tu^* \theta} \right) z \left( |D|^{t} e^{tu} \right)^{-1}
= \sigma \left( |D|^{-t} e^{-tu^* \theta} \right) (z)
\end{align*}
which is the isometry of $\U$ rotating $\$\R^3$ by $\sigma(e^{-tu^* \theta})$ and translating by $2t \log |D|$ along $\gamma_{12}$. By \refprop{rho_rotation}, this rotation is of angle $2t\theta$ about $-(-u^*) k = u^* k \in \$\R^3$. But this is identifying $\$\R^3$ with $T_p \h_1$ using the $\U$-identification; the inward paravector identification has the opposite orientation, with rotation angle $-2t\theta$ and axis $-(\overline{u^* k}) 
= k u'$. By \reflem{quaternion_geometry_facts} we have $ku' = -(u')^* k = - \overline{u} \; k = uk$, the final equality since $u \in \II$.
Calculating the quaternionic distance $d_{12}$ from $W_1^{in}$ to $W_2^{out}$ via rotation at $p_1$ then translation, we obtain $d_{12} = \rho + (-2\theta)(uk)k = 2 \log |D| + 2 \theta u$. Hence $\lambda_{12} = \exp(|D| + \theta u) = |D|e^{\theta u} = D$.

Finally, for general $\kappa_1, \kappa_2 \in S\HH$ there exists $A \in SL_2\$$ such that $A.\kappa_1 = (1,0)$ and $A.\kappa_2 = (0,D)$, where $0 \neq D = \{\kappa_1, \kappa_2\} = \xi_1^* \eta_2 - \eta_1^* \xi_2$.
To see this, consider the matrix $B$ with columns $\kappa_1$ and $\kappa_2 D^{-1}$. As the columns of $B$ are spinors (using \reflem{spinor_right_multiplication}) and
\begin{align*}
\pdet B
&= \pdet \begin{pmatrix}
\xi_1 & \xi_2 D^{-1} \\
\eta_1 & \eta_2 D^{-1}
\end{pmatrix}
= \xi_1^* \eta_2 D^{-1} - \eta_1^* \xi_2 D^{-1} = DD^{-1} = 1,
\end{align*}
$B \in SL_2\$$. This $B$ satisfies $B.(1,0) = \kappa_1$ and $B.(0,D) = \kappa_2$.
Thus $A = B^{-1}$ sends 
$\kappa_1 \mapsto (1,0)$ and $\kappa_2 \mapsto (0,D)$ as required.
This $A$ acts on $\hyp^4$ as a spin isometry, so the lambda length from the spin-decorated horosphere of $\kappa_1$ to that of $\kappa_2$, is equal to the lambda length from the spin-decorated horosphere of $(1,0)$ to $(0,D)$, which from the previous case is $D$ as desired.
\end{proof}

\section{Ptolemy equation}
\label{Sec:Ptolemy}

We now apply Gel'fand--Retakh's theory of noncommutative determinants to prove \refthm{main_thm_3}.

\subsection{Quasideterminants and quasi-Pl\"{u}cker coordinates}
\label{Sec:quasidet_Plucker}

In \cite[sec. II]{Gelfand_Retakh_97}, Gel'fand and Retakh consider various generalised notions of determinants Pl\"{u}cker coordinates for matrices over noncommutative rings. We consider only those notions necessary to define their \emph{left quasi-Pl\"{u}cker coordinates} in the case of $2 \times 4$ quaternionic matrices.

For a square matrix $A$ of arbitrary size $N \times N$ over an arbitrary ring with unit, Gel'fand--Retakh \cite[Defn. 1.1.4, p. 520]{Gelfand_Retakh_97} define a family of \emph{quasideterminants} determinants $|A|_{p,q}$, indexed by integers $p,q$ where $1 \leq p,q \leq N$ . In the case of a $2 \times 2$ matrix, they are given by \cite[Example 1, p. 520]{Gelfand_Retakh_97}
\[
\begin{array}{ll}
|A|_{11} = a_{11} - a_{12} a_{22}^{-1} a_{21} &
|A|_{12} = a_{12} - a_{11} a_{21}^{-1} a_{22} \\
|A|_{21} = a_{21} - a_{22} a_{12}^{-1} a_{11} &
|A|_{22} = a_{22} - a_{21} a_{11}^{-1} a_{12}
\end{array}
\]

For a general $M \times N$ matrix $A$ with $M < N$ over an arbitrary skew field, they 
define a family of \emph{left quasi-Pl\"{u}cker coordinates} $p_{lm}^I (A)$, indexed by integers $l,m$ such that $1 \leq l,m \leq N$, and sets $I = \{ l_1, \ldots, l_{M-1} \}$ of $M-1$ distinct integers  satisfying $1 \leq l_1, \ldots, l_{M-1} \leq N$, none of which is equal to $l$ \cite[sec. II.1.1, p. 526]{Gelfand_Retakh_97}. In the case of $2 \times 4$ matrices over $\HH$, we thus have left quasi-Pl\"{u}cker coordinates $p_{lm}^{I}$ where $1 \leq l,m \leq 4$ and $I = \{l_1\}$ where $1 \leq l_1 \leq 4$ and $l_1 \neq l$. We write $n$ for $l_1$ and $p_{lm}^n (A)$ rather than $p_{ij}^{I} (A)$. Writing $a_{i,j}$ for the $(i,j)$ entry of $A$, in this case $p_{lm}^n (A)$ is defined as
\begin{equation}
\label{Eqn:left_quasi-plucker_2x4}
p_{lm}^n (A) =
\begin{vmatrix}
a_{1,l} & a_{1,n} \\
a_{2,l} & a_{2,n}
\end{vmatrix}^{-1}_{s,l}
\begin{vmatrix}
a_{1,m} & a_{1,n} \\
a_{2,m} & a_{2,n} 
\end{vmatrix}_{s,m},
\end{equation}
where $s = 1$ or $2$; it turns out the result is the same for either choice of $s$.

Let us calculate these quantities for the matrices we shall need.
\begin{lem}
\label{Lem:quasideterminant_for_spinors}
Let $A$ be a quaternionic matrix whose columns lie in $S\HH$,
\[
A = \begin{pmatrix} a & b \\ c & d \end{pmatrix}.
\]
Then
\[
\begin{pmatrix} |A|_{11} & |A|_{12} \\ |A|_{21} & |A|_{22} \end{pmatrix}
=
\begin{pmatrix}
d^{-1*} & -c^{-1*} \\ -b^{-1*} & a^{-1*} 
\end{pmatrix}
\begin{pmatrix}
(\pdet A)^* & 0 \\ 0 & (\pdet A)
\end{pmatrix}.
\]
\end{lem}

\begin{proof}
As the columns of $A$ lie $S\HH$, we have $d^* b, c^* a \in \$\R^3$, hence $d^* b = b^* d$ and $c^* a = a^* c$. Then we compute
\[
d^* |A|_{11} = d^* \left( a - bd^{-1} c \right)
= d^* a - b^* d d^{-1} c
= (\pdet A)^*,
\]
so $|A|_{11} = d^{*-1} \det A$. Similarly we compute
\begin{align*}
-c^* |A|_{12} &= -c^* \left( b - ac^{-1} d \right)
= -c^* b + a^* d
= \pdet A \\
-b^* |A|_{21} &= -b^* \left( c - db^{-1} a \right)
= -b^* c + d^* b b^{-1} a
= (\pdet A)^* \\
a^* |A|_{22} &= a^* \left( d - ca^{-1} b \right)
= a^* d - c^* a a^{-1} b
= \pdet A
\end{align*}
giving the desired result.
\end{proof}

As in \refsec{bracket}, for spinors $\kappa_1, \ldots, \kappa_n$, we denote by $(\kappa_1, \ldots, \kappa_n)$ the $2 \times n$ matrix whose $m$th column is $\kappa_m$. Let $\kappa_m = (\xi_m, \eta_m)$. Recalling that $\pdet (\kappa_1, \kappa_2) = \{ \kappa_1, \kappa_2 \}$ by definition, and letting $A = (\kappa_1, \kappa_2)$, then \reflem{quasideterminant_for_spinors} can be rewritten as
\begin{equation}
\label{Eqn:quasidets_for_spinors}
\begin{array}{ll}
|A|_{11} = \eta_2^{-1*} \{ \kappa_1, \kappa_2 \}^* &
|A|_{12} = -\eta_1^{-1*} \{ \kappa_1, \kappa_2 \} \\
|A|_{21} = -\xi_2^{-1*} \{ \kappa_1, \kappa_2 \}^* &
|A|_{22} = \xi_1^{-1*} \{ \kappa_1, \kappa_2 \}.
\end{array}
\end{equation}
Moreover, \refeqn{left_quasi-plucker_2x4} can be rewritten as
\begin{equation}
\label{Eqn:Plucker_in_spinors}
p_{lm}^n (\kappa_1, \kappa_2, \kappa_3, \kappa_4) = \left| (\kappa_l, \kappa_n) \right|_{sl}^{-1} \; \left| (\kappa_m, \kappa_n) \right|_{sm}
\end{equation}

\begin{lem}
\label{Lem:quasi-Plucker_lambda}
Let $A = (\kappa_1, \kappa_2, \kappa_3, \kappa_4)$ where each $\kappa_m \in S\HH$. Then for $1 \leq i,j,k \leq 4$ and $k \neq i$
\[
p_{lm}^n (A) = \left( \pdet (\kappa_n, \kappa_l) \right)^{-1} \, \pdet (\kappa_n, \kappa_m)
= \left\{ \kappa_n, \kappa_l \right\}^{-1} \, \left\{ \kappa_n, \kappa_m \right\}.
\]
\end{lem}

\begin{proof}
Taking $s=1$ in \refeqn{Plucker_in_spinors} and applying \refeqn{quasidets_for_spinors}, we obtain
\[
p_{lm}^{n} (A) =
\begin{vmatrix}
\xi_l & \xi_{n} \\ \eta_l & \eta_{n}
\end{vmatrix}_{1l}^{-1}
\begin{vmatrix}
\xi_m & \xi_{n} \\ \eta_m & \eta_{n}
\end{vmatrix}_{1m}
=
\left( \eta_{n}^{-1*} \{ \kappa_l, \kappa_{n} \}^* \right)^{-1}
\eta_{n}^{-1*} \{\kappa_m, \kappa_{n} \}^*
=
\{\kappa_l, \kappa_{n}\}^{*-1} \{ \kappa_m, \kappa_{n} \}^*.
\]
Now the result follows from the antisymmetry property $\{ \kappa_1, \kappa_2 \} = - \{ \kappa_2, \kappa_1\}^*$.
\end{proof}

So we obtain \refeqn{quasi-Plucker_lambda}, i.e. 
$p_{lm}^n = \lambda_{n,l}^{-1} \lambda_{n,m}$, where $\lambda_{m,n}$ is the lambda length from the spin-decorated horosphere of $\kappa_m$ to that of $\kappa_n$.

\subsection{Pl\"{u}cker relations}
\label{Sec:Plucker}

We can now prove \refprop{triangle_holonomy}, that 
$\lambda_{12} \lambda_{32}^{-1} \lambda_{31}
\in \$\R^3 \cup \{\infty\}$, or equivalently, since $\$\R^3$ is closed under taking inverses, 
$\lambda_{31}^{-1} \lambda_{32} \lambda_{12}^{-1}
\in \$\R^3 \cup \{\infty\}$.

\begin{proof}[Proof of \refprop{triangle_holonomy}]
Gelfand--Retakh \cite[Prop. 2.1.4 and Example 1, p. 527]{Gelfand_Retakh_97} show that for $2 \times n$ matrices, for any $l,m,n$,
\[
p_{lm}^{n} p_{mn}^{l} p_{nl}^{m} = -1.
\]
In terms of lambda lengths we then have 
\[
\lambda_{nl}^{-1} \lambda_{nm} \lambda_{lm}^{-1} \lambda_{ln} \lambda_{mn}^{-1} \lambda_{ml} = -1
\]
which is equivalent to
\[
\lambda_{nl}^{-1} \lambda_{nm} \lambda_{lm}^{-1} 
= -\lambda_{ml}^{-1} \lambda_{mn} \lambda_{ln}^{-1}.
\]
Since $\lambda_{mn} = - \lambda_{nm}^*$, the right hand side is $\lambda_{lm}^{-1*} \lambda_{nm}^* \lambda_{nl}^{-1*} = \left( \lambda_{nl}^{-1} \lambda_{nm} \lambda_{lm}^{-1} \right)^*$. Thus we obtain
\[
\lambda_{nl}^{-1} \lambda_{nm} \lambda_{lm}^{-1}
= \left( \lambda_{nl}^{-1} \lambda_{nm} \lambda_{lm}^{-1} \right)^*
\]
and, being fixed under $*$, taking $(l,m,n) = (1,2,3)$, we conclude $\lambda_{31}^{-1} \lambda_{32} \lambda_{12}^{-1} \in \$\R^3$, unless it is $\infty$.
\end{proof}

Finally, we prove \refthm{main_thm_3}, the non-commutative Ptolemy equation
\[
\lambda_{02}^{-1} \lambda_{01} \lambda_{31}^{-1} \lambda_{32} + \lambda_{02}^{-1} \lambda_{03} \lambda_{13}^{-1} \lambda_{12} = 1.
\]

\begin{proof}[Proof of \refthm{main_thm_3}]
Gelfand--Retakh \cite[Prop. 2.1.4 and Example 2, p. 527]{Gelfand_Retakh_97} show that for $2 \times 4$ matrices, for any $a,b,l,m$,
\[
p_{ab}^{l} p_{ba}^{m} + p_{am}^{l} p_{ma}^{b} = 1.
\]
Thus we obtain
\[
\lambda_{la}^{-1} \lambda_{lb} \lambda_{mb}^{-1} \lambda_{ma} + \lambda_{la}^{-1} \lambda_{lm} \lambda_{bm}^{-1} \lambda_{ba} = 1.
\]
Taking $(a,b,l,m) = (2,1,0,3)$ then gives the desired result.
\end{proof}

\small

\bibliography{spinref}

\providecommand{\bysame}{\leavevmode\hbox to3em{\hrulefill}\thinspace}
\providecommand{\MR}{\relax\ifhmode\unskip\space\fi MR }
\providecommand{\MRhref}[2]{%
  \href{http://www.ams.org/mathscinet-getitem?mr=#1}{#2}
}
\providecommand{\href}[2]{#2}
\begin{thebibliography}{10}

\bibitem{Ahlfors_84}
Lars~V. Ahlfors, \emph{Old and new in {M}\"obius groups}, Ann. Acad. Sci. Fenn.
  Ser. A I Math. \textbf{9} (1984), 93--105. \MR{752394}

\bibitem{Ahlfors_Mobius85}
\bysame, \emph{M\"{o}bius transformations and {C}lifford numbers}, Differential
  geometry and complex analysis, Springer, Berlin, 1985, pp.~65--73.
  \MR{780036}

\bibitem{Ahlfors_fixedpoints_85}
\bysame, \emph{On the fixed points of {M}\"obius transformations in {${\bf
  R}^n$}}, Ann. Acad. Sci. Fenn. Ser. A I Math. \textbf{10} (1985), 15--27.
  \MR{802464}

\bibitem{Ahlfors_Clifford85}
\bysame, \emph{Clifford numbers and {M}\"obius transformations in {${\bf
  R}^n$}}, Clifford algebras and their applications in mathematical physics
  ({C}anterbury, 1985), NATO Adv. Sci. Inst. Ser. C: Math. Phys. Sci., vol.
  183, Reidel, Dordrecht, 1986, pp.~167--175. \MR{863437}

\bibitem{Ahlfors_Mobius_86}
\bysame, \emph{M\"obius transformations in {${\bf R}^n$} expressed through
  {$2\times 2$} matrices of {C}lifford numbers}, Complex Variables Theory Appl.
  \textbf{5} (1986), no.~2-4, 215--224. \MR{846490}

\bibitem{Ahlfors_Lounesto_89}
Lars~V. Ahlfors and Pertti Lounesto, \emph{Some remarks on {C}lifford
  algebras}, Complex Variables Theory Appl. \textbf{12} (1989), no.~1-4,
  201--209. \MR{1040920}

\bibitem{Aslaksen_96}
Helmer Aslaksen, \emph{Quaternionic determinants}, Math. Intelligencer
  \textbf{18} (1996), no.~3, 57--65. \MR{1412993}

\bibitem{Berenstein_Retakh_18}
Arkady Berenstein and Vladimir Retakh, \emph{Noncommutative marked surfaces},
  Adv. Math. \textbf{328} (2018), 1010--1087. \MR{3771148}

\bibitem{Cao_Waterman_98}
C.~Cao and P.~L. Waterman, \emph{Conjugacy invariants of {M}\"obius groups},
  Quasiconformal mappings and analysis ({A}nn {A}rbor, {MI}, 1995), Springer,
  New York, 1998, pp.~109--139. \MR{1488448}

\bibitem{Cao_07}
Wensheng Cao, \emph{On the classification of four-dimensional {M}\"obius
  transformations}, Proc. Edinb. Math. Soc. (2) \textbf{50} (2007), no.~1,
  49--62. \MR{2294003}

\bibitem{CPW_04}
Wensheng Cao, John~R. Parker, and Xiantao Wang, \emph{On the classification of
  quaternionic {M}\"obius transformations}, Math. Proc. Cambridge Philos. Soc.
  \textbf{137} (2004), no.~2, 349--361. \MR{2092064}

\bibitem{Chevalley_54}
Claude~C. Chevalley, \emph{The algebraic theory of spinors}, Columbia
  University Press, New York, 1954. \MR{60497}

\bibitem{Foreman_04}
B.~Foreman, \emph{Conjugacy invariants of {${\rm Sl}(2,{\mathbb{H}})$}}, Linear
  Algebra Appl. \textbf{381} (2004), 25--35. \MR{2039798}

\bibitem{Fueter_1926}
R.~Fueter, \emph{Sur les groupes improprement discontinus}, C.R. Acad. Sci.
  Paris \textbf{182} (1926), 432--434.

\bibitem{Gelfand_Retakh_97}
I.~Gelfand and V.~Retakh, \emph{Quasideterminants. {I}}, Selecta Math. (N.S.)
  \textbf{3} (1997), no.~4, 517--546. \MR{1613523}

\bibitem{Gelfand_Retakh_91}
I.~M. Gel'fand and V.~S. Retakh, \emph{Determinants of matrices over
  noncommutative rings}, Funktsional. Anal. i Prilozhen. \textbf{25} (1991),
  no.~2, 13--25, 96. \MR{1142205}

\bibitem{Gelfand_Retakh_92}
\bysame, \emph{Theory of noncommutative determinants, and characteristic
  functions of graphs}, Funktsional. Anal. i Prilozhen. \textbf{26} (1992),
  no.~4, 1--20, 96. \MR{1209940}

\bibitem{GGRW_05}
Israel Gelfand, Sergei Gelfand, Vladimir Retakh, and Robert~Lee Wilson,
  \emph{Quasideterminants}, Adv. Math. \textbf{193} (2005), no.~1, 56--141.
  \MR{2132761}

\bibitem{Gongopadhyay_12}
K.~Gongopadhyay, \emph{Algebraic characterization of isometries of the
  hyperbolic 4-space}, International Scholarly Research Notices (2012), no.~1,
  757489.

\bibitem{Gongopadhyay_Kulkarni_09}
Krishnendu Gongopadhyay and Ravi~S. Kulkarni, \emph{{$z$}-classes of isometries
  of the hyperbolic space}, Conform. Geom. Dyn. \textbf{13} (2009), 91--109.
  \MR{2491719}

\bibitem{HIMS_lambda_figure8}
J.A. Howie, D.~Ibarra, D.~V. Mathews, and L.~Su, \emph{Lambda lengths in the
  figure eight knot complement}, arxiv:2411.06368.

\bibitem{Kellerhals01}
Ruth Kellerhals, \emph{Collars in {${\rm PSL}(2,\mathbf{H})$}}, Ann. Acad. Sci.
  Fenn. Math. \textbf{26} (2001), no.~1, 51--72. \MR{1816562}

\bibitem{Lipschitz_1886}
R.~Lipschitz, \emph{Untersuchungen ueber die summe von quadraten}, Bonn, 1886.

\bibitem{Lounesto_Clifford_book_01}
Pertti Lounesto, \emph{Clifford algebras and spinors}, second ed., London
  Mathematical Society Lecture Note Series, vol. 286, Cambridge University
  Press, Cambridge, 2001. \MR{1834977}

\bibitem{Lounesto_Latvamaa_80}
Pertti Lounesto and Esko Latvamaa, \emph{Conformal transformations and
  {C}lifford algebras}, Proc. Amer. Math. Soc. \textbf{79} (1980), no.~4,
  533--538. \MR{572296}

\bibitem{Maass_49}
Hans Maass, \emph{Automorphe {F}unktionen von meheren {V}er\"anderlichen und
  {D}irichletsche {R}eihen}, Abh. Math. Sem. Univ. Hamburg \textbf{16} (1949),
  72--100. \MR{33317}

\bibitem{Mathews_Purcell_Ptolemy_hyperbolic}
D.~V. Mathews and J.~Purcell, \emph{A symplectic basis for 3-manifold
  triangulations}, arxiv:2208.06969, accepted for publication in
  \emph{Communications in Analysis \& Geometry}.

\bibitem{Mathews_Spinors_horospheres}
Daniel~V. Mathews, \emph{Spinors and horospheres}, Adv. Math. \textbf{468}
  (2025), Paper No. 110200. \MR{4878101}

\bibitem{Mathews_Zymaris}
Daniel~V. Mathews and Orion Zymaris, \emph{Spinors and the {D}escartes circle
  theorem}, J. Geom. Phys. \textbf{212} (2025), Paper No. 105458. \MR{4871605}

\bibitem{Parker_Short_09}
John~R. Parker and Ian Short, \emph{Conjugacy classification of quaternionic
  {M}\"obius transformations}, Comput. Methods Funct. Theory \textbf{9} (2009),
  no.~1, 13--25. \MR{2478260}

\bibitem{Penner87}
R.~C. Penner, \emph{The decorated {T}eichm\"{u}ller space of punctured
  surfaces}, Comm. Math. Phys. \textbf{113} (1987), no.~2, 299--339.
  \MR{919235}

\bibitem{Penrose_Rindler84}
R.~Penrose and W.~Rindler, \emph{Spinors and space-time. {V}ol. 1}, Cambridge,
  1984. \MR{776784}

\bibitem{Porteous_69_81}
Ian~R. Porteous, \emph{Topological geometry}, second ed., Cambridge University
  Press, Cambridge-New York, 1981. \MR{606198}

\bibitem{Retakh_OW_report_13}
V.~Retakh, \emph{Noncommutative {L}aurent phenomenon, triangulations and
  surfaces}, joint with A. Berenstein, Mathematisches Forschungsinstitut
  Oberwolfach, Report No. 58/2013, Cluster Algebras and Related Topics, 8--14
  December 2013,
  \url{https://publications.mfo.de/bitstream/handle/mfo/3389/OWR_2013_58.pdf},
  3417--9.

\bibitem{Vahlen_1902}
K.~Th. Vahlen, \emph{Ueber {B}ewegungen und complexe {Z}ahlen}, Math. Ann.
  \textbf{55} (1902), no.~4, 585--593. \MR{1511164}

\bibitem{Wada_90}
Masaaki Wada, \emph{Conjugacy invariants of {M}\"obius transformations},
  Complex Variables Theory Appl. \textbf{15} (1990), no.~2, 125--133.
  \MR{1058518}

\bibitem{Waterman_93}
P.~L. Waterman, \emph{M\"obius transformations in several dimensions}, Adv.
  Math. \textbf{101} (1993), no.~1, 87--113. \MR{1239454}

\bibitem{Retakh_Wilson_lectures}
R.~L. Wilson and V.~Retakh, \emph{Advanced course on {Q}uasideterminants and
  {U}niversal {L}ocalization},
  \url{https://ncatlab.org/nlab/files/RetakhWilson-Quasideterminants.pdf}.

\bibitem{Zhang_97}
Fuzhen Zhang, \emph{Quaternions and matrices of quaternions}, Linear Algebra
  Appl. \textbf{251} (1997), 21--57. \MR{1421264}

\end{thebibliography}
\bibliographystyle{amsplain}

\end{document}